\documentclass{article}
                            
\usepackage{geometry} 
\usepackage{mathtools} 
\usepackage{amsmath}
\usepackage{setspace}  
\usepackage{amssymb} 
\usepackage{amsthm} 
\usepackage{graphicx}
\usepackage{mathrsfs} 
\usepackage{subcaption}
\usepackage{pinlabel}
\usepackage{rotating}

\newtheorem{thm}{Theorem}[section]
\newtheorem{cor}[thm]{Corollary}
\newtheorem{lem}[thm]{Lemma}

\newtheorem{prop}[thm]{Proposition}

\newcommand{\R}{\mathbb{R}}
\newcommand{\C}{\mathbb{C}}
\newcommand{\N}{\mathbb{N}}

\DeclareMathOperator{\Cyl}{Cyl}

\DeclareMathOperator{\Tw}{Tw}

\DeclareMathOperator{\Lab}{Lab}

\begin{document}

\title{An Infinite Family of Links\\with Critical Bridge Spheres}
\author{Daniel Rodman}
\date{}
\maketitle

\begin{abstract}
A closed, orientable, splitting surface in an oriented 3-manifold is a topologically minimal surface of index \(n\) if its associated disk complex is \((n-2)\)-connected but not \((n-1)\)-connected.
A critical surface is a topologically minimal surface of index 2.
In this paper, we use an equivalent combinatorial definition of critical surfaces to construct the first known critical bridge spheres for nontrivial links.
\end{abstract}

\section{Introduction}\label{sec:intro}
In the 1960s Haken developed a framework for studying manifolds that contain incompressible surfaces.
He demonstrated that for such manifolds, one can reduce the proofs of many theorems to induction arguments by using ``hierarchies," consecutively cutting a manifold into pieces along incompressible surfaces until a collection of balls is obtained.
This powerful approach clearly demonstrated the utility of incompressible surfaces in the study of low-dimensional topology.

In 1987 Casson and Gordon introduced the idea of strongly irreducible surfaces \cite{CG87}.
Unlike incompressible surfaces which have no compressing disks, strongly irreducible surfaces have potentially many compressing disks to both sides, but any two compressing disks on opposite sides necessarily intersect.
Somewhat surprisingly, it turns out that many theorems that are easy to prove for manifolds with incompressible surfaces also hold true for manifolds with strongly irreducible surfaces, although the proofs can be somewhat more involved.
Casson and Gordon's watershed result in the theory of strongly irreducible surfaces is that if a closed, connected, orientable 3-manifold's minimal genus Heegaard splitting is not strongly irreducible, then the manifold must contain an essential surface \cite{CG87}.

In 2002, Bachman introduced the notion of a critical surface \cite{Bachman02}, which he followed in 2009 with the introduction of the more general concept of a topologically minimal surface \cite{Bachman09}.
A closed, orientable, splitting surface in an oriented 3-manifold \(M\) is a topologically minimal surface of index \(n\) if its associated disk complex is \((n-2)\)-connected but not \((n-1)\)-connected.
This latter definition provides the framework into which incompressible, strongly irreducible, and critical surfaces all fit:
They are topologically minimal surfaces of index 0, 1, and 2 respectively.
Just like the incompressible and strongly irreducible surfaces before them, critical surfaces and topologically minimal surfaces in general have been used to prove long-standing conjectures that had remained unresolved for many years.
For example, Bachman used these surfaces to prove The Gordon Conjecture\footnote{Gordon Conjecture: If the Heegaard splitting \(U\cup_H V\) is a connect sum of two Heegaard splittings, and \(H\) is stabilized, then one of its summands is stabilized.} and to provide counterexamples to the Stabilization Conjecture\footnote{Stabilization Conjecture: Given any pair of Heegaard splittings, stabilizing the higher genus splitting once results in a stabilization of the other splitting.} \cite{Bachman_Stab_Conj}.

Most of the work done thus far regarding topologically minimal surfaces deals specifically with surfaces which are Heegaard splittings, two-sided surfaces that split a manifold into two compression bodies.
Less is known about topologically minimal bridge surfaces, which are a natural type of surface to consider when the 3-manifold is a link complement.
Lee has shown in \cite{Lee2} that all bridge spheres for the unknot with any number of bridges in \(S^3\) are topologically minimal, and his results provide upper bounds for the indices of such bridge spheres.  In particular, he concludes that the bridge sphere for an unknot in a 3-bridge position has index exactly 2 (i.e., it is critical).

In this paper we provide the first known examples of critical bridge spheres for nontrivial links.
Our construction is inspired by recent work of Johnson and Moriah \cite{Jesse} in which they construct links with bridge surfaces of arbitrarily high distance.
The central result of this paper is the following:

\vspace{6pt}
\noindent\textbf{Theorem \ref{thm:mainthm}.} \textit{There is an infinite family of nontrivial links with critical bridge spheres.}
\vspace{6pt}

In Section \ref{sec:defs} we go over some of the foundational topological definitions upon which this paper relies, including an equivalent, combinatorial definition of a critical surface.
Then in Section \ref{sec:defs} we describe what it means for a link to be in a plat position.
Following that, in Section \ref{sec:setting} we embed a link \(L\) in \(S^3\) in a plat position with bridge sphere \(F\) and discuss some of the specific details of the embedding as well as build some of the tools (certain arcs, loops, disks, and projection maps) which we will use throughout the rest of the paper.
Sections \ref{sec:defs} and \ref{sec:setting} should be considered setup for the rest of the paper, and this is essentially the same setup as in Johnson and Moriah's paper \cite{Jesse}.
In particular we make use of Johnson and Moriah's plat links and projection maps.
In Section \ref{sec:labyrinth} we develop a link diagramatic way to visualize boundary loops of compressing disks for \(F\).
Theorem \ref{thm:mainthm} is proved in Sections \ref{sec:Cintersectsgamma} and \ref{sec:proof}.

\section{Definitions}\label{sec:defs}
Suppose \(\Sigma\) is a compact, orientable surface embedded in a compact, orientable 3-manifold \(M\), and \(D\) is a disk embedded in \(M\) (not necessarily properly).
\(D\) is a \textbf{compressing disk} for \(\Sigma\) if \(D\cap \Sigma=\partial D\), and \(\partial D\) neither bounds a disk in \(\Sigma\) nor is parallel to a boundary component of \(\Sigma\).

%

A compact, orientable surface \(\Sigma\) embedded in a connected 3-manifold \(M\) is a \textbf{splitting surface} if \(M\backslash\Sigma\) has two components.
We associate a simplicial complex \(\Gamma\) called the \textbf{disk complex} to \((M,\Sigma)\) in the following way:
Vertices of \(\Gamma\) are isotopy classes of compressing disks for \(\Sigma\).
A set of \(m+1\) vertices will be filled in with an \(m\)-simplex if and only if the corresponding isotopy classes of compressing disks have pairwise disjoint representatives.
\(\Sigma\) is called a \textbf{topologically minimal surface of index \(n\)} if \(\Gamma\) is empty, or if it is \((n-2)\)-connected but not \((n-1)\)-connected.
\(\Sigma\) is called a \textbf{critical surface} if it is a topologically minimal surface of index 2.
In \cite{Bachman09}, Bachman gives an alternative, combinatorial definition for critical surfaces and proves the two definitions are equivalent.
This second definition, given below, is the one we will utilize in this paper. 

A \textbf{critical surface} is a splitting surface \(\Sigma\subset M\) whose isotopy classes of compressing disks can be partitioned into \(\mathcal{C}_1\sqcup \mathcal{C}_2\) in a way that satisfies the following two conditions:

\begin{enumerate}
\item\label{itm:criticaldef_RedIntersectsBlue} Whenever \([C]\in\mathcal{C}_1\) and \([D]\in\mathcal{C}_2\) for \(C\) and \(D\) on opposite sides of \(\Sigma\), \(\partial C\cap\partial D\) must be nonempty.

\item\label{itm:criticaldef_disjointdisks} For \(i=1,2\), there exists a pair of disjoint compressing disks, one on either side of \(\Sigma\), where each disk belongs to an isotopy class in \(\mathcal{C}_i\).
\end{enumerate}

Any link \(L\subset S^3\) can be isotoped so that all of its maxima lie above all of its minima (with respect to the standard height function on \(S^3\)).
After such an isotopy, \(L\) is said to be in \textbf{bridge position}.
Let \(\eta(L)\) be an open regular neighborhood of \(L\), and let \(F\) be a level sphere in \(S^3\) which separates all of the maxima of \(L\) from all of the minima.
Then the surface \(F'=F\backslash\eta(L)\subset M=S^3\backslash\eta(L)\) (a sphere with a finite number of open disks removed) is called a \textbf{bridge sphere} for \(L\).
See Figure \ref{fig:compression_disk_bdg_sphere}.

In our context, studying the bridge sphere \(F'\subset S^3\backslash\eta(L)\) is equivalent to studying the ``sphere \(F\subset S^3\) with marked points", where the \textbf{marked points} are the points of \(L\cap F\).
A disk in \(S^3\) is considered a compressing disk for the sphere \(F\) with marked points if and only if it is a compressing disk for \(F'\) in \(M\).
This implies that a compressing disk for \(F\) cannot intersect \(L\), and two compressing disks for \(F\) are considered to be in the same isotopy class if and only if they are isotopic in \(S^3/\eta(L)\).

A \textbf{bridge arc} is a component of \(L\backslash F\), and each of these bridge arcs has exactly one critical point, so each is parallel into \(F\).
An isotopy of a bridge arc into \(F\) sweeps out a disk, called a \textbf{bridge disk}.
(Note, bridge disks are not unique, even up to isotopy.)

\begin{figure}[ht]
\centering
\labellist
\pinlabel {\(F\)} at -13 106
\pinlabel {\(L\)} at 345 200
\endlabellist
\includegraphics[scale=.5]{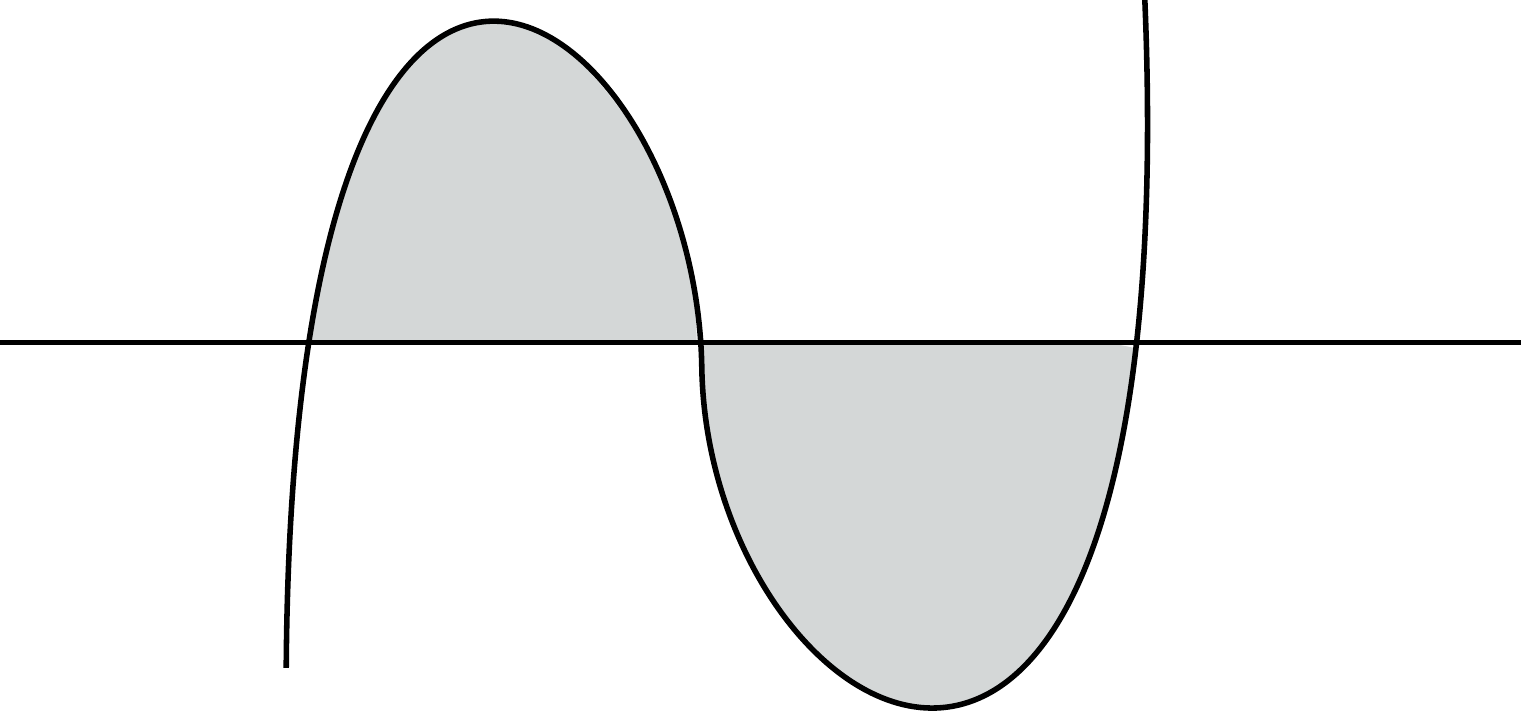}
\caption{If a bridge disk on one side of \(F\) intersects a bridge disk below \(F\) in a single point of the link, and are disjoint otherwise, then the link is perturbed.}
\label{fig:perturbed}
\end{figure}

A link is \textbf{perturbed} if there is a bridge disk above \(F\) and another bridge disk below \(F\) which intersect in a point of \(L\), and which are otherwise disjoint.
We will say the link is \textbf{perturbed at} a bridge arc \(\alpha\) if \(\alpha\) corresponds to either of these two bridge disks.
If a link \(L\) is perturbed, then there is an isotopy of \(L\) through the bridge disks which reduces the number of critical points of \(L\) by one maximum and one minimum.

\begin{figure}[ht]
\centering
\includegraphics[scale=.5]{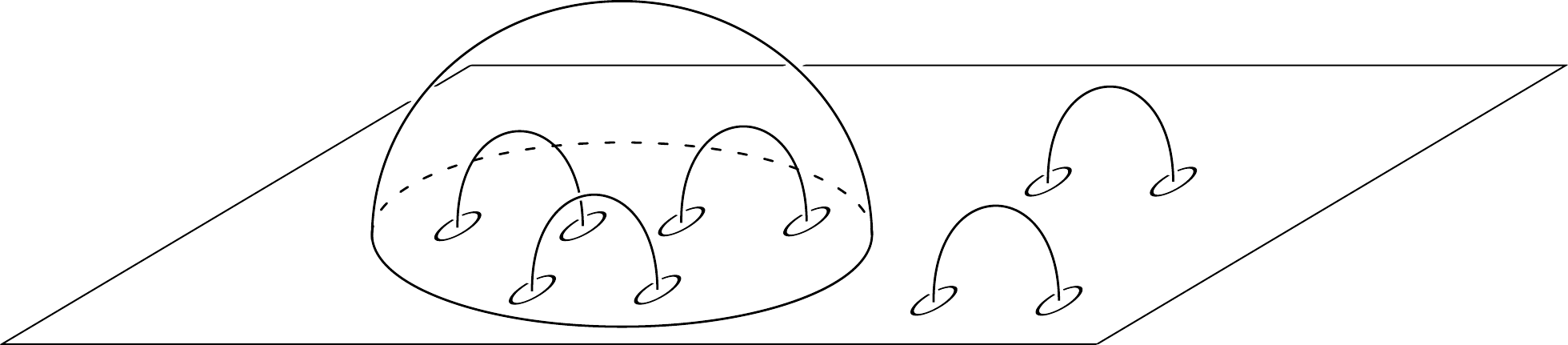}
\caption{A bridge sphere and a compressing disk.}
\label{fig:compression_disk_bdg_sphere}
\end{figure}

\begin{figure}[ht]
\centering
\begin{subfigure}[b]{.45\textwidth}
\includegraphics[width=\textwidth]{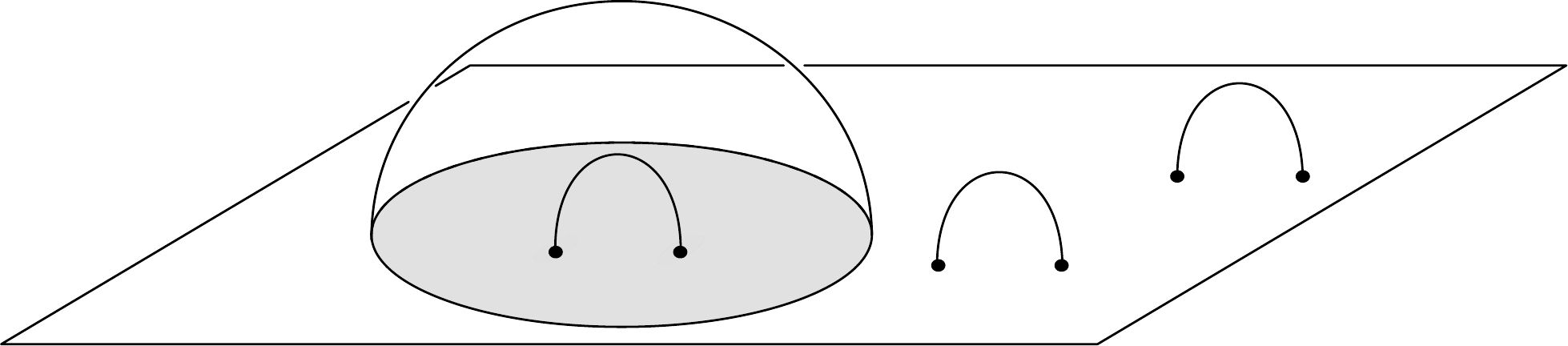}
\caption{}
\end{subfigure}
\begin{subfigure}[b]{.45\textwidth}
\includegraphics[width=\textwidth]{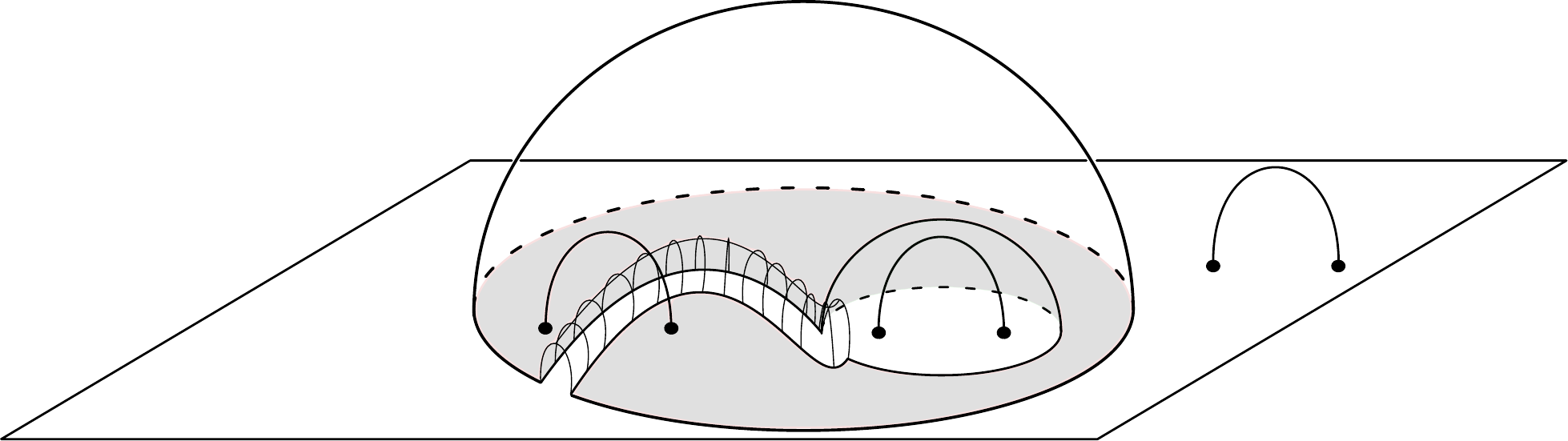}
\caption{}
\end{subfigure}
\caption{The two disks depicted are both caps for the leftmost bridge arc. In each figure, the disk in \(F\) with two marked points is shaded.}
\label{fig:cap_examples}
\end{figure}

A \textbf{cap} for \(F\) is a compressing disk \(C\) such that \(\partial C\) bounds a disk in \(F\) that contains exactly two marked points.
The existence of caps is guaranteed by the existence of bridge disks.
See Figure \ref{fig:cap_examples}.

\begin{figure}[ht]
\centering
\includegraphics[scale=.5]{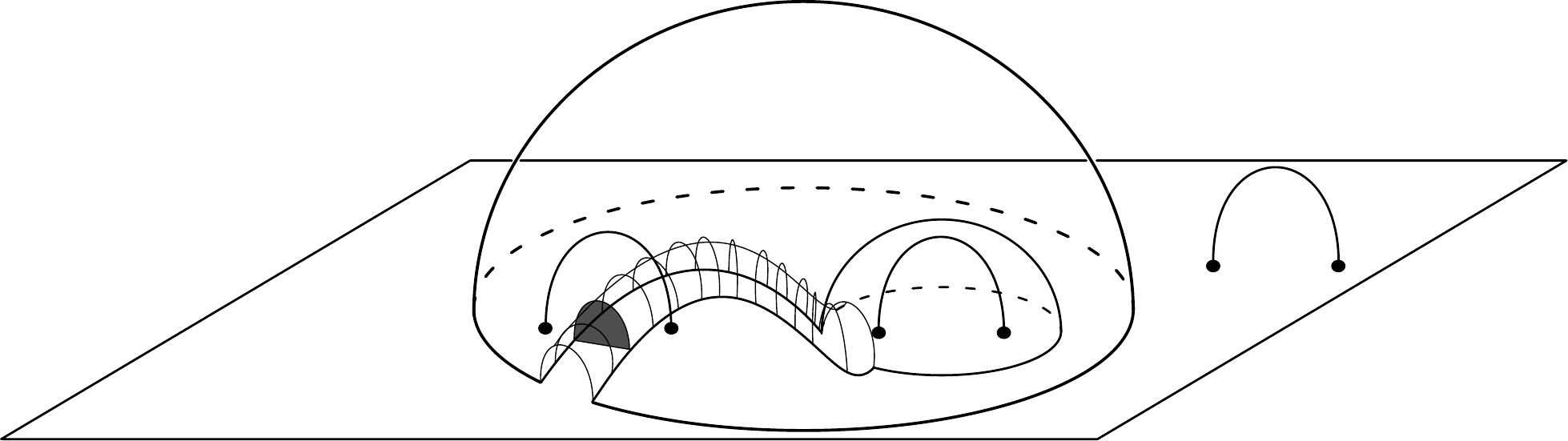}
\caption{A boundary-compressing disk is shaded dark gray.}
\label{fig:boundary-compressing_disk}
\end{figure}

We will also find useful the notions of boundary-compressing  along boundary-compressing disks. Suppose \(\Sigma\) is a surface with boundary properly embedded in a 3-manifold \(M\). A \textbf{boundary-compressing disk} (See Figure \ref{fig:boundary-compressing_disk} for an example) is defined to be a disk \(D\) with the following properties: 
\begin{itemize}
\item \(D\cap\partial M\) is an arc \(\beta\).
\item \(D\cap\Sigma\) is an arc \(\tilde{\beta}\).\footnote{Note: Our definition here is looser than the standard definition since we don't require \(\tilde{\beta}\) to be an \textit{essential} arc in \(\Sigma\). We drop this condition because we will want to perform boundary-compressions on disks, which have no essential arcs.}
\item \(\partial D=\beta\cup\tilde{\beta}\).
\item \(\beta\cap\tilde{\beta}=\partial\beta=\partial\tilde{\beta}\).
\end{itemize}


\begin{figure}[ht]
\centering
\includegraphics[scale=.2]{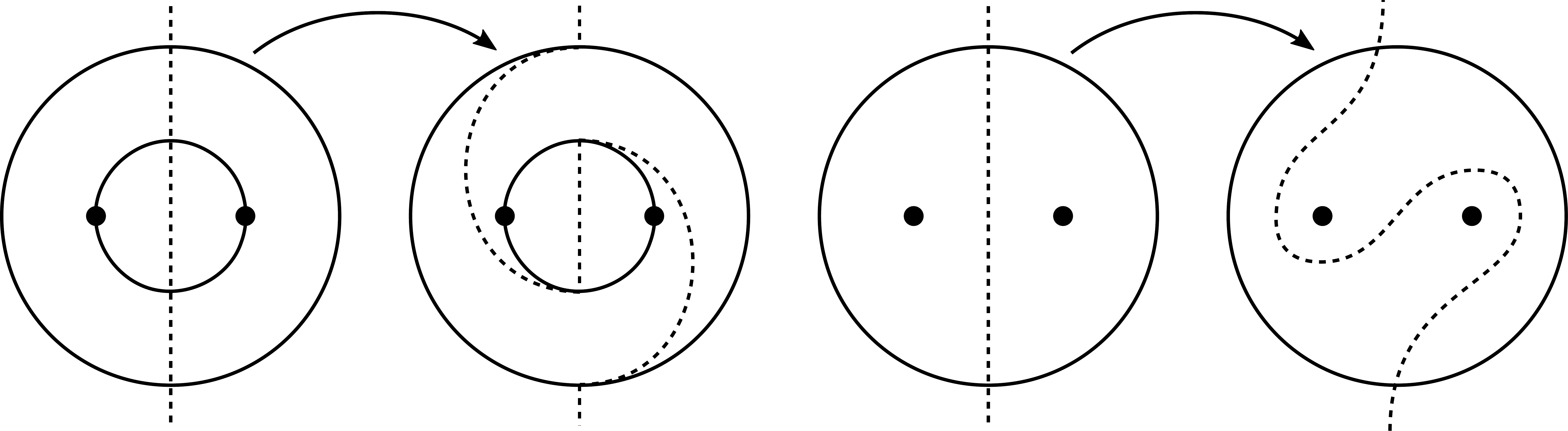}
\put(-305,102){Half twist}
\put(-128,102){Smoothed version}
\put(-341,55){\(D\)}
\put(-358,72){\(A\)}
\caption{Half Twist}
\label{fig:half_twist}
\end{figure}
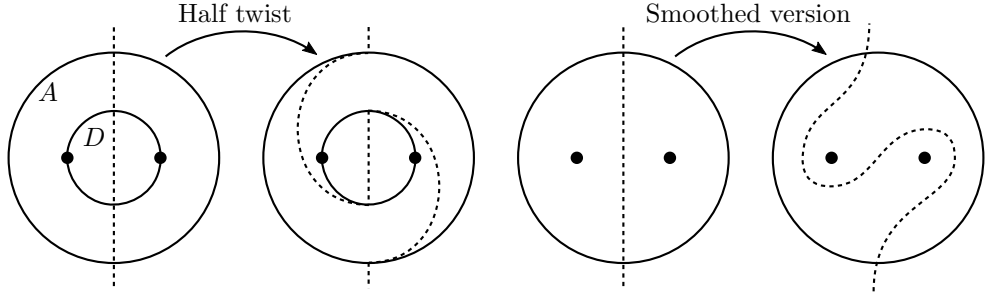

Next we define a half twist. 
Let \(U\) be a twice punctured (marked) disk.
The half twist can be simply understood to be the homeomorphism \(U\rightarrow U\) depicted in Figure \ref{fig:half_twist}.
To be technical, the \textbf{half twist} \(H\) of \(U\) is the homeomorphism \(H:U\rightarrow U\) defined as follows:
View \(U\) as a union \(A\cup_\phi D\) of the closed unit disk \(D\subset\C\) with an outer annulus \(A\) whose intersection is an embedded circle containing the two marked points.
See Figure \ref{fig:half_twist}.
Give \(D\) the coordinates \(re^{i\theta}\), and identify \(A\) with \(S^1\times I\), with the coordinates \({(\theta,t)}\). 
(\(A\) is glued to \(D\) by the map \(\phi\) taking \({(\theta,0)}\) in one component of \(\partial A\) to \(e^{i\theta}\) in \(\partial D\).)
Now define

\[
H(p)=
\begin{cases}
\begin{array}{cl}
(\theta+\pi t,t) & \text{for } p=(\theta,t)\in A \\ 
re^{i(\theta+\pi)} & \text{for } p=re^{i\theta}\in D
\end{array} 
\end{cases}
\]

For any punctured (marked) surface \(S\) with more than one puncture (marked point), if \(l\subset S\) is a loop which bounds an embedded twice-punctured (marked) disk \(U_l\subset S\) we can identify \(U_l\) with \(U\) and define the \textbf{half twist \(H_l\) about \(l\)} to be the homeomorphism \(H_l:S\rightarrow S\) 

\[
H_l(p)=
\begin{cases}
\begin{array}{cl}
H(p) & \text{for } p\in U \\ 
p & \text{for } p\not\in U
\end{array} 
\end{cases}
\]


\begin{figure}[ht]
\centering
\includegraphics[scale=.7]{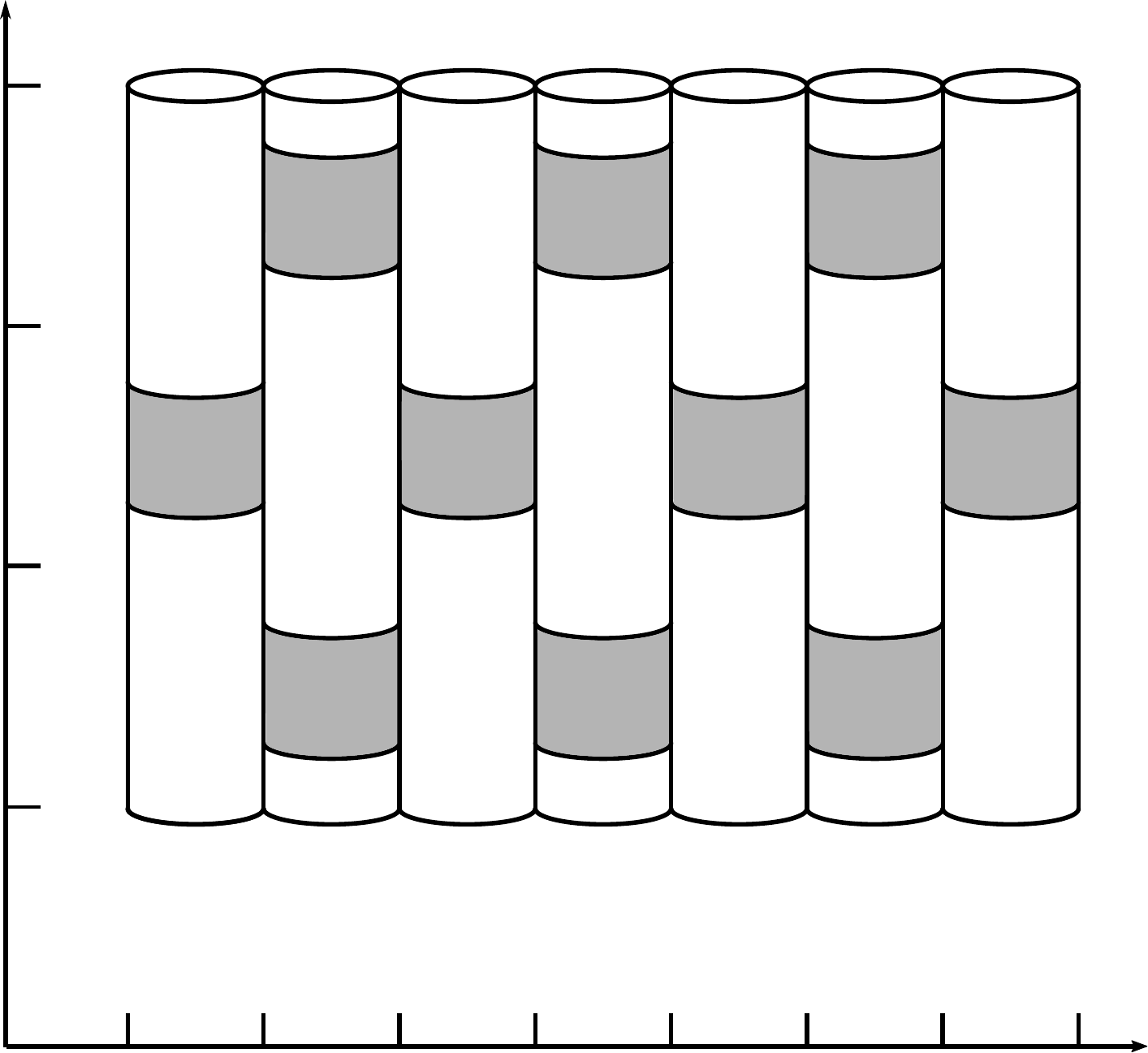}
\put(-292.1,-2){0}
\put(-292.1,58){1}
\put(-292.1,118){2}
\put(-292.1,178){3}
\put(-292.1,238){4}
\put(-292.1,260){\(z\)}
\put(-286.1,-10){0}
\put(-256.1,-10){1}
\put(-222.2,-10){2}
\put(-188.4,-10){3}
\put(-154.5,-10){4}
\put(-120.6,-10){5}
\put(-86.7,-10){6}
\put(-52.9,-10){7}
\put(-19,-10){8}
\put(0,-5){\(x\)}
\put(-212.1,85){\(\Tw_2^1\)}                     
\put(-144.4,85){\(\Tw_2^2\)}
\put(-76.8,85){\(\Tw_2^3\)}
\put(-211.9,205.3){\(\Tw_4^1\)}
\put(-144.3,205.3){\(\Tw_4^2\)}
\put(-76.8,205.3){\(\Tw_4^3\)}
\put(-245.7,145.1){\(\Tw_3^1\)}
\put(-178.1,145){\(\Tw_3^2\)}                    
\put(-110.5,145.1){\(\Tw_3^3\)}
\put(-43,145.1){\(\Tw_3^4\)}                         
\put(-245.7,250){\(\Cyl^1\)}
\put(-211.9,250){\(\Cyl^2\)}
\put(-178.1,250){\(\Cyl^3\)}
\put(-144.3,250){\(\Cyl^4\)}
\put(-110.5,250){\(\Cyl^5\)}
\put(-76.8,250){\(\Cyl^6\)}
\put(-43,250){\(\Cyl^7\)}
\caption{The cylinders which form a frame for plat links. In this figure, \(h=b=4\).}
\label{fig:cylinder_frame_unlabeled}
\end{figure}
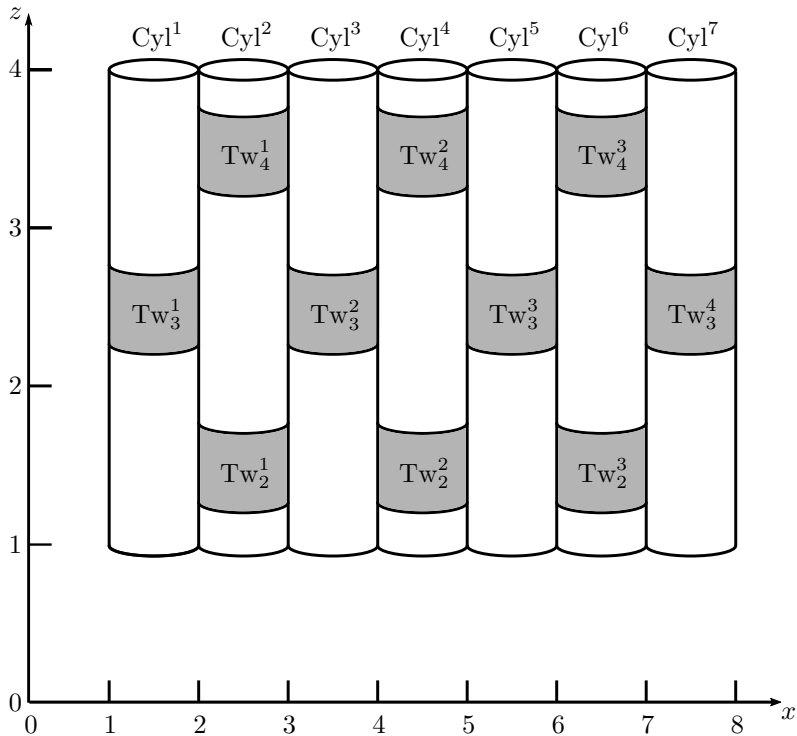

Now we will define plat position for a link.
We start by describing plat position for a braid.
Fix integers \(h\geq 2\) and \(b\geq 3\).  We begin by constructing a set of vertical cylinders in \(\R^3\) which will provide a frame for our braid.  (See Figure \ref{fig:cylinder_frame_unlabeled}.)
In the \(xy\)-plane, define \(c^1,c^2,\hdots, c^{2b-1}\) to be the circles of radius \(\frac{1}{2}\) such that \(c^j\) has center \((j+\frac{1}{2},0,0)\).
Now define \(\Cyl^j\) to be the cylinder obtained by crossing \(c^j\) with the \(z\)-interval \([{1,h}]\).

Next we define the \textbf{twist regions}, subcylinders of \(\{\Cyl^j\}\) depicted in Figure \ref{fig:cylinder_frame_unlabeled}.
Let \(i\in\{2,3,\hdots, h\}\).
For odd \(i\), let \(j\) range from \(1\) to \(b\), for even \(i\), let \(j\) range from \(1\) to \(b-1\). Then define 

\begin{displaymath}
\Tw_i^j = \left\{
\begin{array}{ll}
c^{2j-1}\times {[{i-\frac{3}{4}},{i-\frac{1}{4}}]} &  i \text{ odd}\\
c^{2j}\times {[{i-\frac{3}{4}},{i-\frac{1}{4}}]} &  i \text{ even}
\end{array}
\right.
\end{displaymath} 

In this paper, superscripts will usually denote an object's horizontal position (i.e., in the \(x\)-direction), and subscripts will usually denote an object's vertical position (i.e., in the \(z\)-direction).  This is reminiscent of typical matrix notation: an entry \(a_i^j\) of a matrix is in the \(i^\text{th}\) position vertically and the \(j^\text{th}\) position horizontally.  Here, however, unlike in a matrix, an array of objects is enumerated from bottom to top instead of from top to bottom.

Now we define a (non-closed) braid \(B\) to be in \((h,b)\)-\textbf{plat position} if it satisfies the following conditions:
\begin{enumerate}
\item \(B\) is a \(2b\)-strand braid embedded in \(\bigcup\Cyl^j\) whose bottom endpoints' coordinates are \(({j,0,1})\) for \(j=1,2,\hdots, 2b\) and whose top endpoints' coordinates are \(({j,0,h})\) for \(j=1,2,\hdots, 2b\).

\item Arcs of \(B\) are partitioned into two types: Vertical arcs which lie in the intersection of \(\bigcup\Cyl^j\) with the \(xz\)-plane, and Twisting arcs which are not contained in the \(xz\)-plane but are properly embedded in the twist regions.

\item Every Twisting arc is strictly increasing in a twist region, and as it ascends, it moves either strictly clockwise around the cylinder or strictly counterclockwise.

\end{enumerate}

These conditions guarantee that each twist region contains exactly two arcs of \(B\).
Observe that the endpoints of every Twisting arc must lie in the \(xz\)-plane.
Since Twisting arcs move strictly clockwise or counterclockwise, they intersect the \(xz\)-plane minimally.
That is, there is no isotopy of a Twisting arc in a twist region, relative its endpoints, which would decrease the intersection of that Twisting arc and the \(xz\)-plane. 
Looking down at the link from above, if the Twisting arc moves counterclockwise as it ascends, we define \(t\) to be \(-1\) plus the number of times the Twisting arc intersects the \(xz\)-plane.
Similarly, if the Twisting arc moves clockwise as it ascends, we define \(-t\) to be \(-1\) plus the number of times the Twisting arc intersects the \(xz\)-plane.
In other words, to each twisting arc we assign an integer \(t\) such that the twisting arc travels around the twist region through an angle of \(t\pi\).
Notice that  the other twisting arc in the same twist region must twist through an equal angle \(t\pi\) (or else the two arcs would intersect).
Thus we can define the number \(t\) to be the \textbf{twist number} corresponding to that twist region.
We will use \(t_i^j\) to denote the twist number for \(\Tw_i^j\).

\begin{figure}[ht]
\centering
\includegraphics[scale=.25]{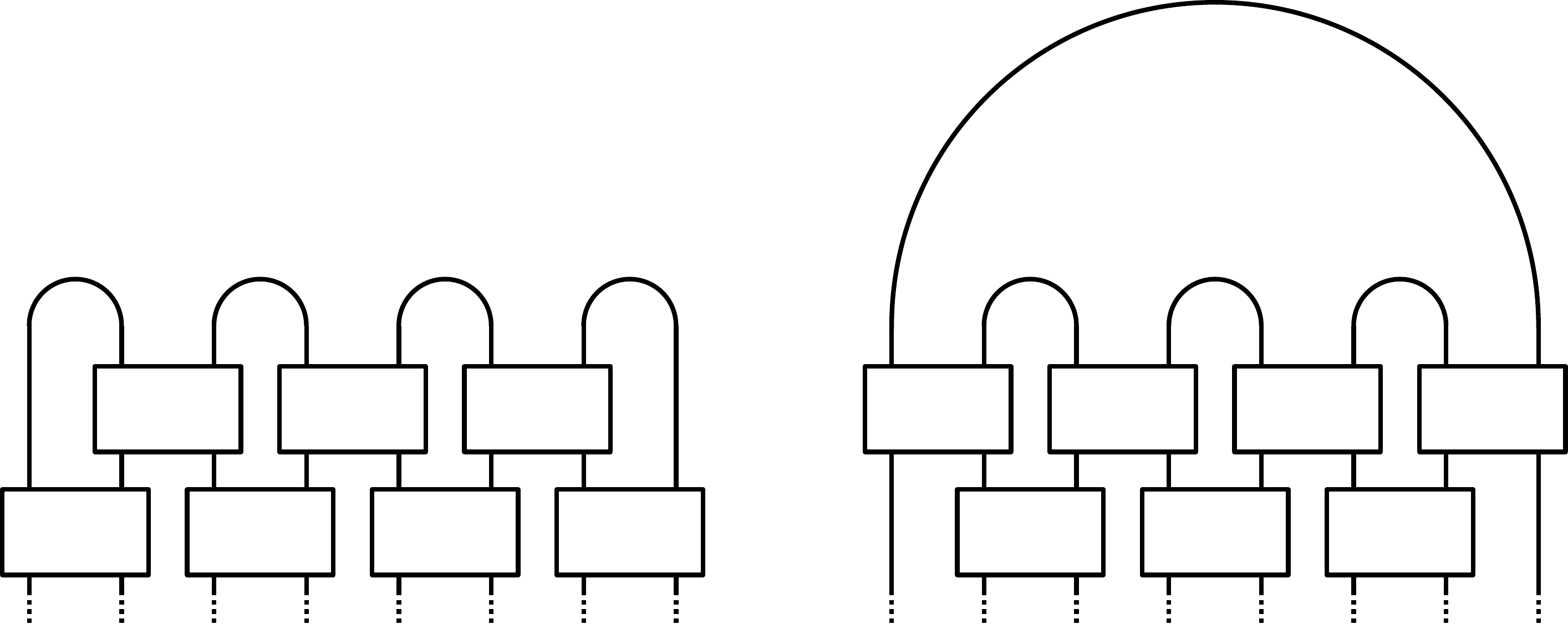}
\vspace{10pt}
\caption{Closing a braid in \({(h,b)}\)-plat position.
The left picture depicts the case when \(h\) is even, and the right picture depicts the case when \(h\) is odd.
In this paper, we only consider \(h=4\).}
\label{fig:plat_link_tops}
\end{figure}

There is a standard way to create a link from a braid in plat position.
See Figure \ref{fig:plat_link_tops}.
Along the bottom of \(B\), we have a row of \(2b\) endpoints at a height of \(z=1\).
For each odd \(j\) between \(1\) and \(2b\), connect the endpoints \((j,0,1)\) and \((j+1,0,1)\) with an lower semicircle in the \(xz\)-plane. 
The way we connect the top endpoints of \(B\) depends on the parity of \(h\).
Notice that if \(h\) is even, then the highest row of twist regions consists of \(b-1\) twist regions.
Along the top of \(B\) we have a row of \(2b\) endpoints at a height of \(z=h\).
For each odd \(j\) between \(1\) and \(2b\), connect the endpoints \((j,0,h)\) and \((j+1,0,h)\) with an upper semicircle in the \(xz\)-plane.
If \(h\) is odd, there are \(b\) twist regions at the top. 
For each even \(j\) between \(2\) and \(2b-2\), connect the endpoints \((j,0,h)\) and \((j+1,0,h)\) with an upper semicircle in the \(xz\)-plane.
Then connect the points \((j,0,1)\) and \((j,0,2b)\) with a larger upper semicircle in the \(xy\)-plane.
See Figure \ref{fig:plat_link_tops}.
Thus we obtain a link from \(B\).

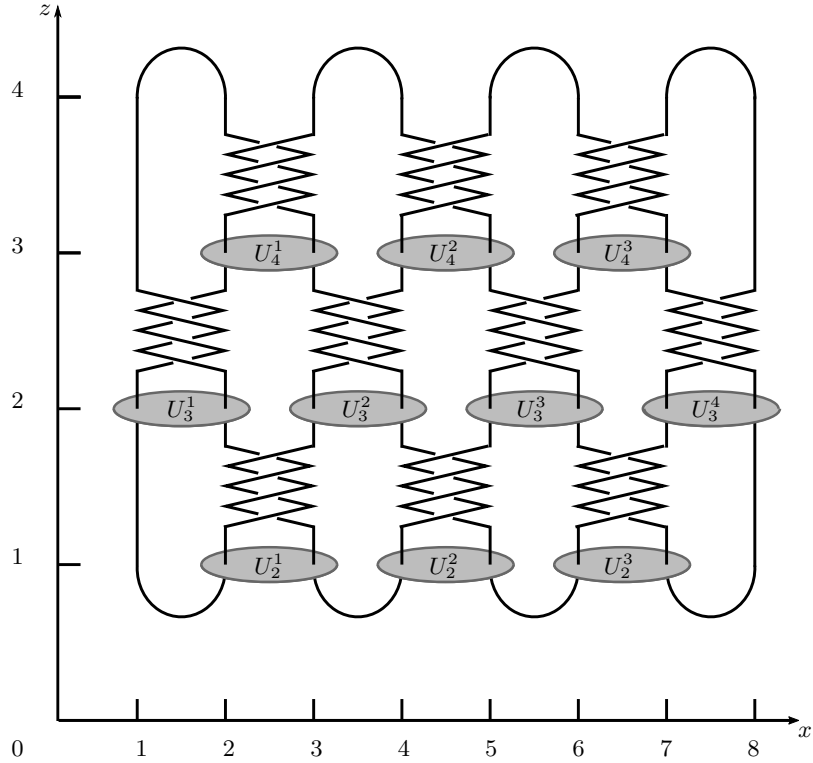
\begin{figure}[ht]
\centering
\labellist \small\hair 2pt
\pinlabel {\(0\)} [b] at -20 -20 
\pinlabel {\(1\)} [b] at -20 85
\pinlabel {\(2\)} [b] at -20 170
\pinlabel {\(3\)} [b] at -20 255
\pinlabel {\(4\)} [b] at -20 340
\pinlabel {\(z\)} [b] at -5 385 
\pinlabel {\(1\)} [b] at 47.75 -20 
\pinlabel {\(2\)} [b] at 95.5 -20
\pinlabel {\(3\)} [b] at 143.25 -20
\pinlabel {\(4\)} [b] at 191 -20
\pinlabel {\(5\)} [b] at 238.75 -20
\pinlabel {\(6\)} [b] at 286.5 -20
\pinlabel {\(7\)} [b] at 334.25 -20
\pinlabel {\(8\)} [b] at 382 -20
\pinlabel {\(x\)} [b] at 410 -10 
\pinlabel {\(U_2^1\)} [b] at 117.5 77 
\pinlabel {\(U_2^2\)} [b] at 213 77
\pinlabel {\(U_2^3\)} [b] at 308.5 77
\pinlabel {\(U_3^1\)} [b] at 69.75 163
\pinlabel {\(U_3^2\)} [b] at 165.25 163
\pinlabel {\(U_3^3\)} [b] at 260.75 163
\pinlabel {\(U_3^4\)} [b] at 356.25 163
\pinlabel {\(U_4^1\)} [b] at 117.5 248
\pinlabel {\(U_4^2\)} [b] at 213 248
\pinlabel {\(U_4^3\)} [b] at 308.5 248
\endlabellist
\includegraphics[scale=.7]{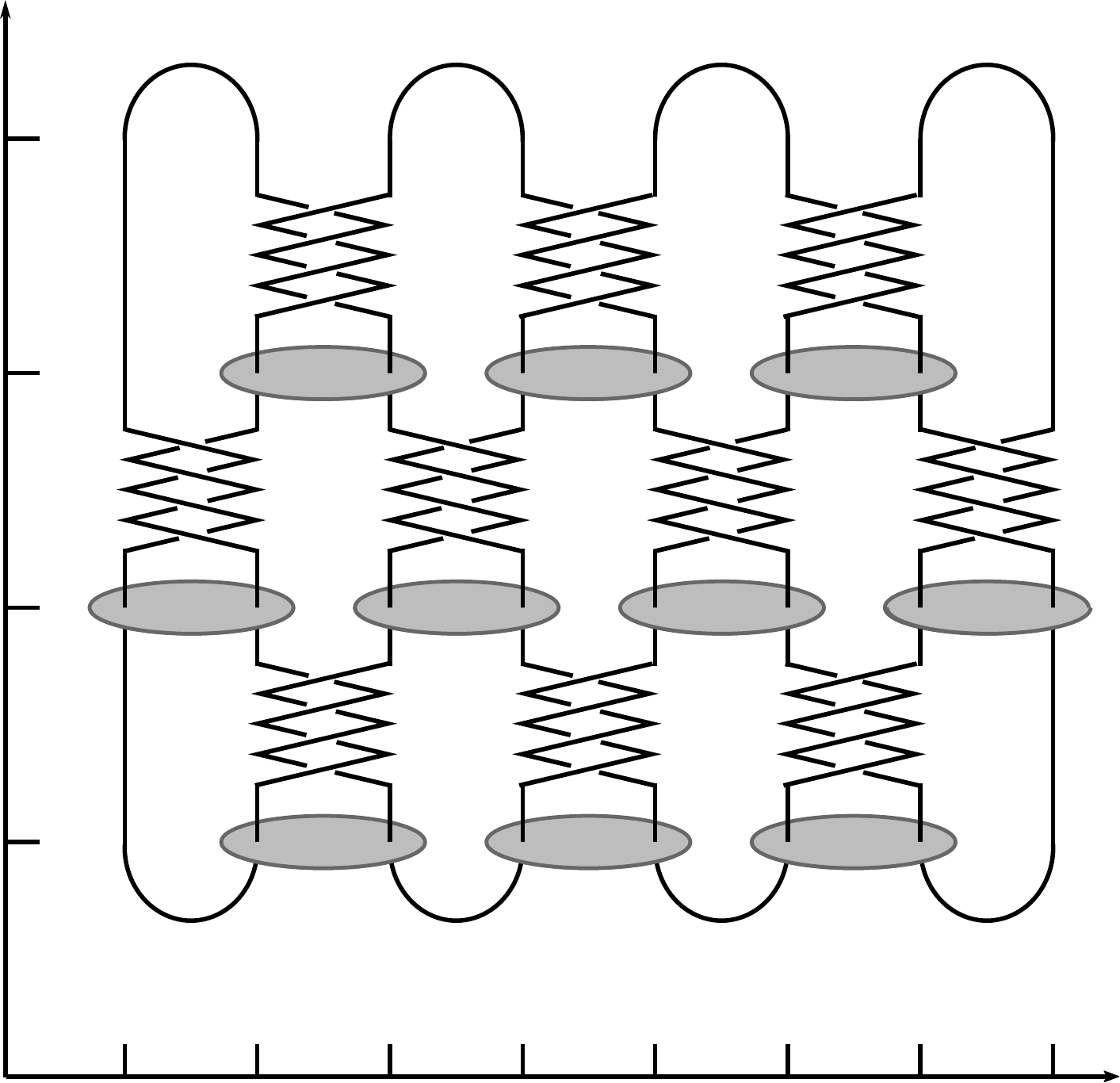}
\vspace{10pt}
\caption{A 4-twisted  \((4,4)\)-plat link with \(t_1^j=t_3^j=4\) for all \(j\) and \(t_2^j=-4\) for all \(j\). Also pictured are the \(U\)-disks defined in Section \protect\ref{sec:setting}.}
\label{fig:plat_link_with_U_disks}
\end{figure}

Any link constructed in this way from a braid in \(({h,b})\)-plat position will be said to be \textbf{a link in \(({h,b})\)-plat position}.
Observe that such a link is in a bridge position with bridge number \(b\).
For \(n\in\N\), a link in  \(({h,b})\)-plat position is called \textbf{\(n\)-twisted} if \(|t_i^j|\geq n\) for all \(i,j\).
See Figure \ref{fig:plat_link_with_U_disks} for an example of a \(4\)-twisted link in \({(4,4)}\)-plat position.

\section{Setting}\label{sec:setting}

\begin{figure}[ht]
\centering
\labellist \small\hair 2pt
\pinlabel {\(-x\)} [b] at 30 81
\pinlabel {\(+x\)} [b] at 440 81
\pinlabel {\(-y\)} at 1 1
\pinlabel {\(-z\)} at 196 35
\pinlabel {\(+z\)} at 196 320
\endlabellist
\includegraphics[scale=.35]{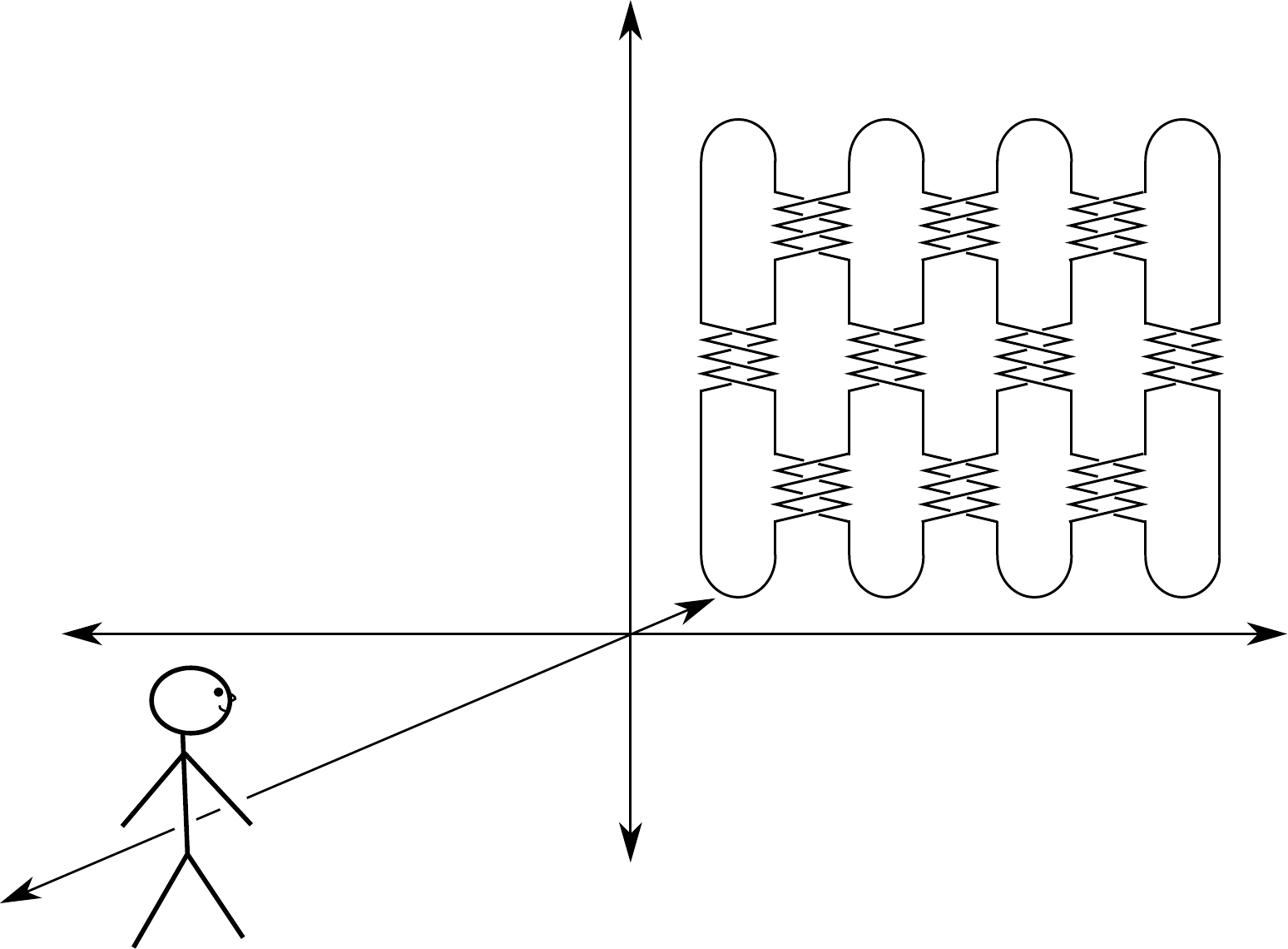}
\caption{Point of view}
\label{fig:point_of_view}
\end{figure}

Let \(f:S^3\rightarrow {[{-\infty},{\infty}]}\) be a height function on \(S^3\), and let \(\Lambda\) be a strictly increasing arc in \(S^3\) with endpoints being the two critical points of \(S^3\).
Then \(S^3\backslash\Lambda\) is homeomorphic to \(\R^3\).
Let \(\phi:\R^3\rightarrow S^3\backslash\Lambda\) be a homeomorphism that respects the height function \(f\) on \(S^3\) and the standard height function on \(\R^3\).
That is, if \(P_i\) is the level plane in \(\R^3\) of height \(i\), and \(F_i\) is the level sphere in \(S^3\) of height \(i\), then \(\phi(P_i)\) is \(F_i\) minus the point \(F_i\cap\Lambda\).
Throughout the rest of the paper, we adopt the point of view of someone standing on the \(-y\) side of the \(xz\)-plane.
This gives meaning to words like ``left," ``right", ``up," ``down," ``horizontal," and ``vertical."
(See Figure \ref{fig:point_of_view}.)

\textbf{The Link \(L\):}
Let \(L\) be a \(2\)-twisted link in \((4,4)\)-plat position with bridge sphere \(F\simeq F_1=f^{-1}(1)\subset S^3\).
Suppose \(L\) has twist numbers \(\left\{t_i^j\right\}\) such that for all \(j\), the twist numbers \(t_2^j,t_4^j\) are positive, and \(t_3^j\) are negative.
We will work with this link \(L\) for the rest of the paper.

Consider the diagram \(D(L)\) for \(L\) obtained by projection to the \((x,z)\)-plane.
Our choice of signs for the twist numbers \(\left\{t_i^j\right\}\) makes \(D(L)\) an alternating diagram, and we know that \(L\) is a split link if and only if \(D(L)\) is a split diagram (see Theorem 4.2 of \cite{Lickorish}).
Since \(D(L)\) is clearly a non-split diagram, it follows that \(L\) is a non-split link.

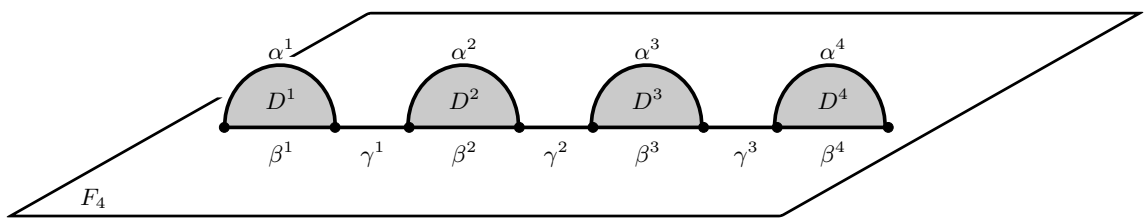
\begin{figure}[h!]
\vspace{10pt}
\centering
\labellist \small\hair 2pt
\pinlabel {\(F_4\)} [b] at 38 2 
\pinlabel {\(\beta^1\)} [b] at 120 20 
\pinlabel {\(D^1\)} [b] at 120 45
\pinlabel {\(\alpha^1\)} [b] at 120 67
\pinlabel {\(\gamma^1\)} [b] at 160 20 
\pinlabel {\(\beta^2\)} [b] at 200 20
\pinlabel {\(D^2\)} [b] at 200 45
\pinlabel {\(\alpha^2\)} [b] at 200 67
\pinlabel {\(\gamma^2\)} [b] at 240 20
\pinlabel {\(\beta^3\)} [b] at 280 20
\pinlabel {\(D^3\)} [b] at 280 45
\pinlabel {\(\alpha^3\)} [b] at 280 67
\pinlabel {\(\gamma^3\)} [b] at 323 20
\pinlabel {\(\beta^4\)} [b] at 361 20
\pinlabel {\(D^4\)} [b] at 361 45
\pinlabel {\(\alpha^4\)} [b] at 361 67
\endlabellist
\includegraphics[width=\textwidth]{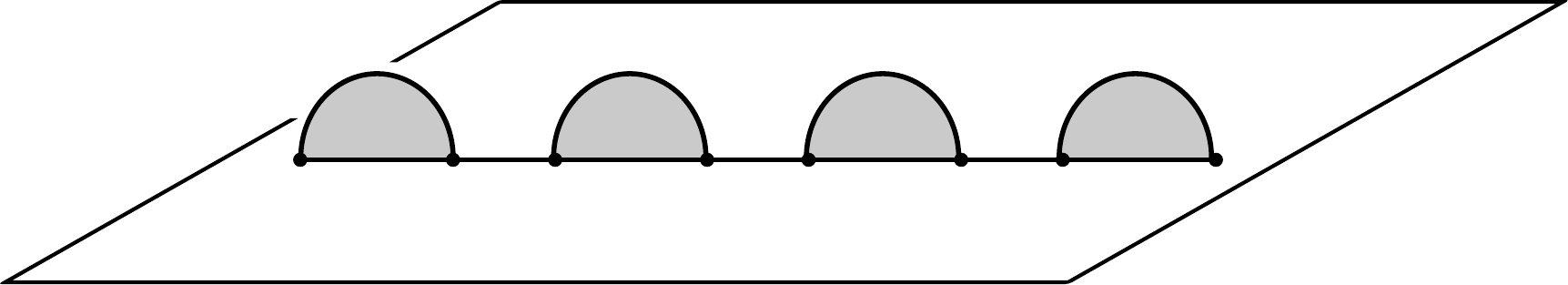}
\caption{The bridge disks \(D^k\) above \(F\)}
\label{fig:bridges}
\end{figure}

\textbf{The Bridge Disks:}
 Refer to Figure \ref{fig:bridges}.
Above \(F_4\) in the \((x,z)\)-plane lie the \(4\) upper bridge arcs.
We will name these arcs consecutively from left to right, \(\alpha^1,\alpha^2,\alpha^3,\alpha^4\).
Vertical projection of the bridge arcs into \(F\) gives us \(4\) straight line segments at level \(4\) which we will name consecutively from left to right, \(\beta^1\), \(\beta^2\), \(\beta^3\), and \(\beta^4\).
We will refer to these as \textbf{\(\beta\)-arcs}.
Observe that \(\alpha^j\) and \(\beta^j\) form a loop.
Let \(D^j\) be the disk in the \((x,z)\)-plane bounded by this loop.
Notice each \(D^j\) is a bridge disk.
Define \(\gamma^j\) to be the straight line segment connecting the points \({({2j},0,4)}\) and \({({2j+1},0,4)}\).
We will refer to these arcs as \textbf{\(\gamma\)-arcs}.

\begin{figure}[h!]
\vspace{10pt}
\centering
\labellist \small\hair 2pt
\pinlabel {\(\beta'\)} [b] at 1358 160 
\pinlabel {\({D^4}'\)} at 1365 100 
\pinlabel {\({\alpha^4}'\)} at 1520 60 
\endlabellist
\includegraphics[scale=.13]{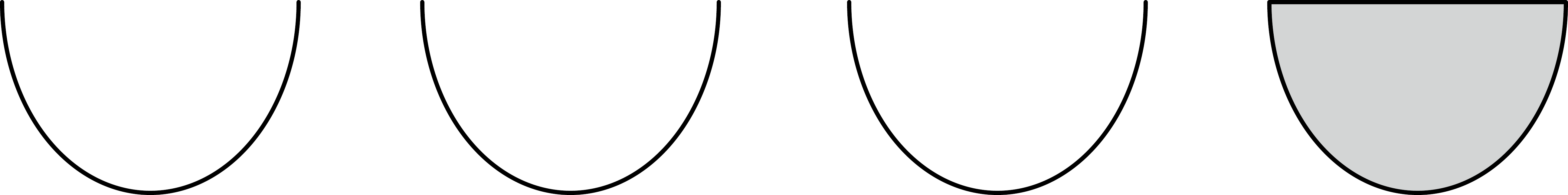}
\caption{The bridge arcs below \(F\).
Of these, we care especially about the rightmost bridge arc \({\alpha^4}'\), the straight line segment \(\beta'\) between its endpoints, and the bridge disk \({D^4}'\) cobounded by \({\alpha^4}'\) and \(\beta'\).}
\label{fig:alpha_4_prime}
\end{figure}

Below \(F_1\) in the \((x,z)\)-plane lie the \(4\) lower bridge arcs, which we will call consecutively from left to right, \({\alpha^1}'\), \({\alpha^2}'\), \({\alpha^3}'\), and \({\alpha^4}'\).
We will also define \(\beta'\) to be the straight line segment between the endpoints of \({\alpha^4}'\).
(See Figure \ref{fig:alpha_4_prime}.)
\(\beta'\cup{\alpha^4}'\) is a loop which bounds a bridge disk in the \((x,z)\)-plane which we will call \({D^4}'\).

%

\textbf{\(l\)-loops and \(U\)-disks}: For \(i=1,3\), define \(l_i^j\) to be the circle of radius \(\frac{3}{4}\) in the horizontal plane \(z=i-1\) centered at the point \({(2j+\frac{1}{2},0,i-1)}\). For \(i=2\), define \(l_2^j\) to be the circle of radius \(\frac{3}{4}\) in the horizontal plane \(z=2\) centered at the point \({(2j-\frac{1}{2},0,i-1)}\). Then each \(l_i^j\) bounds a twice-punctured disk which we call \(U_i^j\). Each disk \(U_i^j\) lies below a corresponding twist region \(\Tw_i^j\). See Figure \ref{fig:plat_link_with_U_disks}.

\textbf{\(\sigma\)-projection}: It will be convenient to be able to talk about simple vertical projections.  If \(P_i\) denotes the horizontal plane in \(\R^3\) of height \(i\), let \(\sigma_i:\R^3\rightarrow P_i\) be the vertical projection map defined by \((x,y,z)\mapsto(x,y,i)\).

\textbf{\(\tau\)-projection}: We will also need a more subtle type of projection which respects \(L\).
We closely follow Johnson and Moriah \cite{Jesse}:
Note that \(L\) intersects each \(F_i\) in the same number of points and these points vary continuously as \(i\) varies from \(1\) to \(h\).
This can be thought of as an isotopy of these \(2b\) marked points in \(F\) which extends to an ambient isotopy of \(F_h\).
To be precise, there is a projection map \(\tau_i:F\times [1,h]\rightarrow F_i\) for each \(i\in[1,h]\) that sends each arc component of the plat braid to a point \((j,0,i)\) for some \(j=1,2,\cdots,2b\) and defines a homeomorphism \(F_{i'}\rightarrow F_i\) for each \(i'\in[1,h]\).
These homeomorphisms are canonical up to isotopy fixing the points \(L\cap F_i\), and the induced homeomorphism \(F_i\rightarrow F_i\) is the identity.
Further, for each \(i=2,3,4\), the homeomorphism induced by \(\tau_i:F_{i-1}\rightarrow F_i\) is a composition of half twists about the \(l\)-loops, which can be expressed as \(\Pi_j H_{l_i^j}^{t_i^j}\).
(Note that since the \(U\)-disks at any given level are pairwise disjoint, the corresponding half twists all commute with each other.)
Therefore the homeomorphism \(f:F_{0}\rightarrow F_4\) induced by \(\tau_4\) can be expressed as \(f=\left(\prod_j H_{l_4^j}^{t_4^j}\right)\left(\prod_j H_{l_3^j}^{t_3^j}\right)\left(\prod_j H_{l_2^j}^{t_2^j}\right)\).

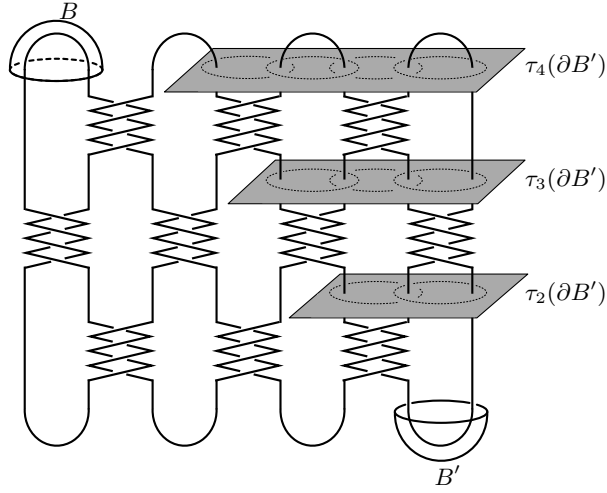
\begin{figure}[ht]
\centering
\labellist \small\hair 2pt
\pinlabel {\(B\)} at 45 340
\pinlabel {\(B'\)} [r] at 345 -10
\pinlabel {\(\tau_2(\partial B')\)} at 420 127
\pinlabel {\(\tau_3(\partial B')\)} at 420 213
\pinlabel {\(\tau_4(\partial B')\)} at 420 299
\endlabellist
\includegraphics[scale=.5]{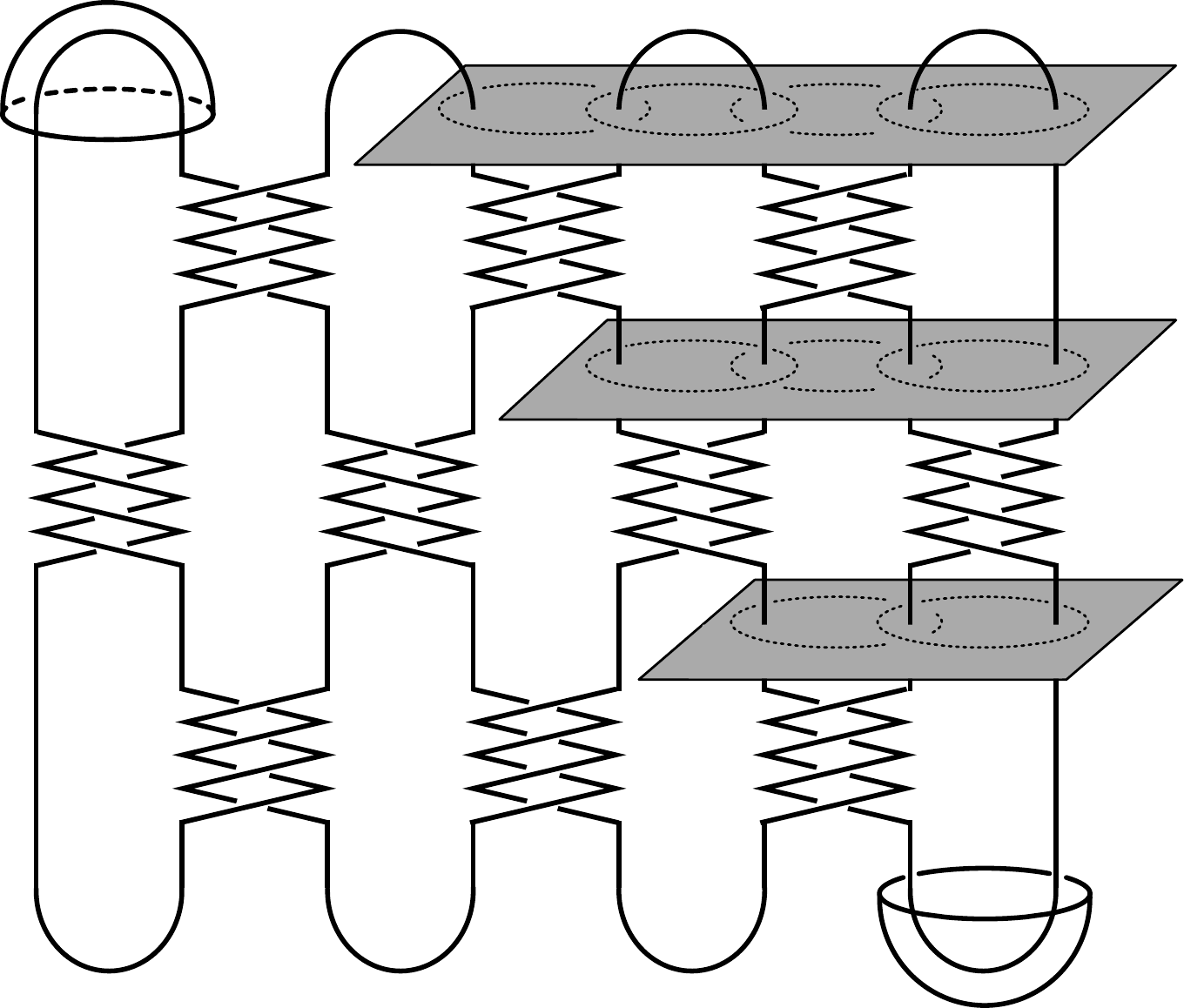}
\vspace{10pt}
\caption{The two blue compressing disks, \(B\) and \(B'\)}
\label{fig:B_and_Bprime_disjoint}
\end{figure}

\textbf{Disk partition}: Our goal is to show that \(F\) is critical, so we need to exhibit a partition \(\mathcal{C}_1\sqcup\mathcal{C}_2\) of the isotopy classes of the compressing disks for \(F\).
(Recall the definition of \textit{critical} from the beginning of Section \ref{sec:defs}.)
Above \(F\), let \(B\) be the frontier of a regular neighborhood of \(D^1\).
Below \(F\), let \(B'\) be the frontier of a regular neighborhood of \({D^4}'\).
\(B\) and \(B'\)  are depicted in Figure \ref{fig:B_and_Bprime_disjoint}.
Let \([B],[B']\in\mathcal{C}_2\), and let all of the other isotopy classes of compressing disks be in \(\mathcal{C}_1\).
It will be convenient to refer to a compressing disk \(C\) as \textbf{red} if \([C]\in\mathcal{C}_1\) and \textbf{blue} if \([C]\in\mathcal{C}_2\).


\section{The Labyrinth}\label{sec:labyrinth}


\begin{figure}[ht]
\centering
\labellist \small\hair 2pt
\pinlabel {\(G^\text{b}\)} at 325 75
\pinlabel {\(G^\text{o}\)} at 495 405
\pinlabel {\(G^\text{p}\)} at 655 75
\pinlabel {\(\tau_4(\partial B')\)} at 710 360
\pinlabel {Lab} at 600 440
\pinlabel {\(\beta^2\)} at 140 255
\pinlabel {\begin{rotate}{60}{brown}\end{rotate}} at 170 250
\pinlabel {\(\gamma^2\)} at 230 255
\pinlabel {\(\beta^3\)} at 360 255
\pinlabel {\begin{rotate}{60}{orange}\end{rotate}} at 330 250
\pinlabel {\begin{rotate}{60}{purple}\end{rotate}} at 495 250
\pinlabel {\(\gamma^3\)} at 610 255
\pinlabel {\(\beta^4\)} at 740 255
\endlabellist
\includegraphics[width=\textwidth]{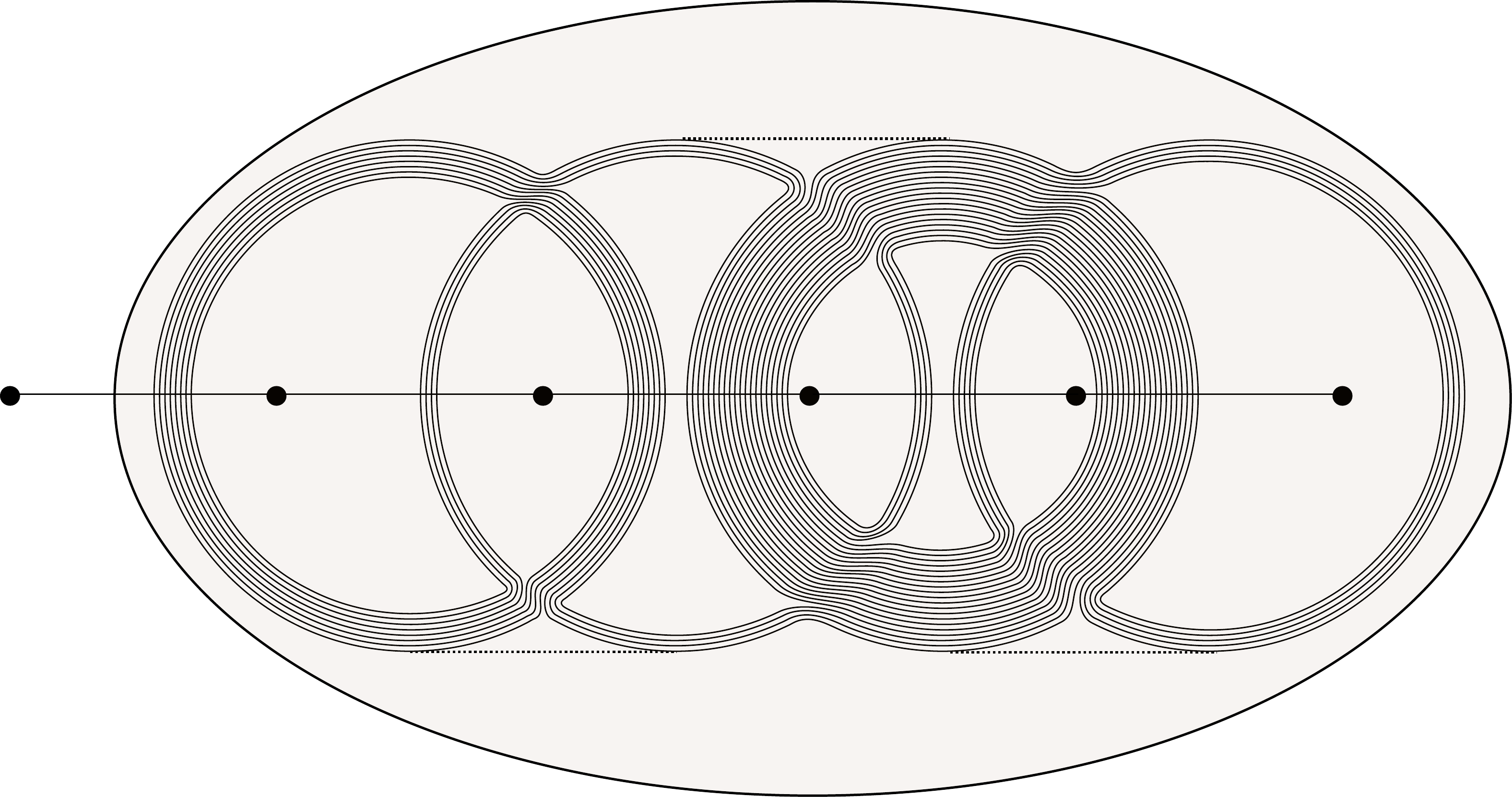}
\caption{The thin spiraling line is \(\tau_4(\partial B')\) in the case when \(t_1^j=t_3^j=2\) and \(t_2^j=-2\) for all \(j\).
Also depicted is the disk Lab which contains \(\tau_4(\partial B')\), and the three gates, \(G^\text{b}\), \(G^\text{o}\), and \(G^\text{p}\), as well as the three colored points.
In general, the five marked points in Lab may be permuted differently when different twist numbers are chosen.}
\label{fig:boundary_of_Bprime_two_twists_revised_3}
\end{figure}
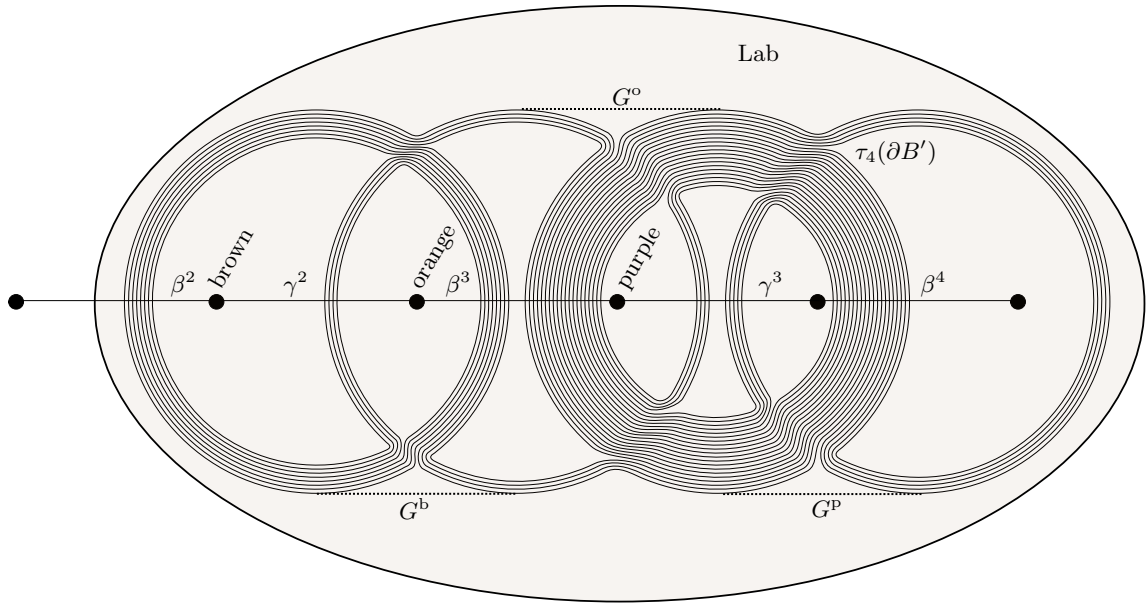

We are interested in the image of \(\partial B'\) under \(f\), the homeomorphism induced by \(\tau_4\).
As \(\tau_4\) takes \(\partial B'\) up from level 1 to level 4, \(\partial B'\) undergoes a series of twists, following the strands of the link.
As explained in Section \ref{sec:setting}, \(f\) can be expressed as a product of half twists around the \(l\)-loops.
Explicitly, \(f=\left(\prod_j H_{l_4^j}^{t_4^j}\right)\left(\prod_j H_{l_3^j}^{t_3^j}\right)\left(\prod_j H_{l_2^j}^{t_2^j}\right)\).
In Figure \ref{fig:boundary_of_Bprime_two_twists_revised_3} we show exactly what \(\partial B'\) looks like under \(f\) in a case where \(|t_i^j|=2\) for all \(j\).

\begin{figure}[h!]
\centering
\labellist \small\hair 2pt
\pinlabel {\(H_{l_i^j}\)} at 310 214
\pinlabel {\(U_i^j\)} at 60 180
\pinlabel {\(\lambda\)} at 125 200
\endlabellist
\includegraphics[width=\textwidth]{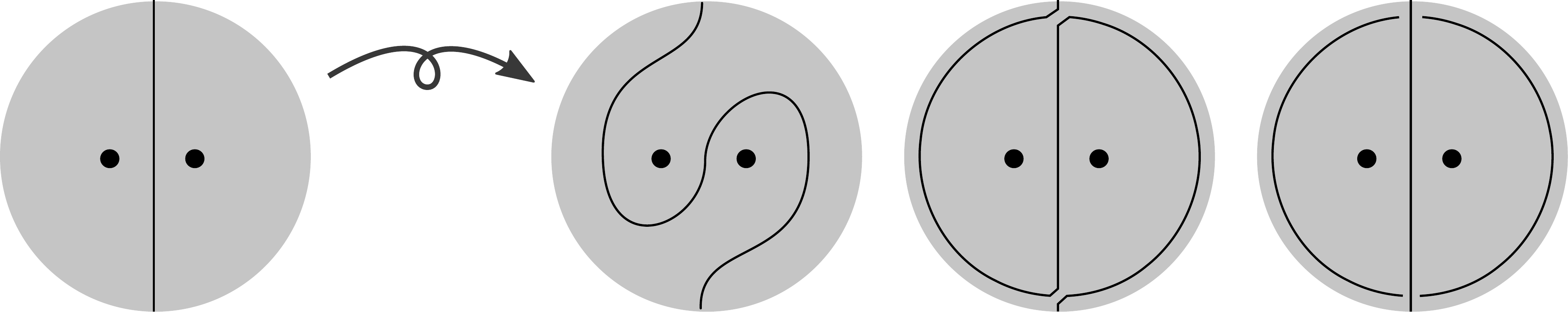}
\caption{How to view a half twist about an \(l\)-loop as a link diagram}
\label{fig:half_twist_surgery_one_arc2}
\end{figure}

\begin{figure}[h!]
\centering
\includegraphics[scale=.2]{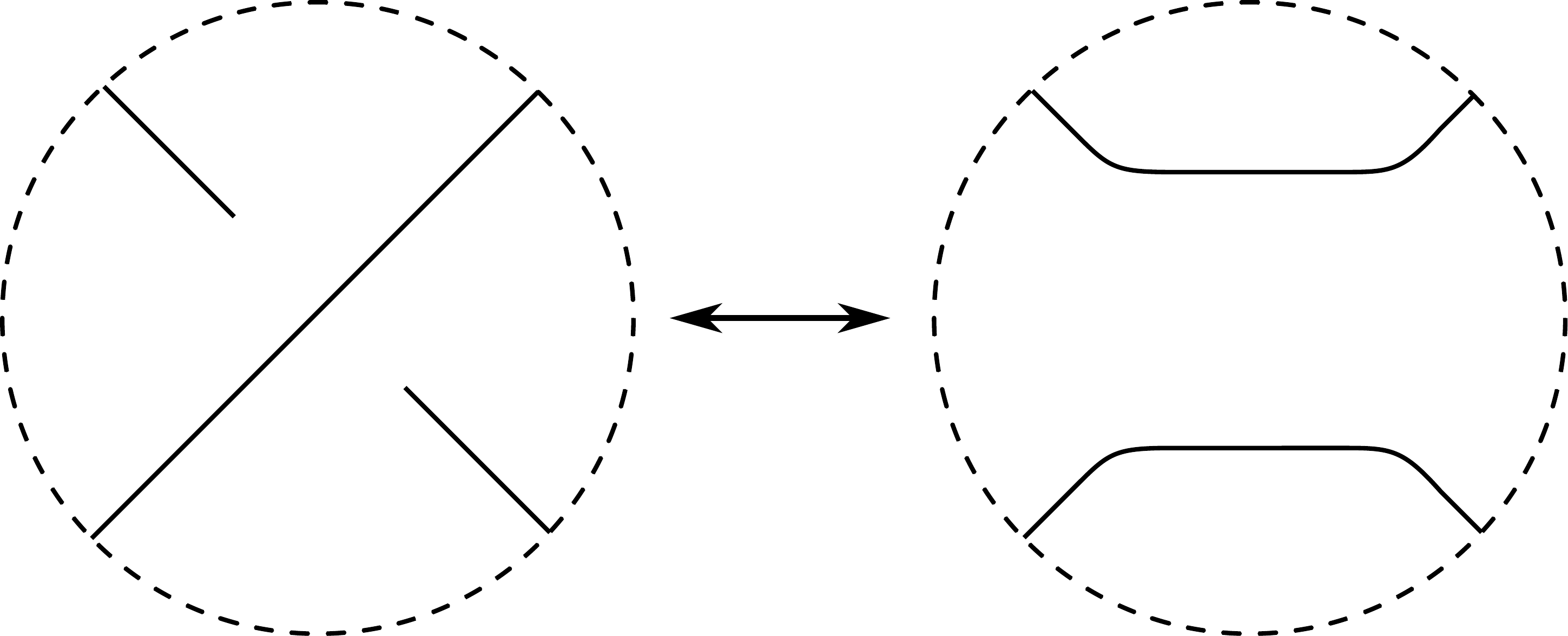}
\caption{Rule for smoothing a crossing}
\label{fig:surgery_smoothing_crossing2}
\end{figure}

Of course, we want to understand what \(\tau_4(\partial B')\) will look like in general for any set of twist numbers \(\left\{t_i^j\right\}\).
To do so, we will devise a convenient pictorial notation for loops if \(F\) which undergo half twists about \(l\)-loops.
Suppose \(\lambda\) is a loop in \(F\) with one arc of intersection with \(U_i^j\) which passes between the two marked points, as in the first picture in Figure \ref{fig:half_twist_surgery_one_arc2}.
The second and third pictures in Figure \ref{fig:half_twist_surgery_one_arc2} are homotopic pictures of \(H_{l_i^j}(\lambda)\).
Observe that the third picture is reminiscent of knot diagram smoothing.
In fact, using the smoothing rule in Figure \ref{fig:surgery_smoothing_crossing2}, we can ``unsmooth" \(H_{l_i^j}(\lambda)\) into the fourth picture, which consists of a link diagram on \(F\) with two components: 1) the original \(\lambda\) and 2) \(l_i^j\).
This suggests a convenient algorithm to draw any such loop \(\lambda\) after a (positive) half twist about an \(l\)-loop:
We start by drawing \(\lambda\).
Next we draw the \(l\)-loop so that it passes \textit{under} \(\lambda\).
Last we smooth the crossings using the rule in Figure \ref{fig:surgery_smoothing_crossing2}.
If we want to perform a negative half twist instead, the procedure is the same, except that we draw the \(l\)-loop crossing \textit{over} \(\lambda\).

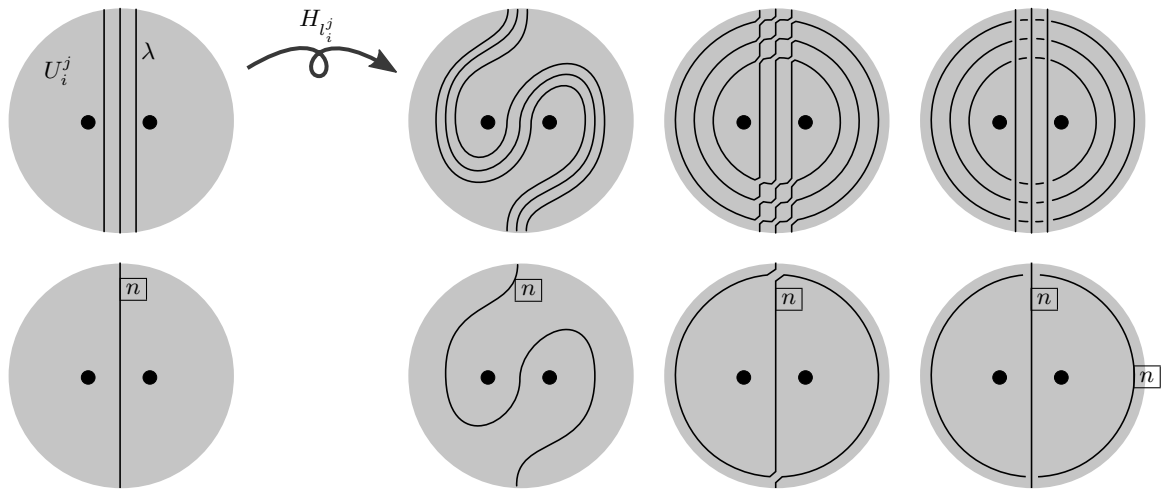
\begin{figure}[ht]
\centering
\labellist \small\hair 2pt
\pinlabel {\(H_{l_i^j}\)} at 310 465
\pinlabel {\(U_i^j\)} at 50 420
\pinlabel {\(\lambda\)} at 140 440
\pinlabel {
	\setlength{\fboxsep}{2pt}
	\fbox{\(n\)}
	} at 125 200
\pinlabel {
	\setlength{\fboxsep}{2pt}
	\fbox{\(n\)}
	} at 521 199
\pinlabel {
	\setlength{\fboxsep}{2pt}
	\fbox{\(n\)}
	} at 782 190
\pinlabel {
	\setlength{\fboxsep}{2pt}
	\fbox{\(n\)}
	} at 1038 190
\pinlabel {
	\setlength{\fboxsep}{2pt}
	\fbox{\(n\)}
	} at 1141 112
\endlabellist
\includegraphics[width=\textwidth]{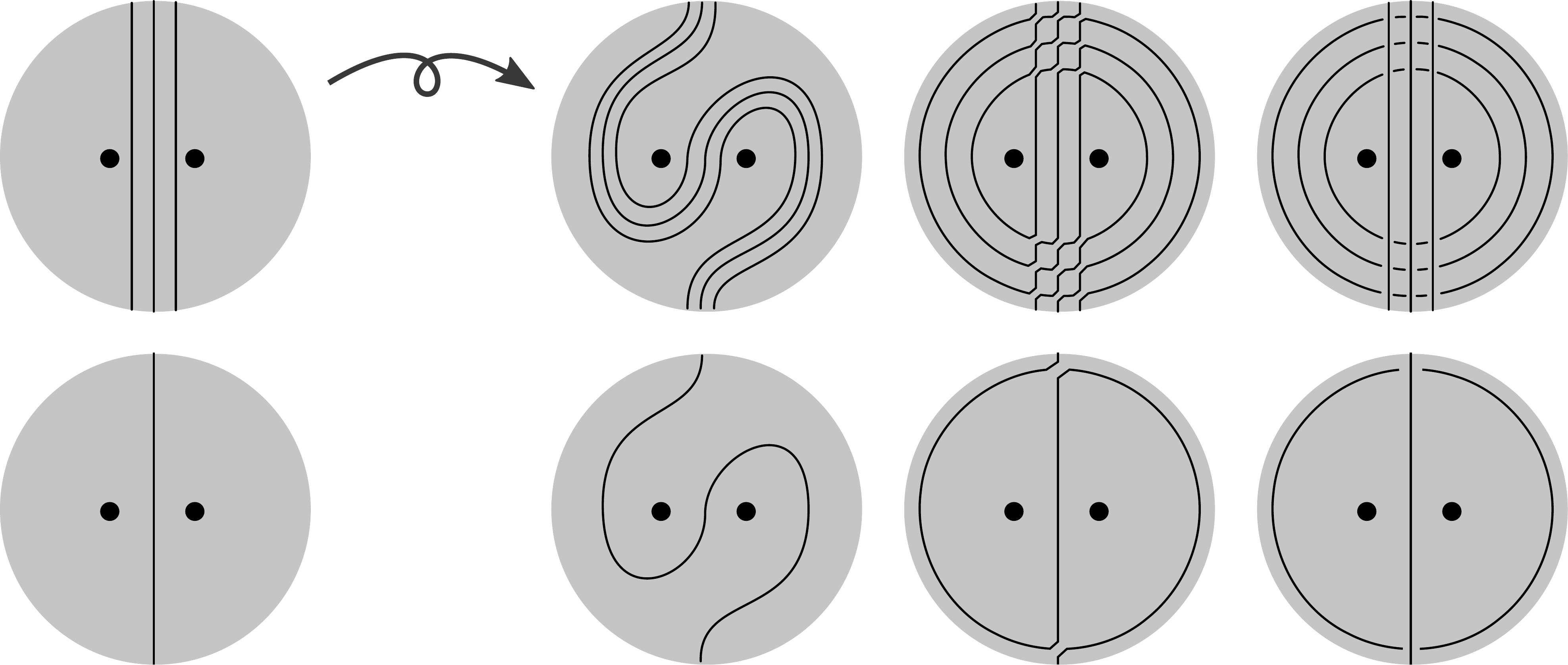}
\caption{How to view a half twist about an \(l\)-loop as a link diagram.
In the top row of pictures, we show specifically what the situation looks like when \(n=3\).
The bottom row shows that the whole process can be done with an arbitrary number \(n\) of strands through \(U_i^j\).
In all the following figures that depict half twists, when we label the number of parallel strands present, as in the bottom row here, we will box the labels with squares.}
\label{fig:half_twist_general2}
\end{figure}

We can generalize this process to the situation where there are \(n\) strands of \(\lambda\) passing through \(U_i^j\).
In Figure \ref{fig:half_twist_general2}, the top row is an example in which \(n=3\), and the bottom row depicts the general case.
As in Figure \ref{fig:half_twist_surgery_one_arc2}, the second and third pictures of each row in  Figure \ref{fig:half_twist_general2} are homotopic pictures of \(H_{l_i^j}(\lambda)\), and we can ``unsmooth" \(H_{l_i^j}(\lambda)\) into the fourth picture, which consists of a link diagram on \(F\).
This time, the link diagram contains \(n+1\) components: \(\lambda\) and \(n\) parallel copies of \(l_i^j\).
Thus our generalized procedure for drawing a half twist of \(\lambda\) about \(l_i^j\) is the following: 
First determine the number \(n\) of times that \(\lambda\) passes between the two punctures of \(U_i^j\), then we draw \(n\) parallel, disjoint copies of \(l_i^j\) so that at each crossing, \(l_i^j\) passes \textit{under} \(\lambda\) for a positive half twist, or \textit{over} \(\lambda\) for a negative half twist, and then we smooth the crossings using the rule from Figure \ref{fig:surgery_smoothing_crossing2}.

Note that to perform \(|t_i^j|\) consecutive half twists around \(l_i^j\), we simply perform this process \(|t_i^j|\) times, taking care that every time we add new parallel copies of \(l_i^j\), they are nested inside the ones previously drawn.
Thus if \(\lambda\) is half-twisted around \(l_i^j\) a total of \(|t_i^j|\) times, and \(\lambda\) passes between the punctures of \(U_i^j\) a total of \(n\) times, then to build the link diagram which depicts \(H_{l_i^j}^{t_i^j}(\lambda)\), we will add a total of \(n|t_i^j|\) parallel copies of \(l_i^j\) (all passing under \(\lambda\) if \(t_i^j>0\) or over \(\lambda\) if \(t_i^j<0\)).
Thus after any sequence of half twists about \(l\)-loops, we obtain a link diagram representing \(\lambda\).
Instead of drawing all of the parallel copies of each \(l\)-loop, we will draw only one of each and label it with the number {\setlength{\fboxsep}{2pt}\fbox{\(n\)}} of strands it represents.

\begin{figure}[h!]
\centering
\labellist \small\hair 2pt
\pinlabel {
	\setlength{\fboxsep}{2pt}
	\fbox{7}
	} at 20 60
\pinlabel {
	\setlength{\fboxsep}{2pt}
	\fbox{3}
	} at 20 30
\pinlabel {
	\setlength{\fboxsep}{2pt}
	\fbox{7}
	} at 20 142
\pinlabel {
	\setlength{\fboxsep}{2pt}
	\fbox{3}
	} at 20 172
\pinlabel {=} at 101 44
\pinlabel {=} at 101 156
\pinlabel {=} at 213 44
\pinlabel {=} at 213 156
\pinlabel {=} at 325 44
\pinlabel {=} at 325 156
\endlabellist
\includegraphics[width=\textwidth]{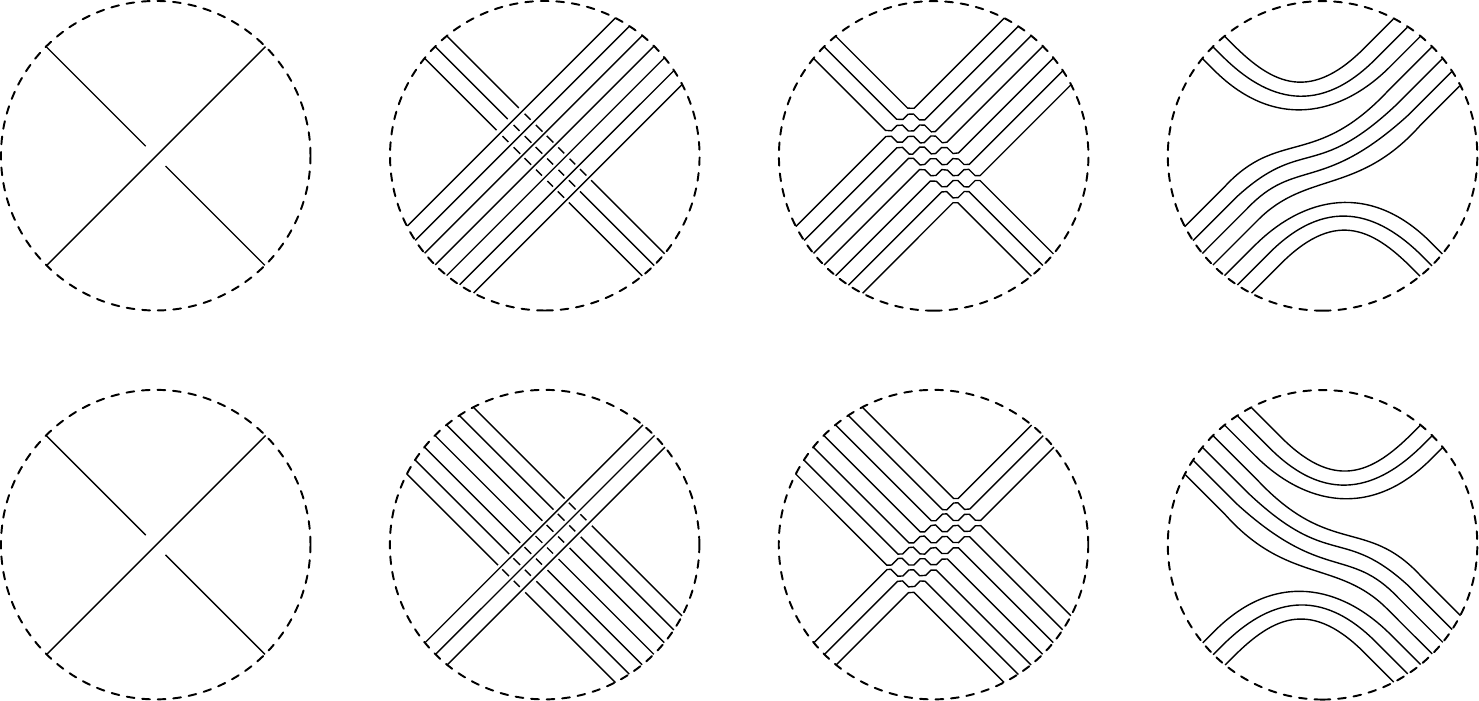}
\caption{After twisting a loop \(\lambda\) about \(l\)-loops, we obtain a link diagram on \(F\).
Here is an example of a crossing in such a diagram: In this crossing, one strand represents 3 parallel strands, and the other represents 7 parallel strands.
Consider the top row of pictures, where the over-strand has a larger number.
The first picture is a neighborhood of the crossing of the link diagram we have obtained.
The second picture explicitly shows all of the parallel strands at the crossing.
The third picture shows the result after smoothing all 21 crossings.
The fourth picture is a smoothed version of the third.
All four pictures represent the same thing.
The bottom row is similar, but this time the under-strand has the larger number.}
\label{fig:multi_strand_crossing}
\end{figure}
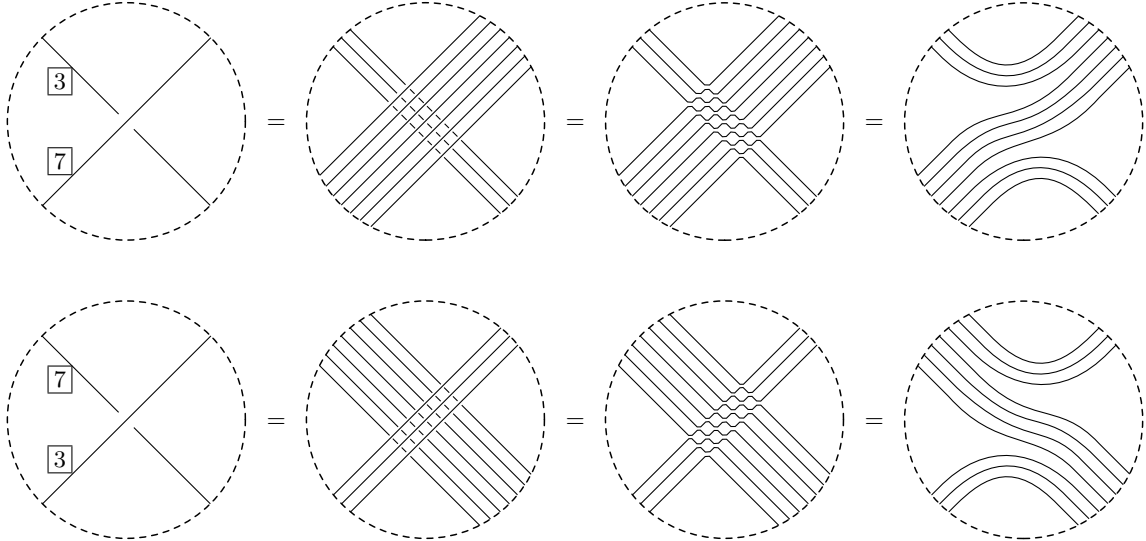

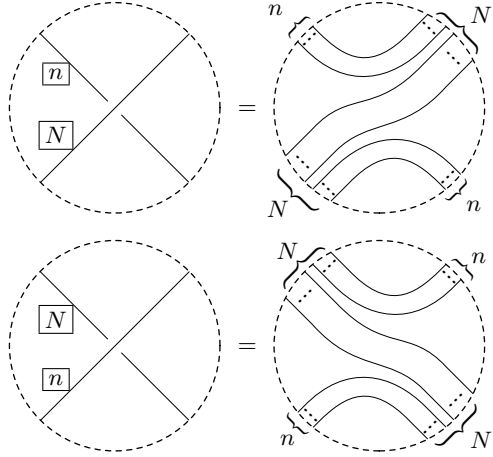
\begin{figure}[h!]
\centering
\labellist \small\hair 2pt
\pinlabel {
	\setlength{\fboxsep}{2pt}
	\fbox{\(N\)}
	} at 20 56
\pinlabel {
	\setlength{\fboxsep}{2pt}
	\fbox{\(n\)}
	} at 20 30.5
\pinlabel {
	\setlength{\fboxsep}{2pt}
	\fbox{\(N\)}
	} at 20 134
\pinlabel {
	\setlength{\fboxsep}{2pt}
	\fbox{\(n\)}
	} at 20 160
\pinlabel {=} at 101 44
\pinlabel {=} at 101 145
\pinlabel {
\begin{rotate}{45}
\(
\left\{
\begin{tabular}{c}
• \\ 
• \\ 
\end{tabular}
\right.
\)
\end{rotate}
} at 122 105
\pinlabel {\(N\)} at 114 103
\pinlabel {
\begin{rotate}{225}
\(
\left\{
\begin{tabular}{c}
• \\ 
• \\ 
\end{tabular}
\right.
\)
\end{rotate}
} at 194 183
\pinlabel {\(N\)} at 200 185
\pinlabel {
\begin{rotate}{135}
\(
\left\{
\begin{tabular}{c}
• \\ 
• \\ 
\end{tabular}
\right.
\)
\end{rotate}
} at 197 9
\pinlabel {\(N\)} at 200 5
\pinlabel {
\begin{rotate}{-45}
\(
\left\{
\begin{tabular}{c}
• \\ 
• \\ 
\end{tabular}
\right.
\)
\end{rotate}
} at 121 80
\pinlabel {\(N\)} at 118 85
\pinlabel {
\begin{rotate}{-45}
\(
\left\{
\begin{tabular}{c}
• \\ 
\end{tabular}
\right.
\)
\end{rotate}
} at 120 181
\pinlabel {\(n\)} at 113 187
\pinlabel {
\begin{rotate}{135}
\(
\left\{
\begin{tabular}{c}
• \\ 
\end{tabular}
\right.
\)
\end{rotate}
} at 196 110
\pinlabel {\(n\)} at 197 104
\pinlabel {
\begin{rotate}{225}
\(
\left\{
\begin{tabular}{c}
• \\ 
\end{tabular}
\right.
\)
\end{rotate}
} at 195 82
\pinlabel {\(n\)} at 199 82
\pinlabel {
\begin{rotate}{45}
\(
\left\{
\begin{tabular}{c}
• \\ 
\end{tabular}
\right.
\)
\end{rotate}
} at 123 7
\pinlabel {\(n\)} at 118 5
\endlabellist
\includegraphics[scale=.9]{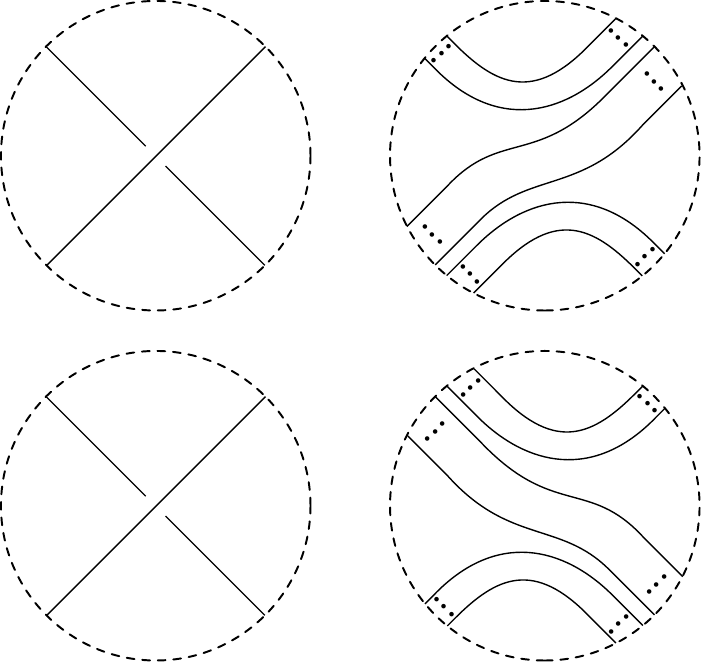}
\caption{Here we depict how to interpret a crossing in the link diagram of \(N\) strands with \(n\) strands, where \(N>n\).}
\label{fig:multi_strand_crossing_general}
\end{figure}

In the resulting link diagram, we will have crossings involving two strands which are labeled with different numbers.
Figure \ref{fig:multi_strand_crossing} depicts an example in which we show how to recover what \(\lambda\) actually looks like near such a crossing, where a strand marked with a {\setlength{\fboxsep}{2pt}\fbox{7}} crosses a strand marked with a {\setlength{\fboxsep}{2pt}\fbox{3}}.
Figure \ref{fig:multi_strand_crossing_general} shows the general case.

\begin{figure}[!ht]
\centering
\labellist \small\hair 2pt
\pinlabel {Level 1} at -100 1545                
\pinlabel {\(\partial B'\)} [l] at 1220 1545
\pinlabel {
	\setlength{\fboxsep}{2pt}
	\fbox{1}
	} at 1125 1646
\pinlabel {Level 2} at -100 1257              
\pinlabel {\(\tau_2(\partial B')\)} [l] at 1220 1257
\pinlabel {
	\setlength{\fboxsep}{2pt}
	\fbox{1}
	} at 1125 1359                     
\pinlabel {
	\setlength{\fboxsep}{2pt}
	\fbox{\(t_2^3\)}
	} at 925 1385
\pinlabel {Level 3} at -100 969                
\pinlabel {\(\tau_3(\partial B')\)} [l] at 1220 969
\pinlabel {
	\setlength{\fboxsep}{2pt}
	\fbox{1}
	} at 1125 1096            
\pinlabel {
	\setlength{\fboxsep}{2pt}
	\fbox{\(t_2^3\)}
	} at 915 1098
\pinlabel {
	\setlength{\fboxsep}{2pt}
	\fbox{\(t_2^3|t_3^4|\)}
	} at 1083 1018
\pinlabel {
	\setlength{\fboxsep}{2pt}
	\fbox{\(t_2^3|t_3^3|\)}
	} at 710 1102
\pinlabel {Level 3} at -100 681                
\pinlabel {(simplified)} at -100 631 
\pinlabel {\(\tau_3(\partial B')\)} [l] at 1220 681
\pinlabel {
	\setlength{\fboxsep}{2pt}
	\fbox{\(t_2^3\)}
	} at 915 812            
\pinlabel {
	\setlength{\fboxsep}{2pt}
	\fbox{\(1+t_2^3|t_3^4|\)}
	} at 1173 794
\pinlabel {
	\setlength{\fboxsep}{2pt}
	\fbox{\(t_2^3|t_3^3|\)}
	} at 710 813
\pinlabel {Level 4} at -100 393                
\pinlabel {\(\tau_4(\partial B')\)} [l] at 1220 393
\pinlabel {
	\setlength{\fboxsep}{2pt}
	\fbox{\(t_2^3\)}
	} at 941 535            
\pinlabel {
	\setlength{\fboxsep}{2pt}
	\fbox{\(t_4^3(1+t_2^3|t_3^3|+t_2^3|t_3^4|)\)}
	} at 346 556            
\pinlabel {
	\setlength{\fboxsep}{2pt}
	\fbox{\(1+t_2^3|t_3^4|\)}
	} at 1170 506
\pinlabel {
	\setlength{\fboxsep}{2pt}
	\fbox{\(t_2^3|t_3^3|\)}
	} at 710 525
\pinlabel {
	\setlength{\fboxsep}{2pt}
	\fbox{\(t_2^3|t_3^3|t_4^2\)}
	} at 560 525
\pinlabel {Level 4} at -100 105                
\pinlabel {(simplified)} at -100 55
\pinlabel {\(\tau_4(\partial B')\)} [l] at 1220 105
\pinlabel {
	\setlength{\fboxsep}{2pt}
	\fbox{\(t_2^3+t_4^3(1+t_2^3|t_3^3|+t_2^3|t_3^4|)\)}
	} at 282 266          
\pinlabel {
	\setlength{\fboxsep}{2pt}
	\fbox{\(1+t_2^3|t_3^4|\)}
	} at 1173 215
\pinlabel {
	\setlength{\fboxsep}{2pt}
	\fbox{\(t_2^3|t_3^3|\)}
	} at 710 237
\pinlabel {
	\setlength{\fboxsep}{2pt}
	\fbox{\(t_2^3|t_3^3|t_4^2\)}
	} at 560 237
\endlabellist
\includegraphics[scale=.29]{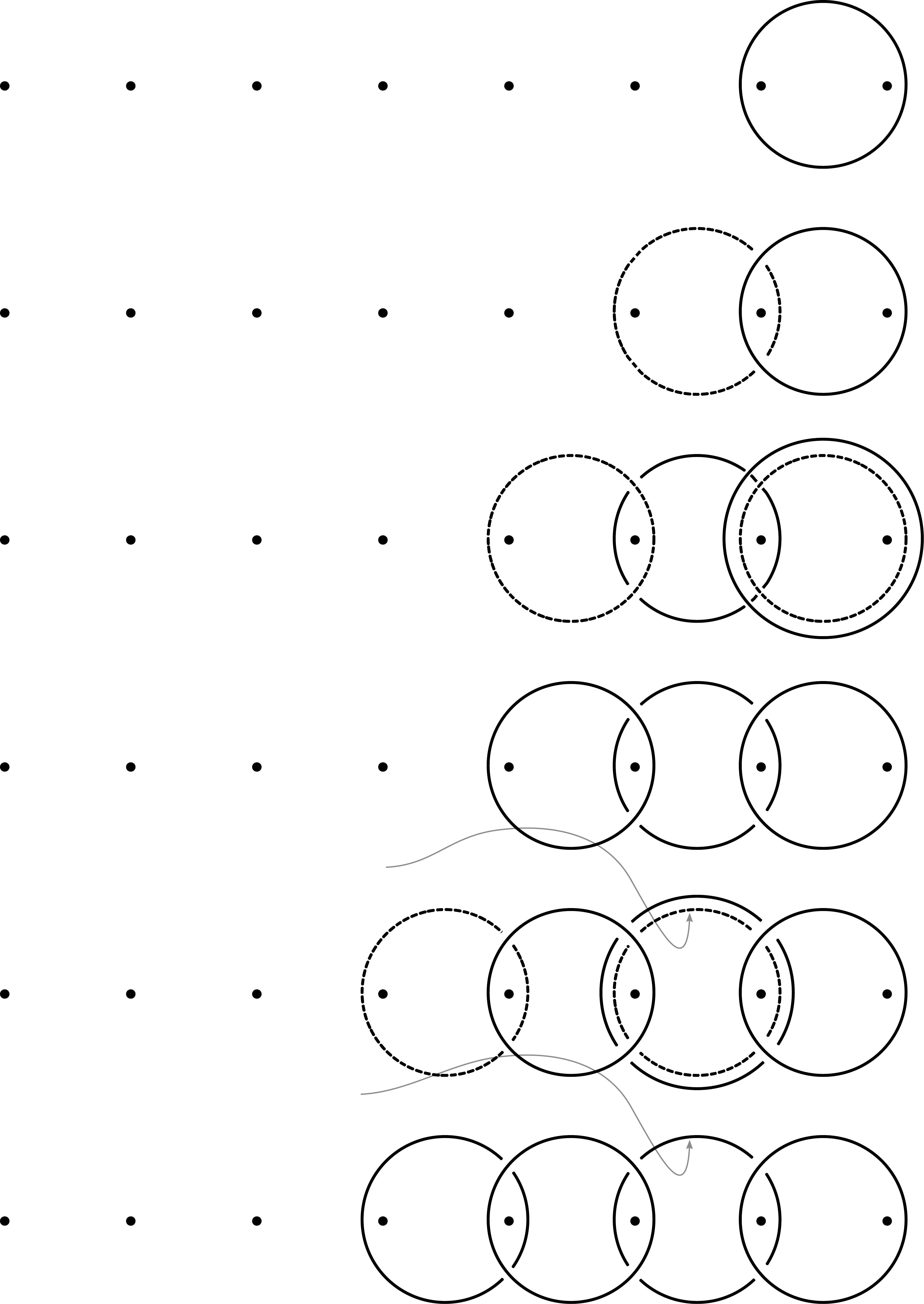}
\caption{The components of the link diagram that are newly added as a result of having just moved up a level are highlighted by being drawn with dotted lines.
The reader may find it helpful to compare this Figure to Figure \protect\ref{fig:B_and_Bprime_disjoint}, which shows the link diagram's location with respect to \(L\).}
\label{fig:obtaining_4_circles_with_crossing_info}
\end{figure}

\begin{figure}[h!]
\centering
\labellist \small\hair 2pt
\pinlabel {
	\setlength{\fboxsep}{2pt}
	\fbox{\(N^1\)}
	} at 520 330 
\pinlabel {
	\setlength{\fboxsep}{2pt}
	\fbox{\(N^2\)}
	} at 680 330 
\pinlabel {
	\setlength{\fboxsep}{2pt}
	\fbox{\(N^3\)}
	} at 935 330 
\pinlabel {
	\setlength{\fboxsep}{2pt}
	\fbox{\(N^4\)}
	} at 1095 330 
\pinlabel {\(G^\text{b}\)} at 648 90 
\pinlabel {\(G^\text{o}\)} at 809 333 
\pinlabel {\(G^\text{p}\)} at 969 90 
\pinlabel {\(\Lab\)} at 870 400
\pinlabel {\(\beta^1\)} at 90 220
\pinlabel {\(\gamma^1\)} at 251 220
\pinlabel {\(\beta^2\)} at 411 220
\pinlabel {\(\gamma^2\)} at 573 220
\pinlabel {\(\beta^3\)} at 733 220
\pinlabel {\(\gamma^3\)} at 893 220
\pinlabel {\(\beta^4\)} at 1054 220
\endlabellist
\includegraphics[width=\textwidth]{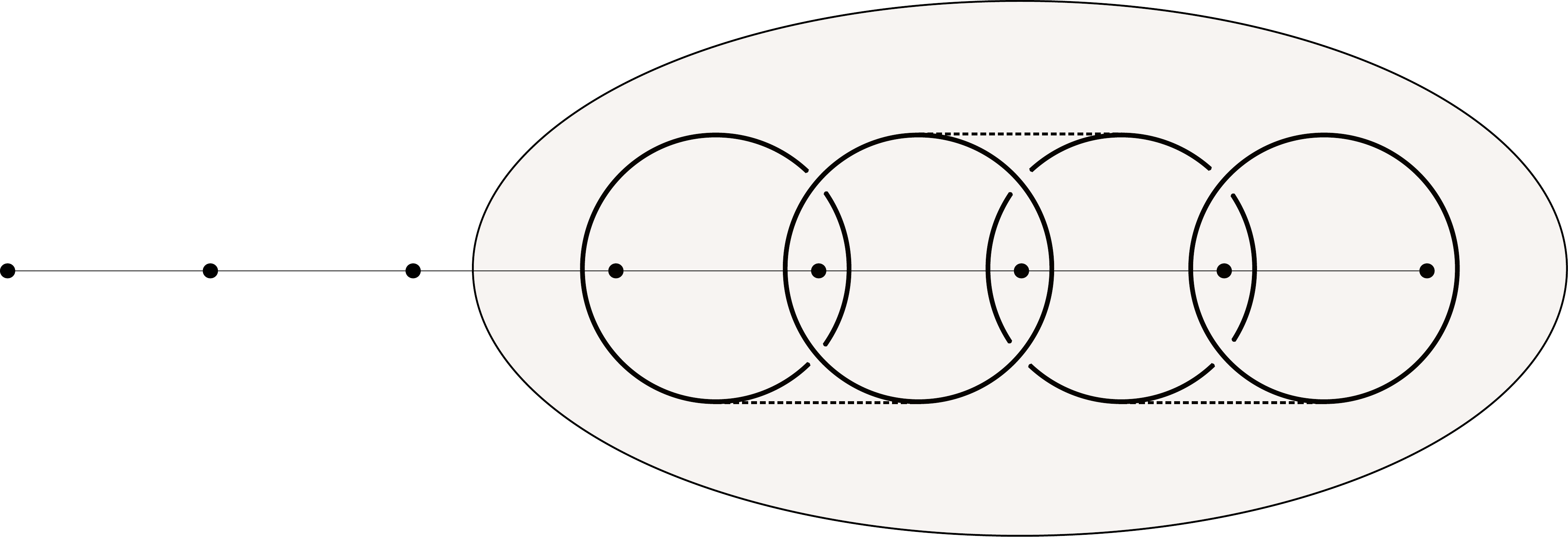}
\vspace{20pt}
\caption{This is the link diagram representing \(\tau_4(\partial B')\).
\(N^1=t_2^3|t_3^3|t_4^2\),  \(N^2=t_2^3|t_3^3|\),  \(N^3={t_2^3+t_4^3(1+t_2^3|t_3^3|+t_2^3|t_3^4|)}\), and  \(N^4=1+t_2^3|t_3^4|\).
The following inequalities hold: \(N^1>N^2\), \(N^2<N^3\), and \(N^3>N^4\).}
\label{fig:4_circles_with_crossing_info}
\end{figure}

Our purpose in developing this link diagramatic representation of a twisted loop is to understand what \(\partial B'\) looks like after performing (in order) half twists about all of the \(l\)-loops for any set of twist numbers \(\left\{t_i^j\right\}\).
With this method in hand, this is a straightforward task.
Refer to Figure \ref{fig:obtaining_4_circles_with_crossing_info} as we describe this process.
We start with the simple circle \(\partial B'\) in \(F\) at Level 1, and then push it up to Level 2 (i.e., we apply \(\tau_2\)).
\(\partial B'\) had to pass through the twist region \(\Tw_2^3\), so this means it underwent \(t_2^3\) half twists about \(l_2^3\).
Before twisting (i.e., at Level 1), \(\partial B'\) passed one time (i.e., \(n=1\)) between the punctures of \(U_2^3\); therefore what \(\partial B'\) looks like at Level 2 is a link diagram consisting of the original circle \(\partial B'\), plus \(1\times t_2^3\) parallel copies of \(l_i^j\).
Since we specified in Section \ref{sec:setting} that \(t_2^j\) is positive for all \(j\), each copy of \(l_i^j\) will be drawn to pass under \(\partial B'\).
Now the original \(\partial B'\) and the \(t_2^3\) copies of \(l_i^j\) together with the crossing information make a link diagram which represents the image of \(\partial B'\) under \(\tau_2\).
Since the value of \(t_2^3\) is arbitrary, we cannot smooth the crossings to see exactly what \(\tau_2(\partial B')\) looks like.
But that is not a problem;
we can still push this loop up from Level 2 to Level 3.
As we do, we add copies of \(l_3^3\) and \(l_3^4\) to the link diagram.
This time, since \(t_3^j<0\) for all \(j\), these \(l\)-loops will all be drawn with overcrossings, which gives us a link diagram representing \(\tau_3(\partial B')\).
Finally we push this up to Level 4, adding copies of \(l_4^2\) and \(l_4^3\) (with undercrossings) to the link diagram.
In this way we see that \(\tau_4(\partial B')\) can be represented by the diagram of the four-component unlink in Figure \ref{fig:4_circles_with_crossing_info}.

Each circle in the diagram in Figure \ref{fig:4_circles_with_crossing_info} is marked with a number \(N^i\) indicating how many parallel copies of that circle are present, where \(N^1=t_2^3|t_3^3|t_4^2\),  \(N^2=t_2^3|t_3^3|\),  \(N^3=t_2^3+t_4^3(1+t_2^3|t_3^3|+t_2^3|t_3^4|)\), and  \(N^4=1+t_2^3|t_3^4|\).
Recall that for all \(j\), \(t_2^j,t_4^j\geq 2\), and \(t_3^j\leq -2\).
No matter what values the twist numbers \(\left\{t_i^j\right\}\) take on, provided they follow this rule, it can be shown that \(N^i>0\) for each \(i\), and these three inequalities hold true: \(N^1>N^2\), \(N^2<N^3\), and \(N^3>N^4\).
This means that at each of the six crossings, the strand labeled with the higher number passes under the strand labeled with the lower number, so all six crossings look like the top picture in Figure \ref{fig:multi_strand_crossing_general}.

\begin{prop}\label{prop:BprimeandDminpos}
The position of \(\partial B'\) described by the link diagram in Figure \ref{fig:4_circles_with_crossing_info} is minimal with respect to \(\beta^2\cup\beta^3\cup\beta^4\).
\end{prop}

\begin{proof}
The Bigon Criterion in \cite{bigoncrit} tells us that as long as \(\partial B'\) cobounds no bigons with \(\beta^2\cup\beta^3\cup\beta^4\), then they are in minimal position.
Consider \(\partial B'\) as being cut into component strands by \(\beta^2\cup\beta^3\cup\beta^4\).
To bound a bigon, one of these strands would have to have both endpoints on a single \(\beta\)-arc. However, each strand has endpoints on distinct \(\beta\)-arcs.
\end{proof}

Let Lab be the disk in \(F_4\) which contains \(\tau_4(\partial B')\) and whose complement is a regular neighborhood of \(\beta^1\cup\gamma^1\).
(Lab is depicted in Figures \ref{fig:boundary_of_Bprime_two_twists_revised_3} and \ref{fig:4_circles_with_crossing_info}, and also, Lab is isotopic in \(F\) to the top grey rectangle in Figure \ref{fig:B_and_Bprime_disjoint}.)
We will refer to Lab as the \textbf{Labyrinth}.

Define \(G^\text{b},G^\text{o}\), and \(G^\text{p}\) to be the dotted arcs depicted in Figures \ref{fig:boundary_of_Bprime_two_twists_revised_3} and \ref{fig:4_circles_with_crossing_info}.
We will call these the \textbf{brown, orange}, and \textbf{purple gates} of the labyrinth, respectively.
Each gate is an arc in \(\Lab\) with endpoints on \(\partial B'\).

\begin{figure}[h!]
\centering
\includegraphics[scale=.4]{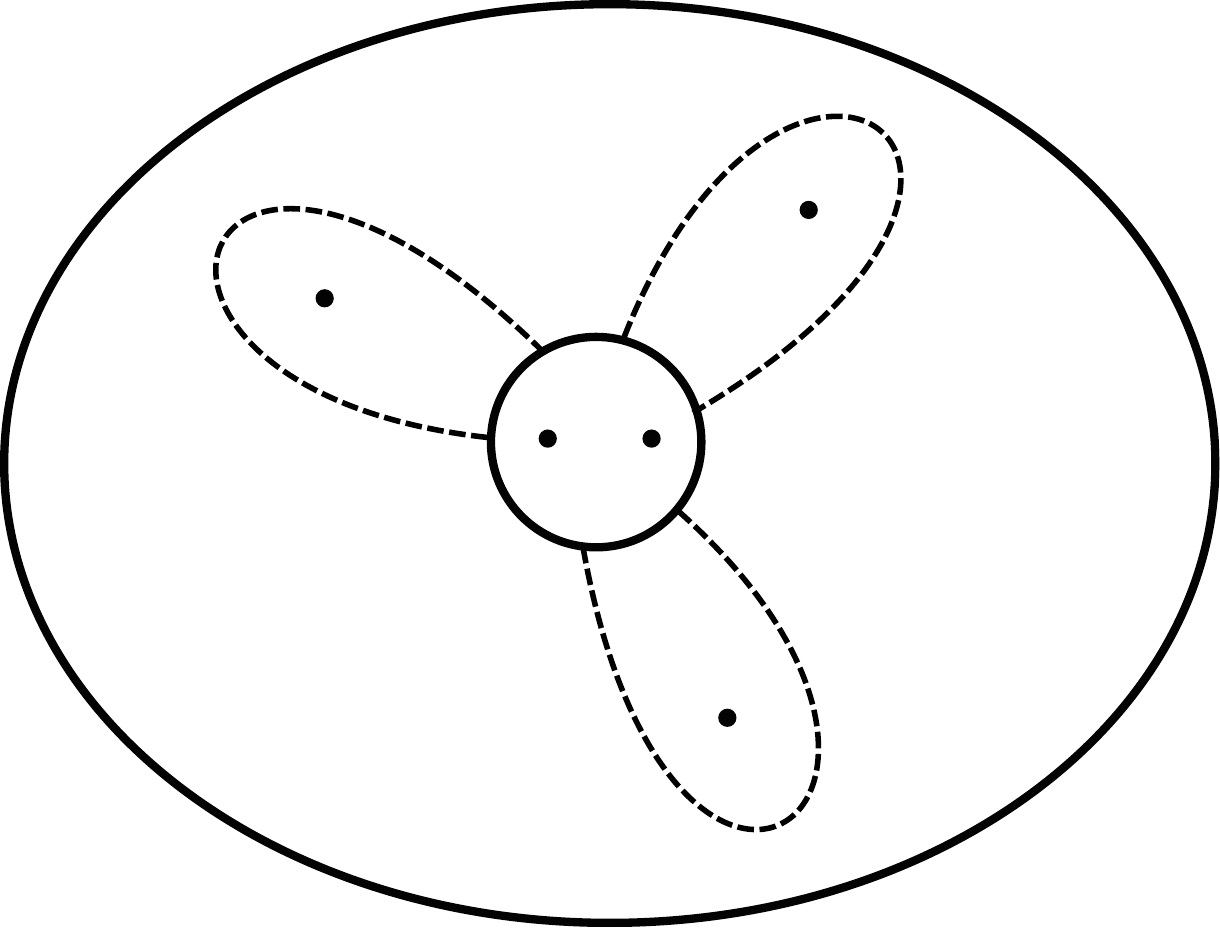}
\caption{
Proposition \ref{prop:lab_structure} asserts that Lab is isotopic to this picture.
\(\partial B'\) and the three gates cut \(\Lab\) into five components: three once-punctured disks, an annulus, and a twice-punctured disk.
}
\label{fig:lab_structure}
\end{figure}

\begin{prop}\label{prop:lab_structure}
\(\partial B'\) and the three gates cut \(\Lab\) into five components: three once-punctured disks, an annulus, and a twice-punctured disk.
\end{prop}

\begin{proof}
Since each crossing in our link diagram in Figure \ref{fig:4_circles_with_crossing_info} looks like the top picture of Figure \ref{fig:multi_strand_crossing_general}, it is apparent that there is an annulus component, which we will call \(A\).
\(A\)  has \(\partial(\Lab)\) as one boundary component, and the other component of \(\partial A\) alternates between the three gates and three subarcs of \(\partial B'\).
Observe that \(A\) is not punctured.
\(\partial B'\) cuts off the twice-punctured disk \(U_2^3\) from \(\Lab\) by definition, and clearly none of the gates are in this disk.
\(\Lab\hspace{-2pt}\backslash U_2^3\) is a thrice-punctured annulus in which the three gates are properly embedded but not nested;
If any subset of the gates were nested, then at least one of them would not be a boundary component of \(A\). 
None of the gates can be parallel to \(\partial B'\) because that would either force \(A\) to be punctured or force the other gates to be nested.
Thus the only possible configuration is (isotopic to) Figure \ref{fig:lab_structure}.
\end{proof}

We will refer to the three once-punctured disks in Proposition \ref{prop:lab_structure} as the \textbf{brown}, \textbf{orange}, and \textbf{purple punctured disks}, according to the color of the corresponding gate, and we will call the marked point contained therein a \textbf{brown}, \textbf{orange}, or \textbf{purple marked point}.
In the brown disk, there is a unique arc (up to isotopy) connecting the brown point to \(G^\text{b}\), which we will call the \textbf{brown escape route}.
We define the \textbf{orange} and \textbf{purple escape routes} similarly.
(The three escape routes are depicted in the second picture of Figure \ref{fig:escape_route_intersects_b}.)
We can think of the colored points as escaping from a maze whose walls are (\(\partial B'\)), and the escape routes are the paths they take to get to the exits (the gates).

\section{Red disks enter the labyrinth}\label{sec:Cintersectsgamma}

Refer to Figure \ref{fig:bridges}.

\begin{lem}\label{lem:Cintersectsgamma}
 If \(R\) is a red disk above \(F\), then \(R\) must intersect \(\gamma^2\) or \(\gamma^3\).
\end{lem}

\begin{proof}
\(F\) cuts \(S^3\) into two 3-balls: \(M_+\) above \(F\) and \(M_-\) below \(F\).
\(R\) divides \(M_+\) into two 3-balls which we will call \(M_+^1\) and \(M_+^2\) and refer to as the two sides of \(R\).

\begin{figure}[h!]
\centering
\labellist \small\hair 2pt
\pinlabel {\(F\)} at 30 10
\pinlabel {\(R\)} [b] at 200 130
\pinlabel {\(\alpha^k\)} at 320 45
\pinlabel {\(\alpha^{k+1}\)} at 422 32
\endlabellist
\includegraphics[scale=.5]{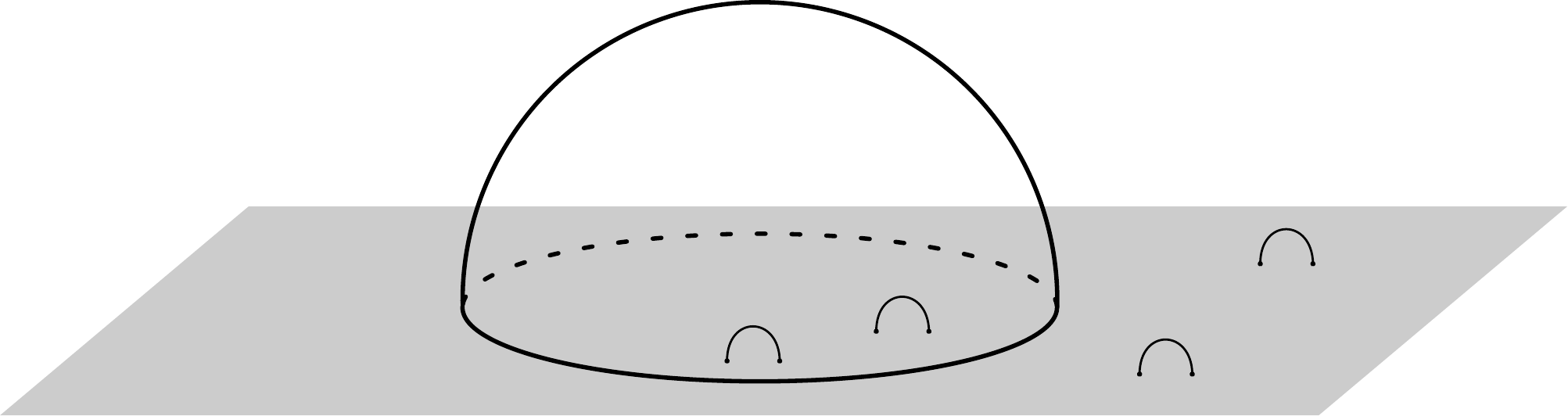}
\caption{In Case 1, \(\alpha^2,\alpha^3\), and \(\alpha^4\) appear on both sides of \(R\).}
\label{fig:alpha2_separated_from_alpha3}
\end{figure}
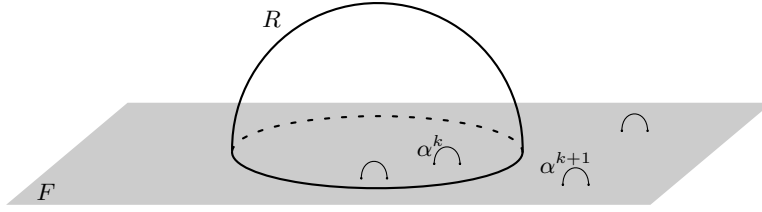

\textbf{Case 1:} \(\alpha^2,\alpha^3\), and \(\alpha^4\) are not all on the same side of \(R\).
Then for \(k=2\) or \(k=3\), \(\alpha^k\) and \(\alpha^{k+1}\) are on opposite sides of \(R\).
Since \(\gamma^k\) connects an endpoint of \(\alpha^k\) and an endpoint of \(\alpha^{k+1}\), \(\gamma^k\) must intersect \(R\).

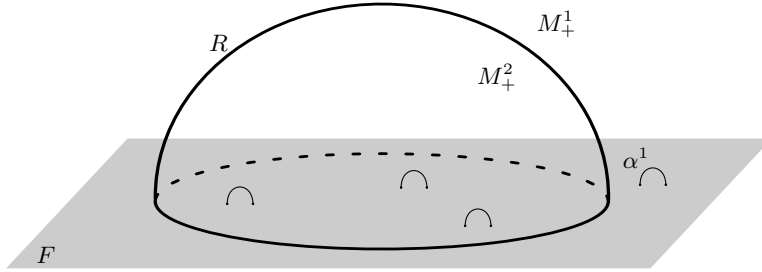
\begin{figure}[h!]
\centering
\labellist \small\hair 2pt
\pinlabel {\(F\)} at 30 10
\pinlabel {\(R\)} [b] at 160 160
\pinlabel {\(\alpha^1\)} at 475 83
\pinlabel {\(M_+^2\)} [b] at 370 130
\pinlabel {\(M_+^1\)} [b] at 415 170
\endlabellist
\includegraphics[scale=.5]{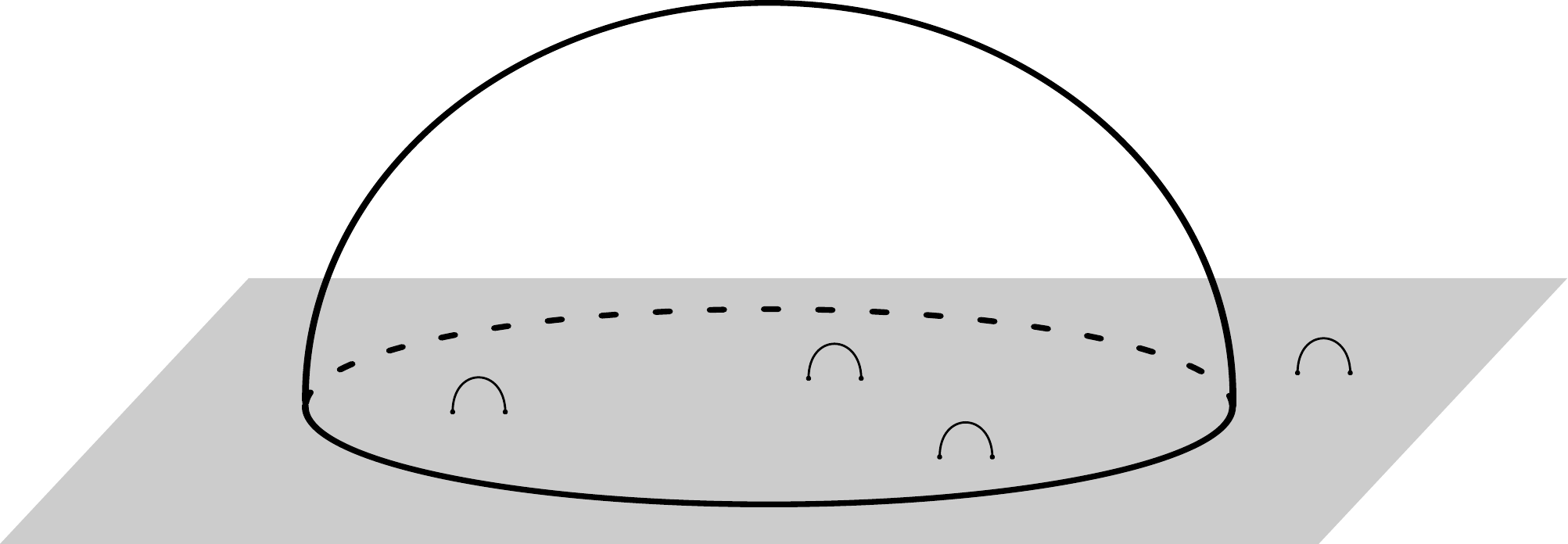}
\caption{In Case 2, \(\alpha^1\subset M_+^1\) and \(\alpha^2,\alpha^3,\alpha^4\subset M_+^2\).}
\label{fig:R_a_cap_for_alpha1}
\end{figure}

\textbf{Case 2:} All of the bridge arcs \(\alpha^2,\alpha^3\), and \(\alpha^4\) are on the same side of \(R\).
Without loss of generality, say they are in \(M_+^2\).
See Figure \ref{fig:R_a_cap_for_alpha1}.
Notice that \(\alpha^1\) cannot also be in \(M_+^2\) because that would imply \(\partial R\) is null-homotopic in \(F\), contradicting the fact that \(R\) is a compressing disk.
Therefore, in Case 2, \(R\) separates \(\alpha^1\) from the other bridge arcs, which implies \(R\) is a cap for \(\alpha^1\).
Let \(\mathfrak{D}=D^2\cup D^3\cup D^4\), and isotope \(R\) to minimize \(\#|R\cap \mathfrak{D}|\).

\textbf{Subcase 2.1:} \(R\) is disjoint from \(\mathfrak{D}\).

Assume (for contradiction) that \(R\) is also disjoint from  \(\gamma^2\) and \(\gamma^3\).
Then we have a single straight arc \({\Gamma=\beta^2\cup\gamma^2\cup\beta^3\cup\gamma^3\cup\beta^4}\) disjoint from \(\partial R\) because \(\beta^j\subset D^j\subset \mathfrak{D}\).
So \(\partial R\) cuts \(F\) into two disks, one of which contains \(\Gamma\), and the other of which contains \(\gamma^1\).
But of course, there is only one such loop in \(F\), which is \(\partial B\).
Thus \(R\simeq B\), so the red disk \(R\) is blue, a contradiction.
We conclude that in this subcase, \(\partial R\) must intersect \(\gamma^2\) or \(\gamma^3\).

\textbf{Subcase 2.2:} \(R\) is not disjoint from \(\mathfrak{D}\).

\begin{figure}[h!]
\centering
\labellist \small\hair 2pt
\pinlabel {\(F\)} at 30 10
\pinlabel {\(\alpha^k\)} [b] at 250 160
\pinlabel {\(D^k\)} [b] at 290 130
\pinlabel {\(\tilde{\alpha}\)} [b] at 272 60
\pinlabel {\(\tilde{D}\)} [b] at 287 40
\pinlabel {\(\gamma^{k-1}\)} [b] at 140 43
\pinlabel {\(\gamma^{k}\)} [b] at 440 43
\pinlabel {\(\tilde{\beta}\)} at 287 17
\endlabellist
\includegraphics[scale=.5]{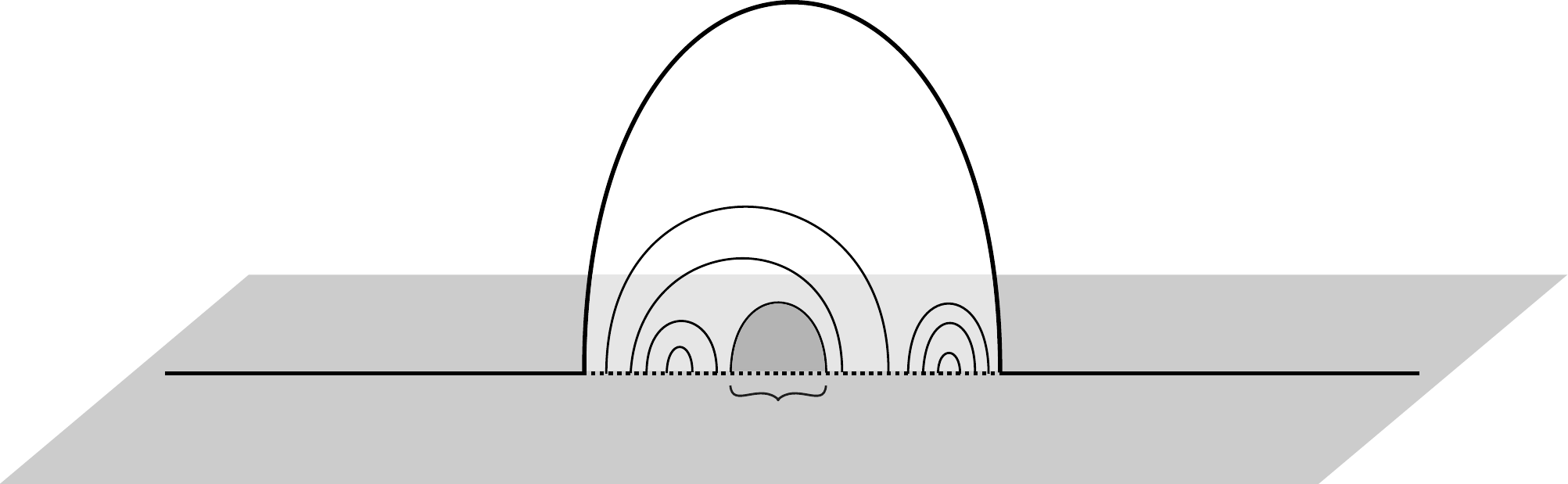}
\caption{In Subcase 2.2, \(R\cap D^k\) is nonempty.  Possible arcs of intersection are depicted on \(D^k\).}
\label{fig:C_intersects_Bk_unlabeled}
\end{figure}
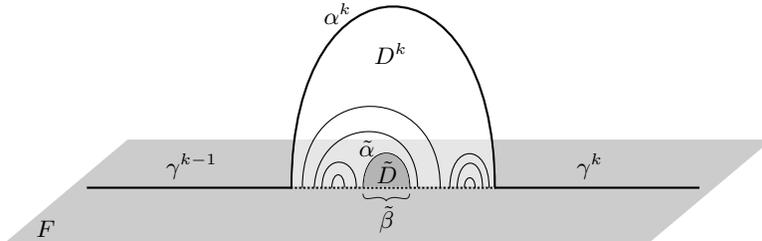

For some \(D^k\subset\mathfrak{D}\), \(R\cap D^k\) is nonempty.
\(R\cap D^k\) cannot contain loop intersections because \(\#|R\cap \mathfrak{D}|\) is minimal, and \(S^3\backslash\eta(L)\) is irreducible.
\(R\cap D^k\) cannot contain an arc with either endpoint on \(\alpha^k\) because that would imply that \(R\) intersects \(L\).
Thus the components of intersection in \(D^k\) must all be arcs with both endpoints on \(\beta^k\).
See Figure \ref{fig:C_intersects_Bk_unlabeled}.

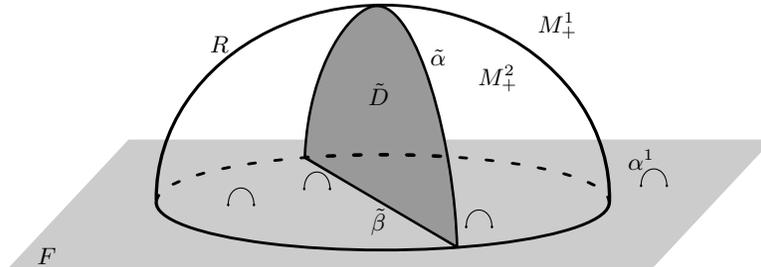
\begin{figure}[h!]
\centering
\labellist \small\hair 2pt
\pinlabel {\(F\)} at 30 10
\pinlabel {\(R\)} [b] at 160 160
\pinlabel {\(\alpha^1\)} at 478 83
\pinlabel {\(M_+^2\)} [b] at 370 130
\pinlabel {\(M_+^1\)} [b] at 415 170
\pinlabel {\(\tilde{D}\)} [b] at 280 120
\pinlabel {\(\tilde{\alpha}\)} [b] at 325 150
\pinlabel {\(\tilde{\beta}\)} at 280 36
\endlabellist
\includegraphics[scale=.5]{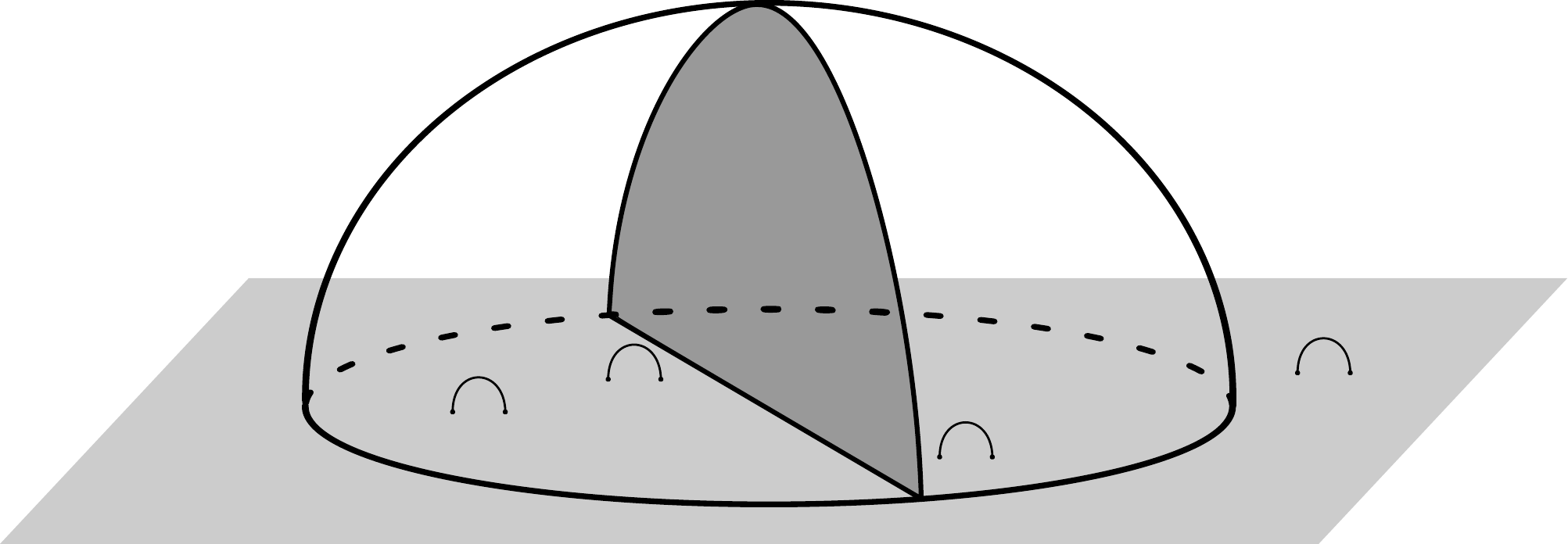}
\caption{In Subcase 2.2, \(R\cap D^k\) is nonempty.  Here we see \(R\) and the innermost disk \(\tilde{D}\).}
\label{fig:C_with_D_prime_unlabeled}
\end{figure}

Take an  arc of intersection outermost on \(D^k\) and call it \(\tilde{\alpha}\).
Let \(\tilde{\beta}\) be the subarc of \(\beta^k\) that shares endpoints with \(\tilde{\alpha}\).
Then let \(\tilde{D}\) be the subdisk of \(D^k\) cobounded by \(\tilde{\alpha}\) and \(\tilde{\beta}\).
See Figures \ref{fig:C_intersects_Bk_unlabeled} and \ref{fig:C_with_D_prime_unlabeled}.
\(\tilde{D}\subset M_+^j\) for either \(j=1\) or \(j=2\).
Whichever is the case, \(\tilde{D}\) cuts \(M_+^j\) into two balls which we call the sides of \(\tilde{D}\).
Recall that we assume \(\alpha^2,\alpha^3,\alpha^4\subset M_+^2\), so therefore \(\alpha^1\subset M_+^1\).
\(\tilde{D}\) cannot be in \(M_+^1\) with \(\alpha^1\).
If it were, then since \(\alpha^1\) cannot intersect \(R\) or \(\tilde{D}\), \(\alpha^1\) would have to be completely contained on one side or the other of \(\tilde{D}\).
Then the other side of \(\tilde{D}\) would be an empty 3-ball through which we can isotope \(R\), removing at least one component of intersection with \(D^k\), which is a contradiction since \(\#|R\cap \mathfrak{D}|\) is already minimal.
Thus \(\tilde{D}\) is in \(M_+^2\) with \(\alpha^2,\alpha^3\), and \(\alpha^4\).

If \(\alpha^2,\alpha^3\), and \(\alpha^4\) were on the same side of \(\tilde{D}\), that would mean that the other side of \(\tilde{D}\) was an empty 3-ball through which we could isotope \(R\), removing the arc of intersection \(\tilde{\alpha}\subset R\cap\mathfrak{D}\), contradicting minimality.
Thus it must be the case that there are some \(\alpha\)-arcs in either side of \(\tilde{D}\).

Then for \(k=2\) or \(k=3\), \(\alpha^k\) and \(\alpha^{k+1}\) are on opposite sides of \(\tilde{D}\).
An endpoint of \(\alpha^k\) and an endpoint of \(\alpha^{k+1}\) are connected by the arc \(\gamma^k\) in \(F\).
None of the \(\gamma\)-arcs have interiors that intersect \(\mathfrak{D}\), so \(\gamma^k\) must pass through \(\partial R\) in order to connect \(\alpha^k\) to \(\alpha^{k+1}\).
Thus we have proved that in Subcase 2.2, \(\partial R\cap\left(\gamma^2\cup\gamma^3\right)\neq\emptyset\).

Now we have proved both subcases, concluding Case 2 and finishing the proof of Lemma \ref{lem:Cintersectsgamma}.
\end{proof}

\begin{cor}\label{cor:onegate}
Every red disk intersects at least one of \(\tau_4(\partial B')\), \(G^\text{b}\), \(G^\text{o}\), and \(G^\text{p}\).
\end{cor}

\begin{proof}
Recall that by Proposition \ref{prop:lab_structure}, \(\Lab\backslash\left(\tau_4(\partial B')\cup G^\text{b}\cup G^\text{o}\cup G^\text{p}\right)\) has five components:
three (colored) once-punctured disks,
one twice-punctured disk (\(U_2^3\)),
and an annulus.
See Figure \ref{fig:lab_structure}.
Let \(R\) be a red compressing disk above \(F\).
The boundary of a compressing disk cannot bound a once-punctured disk in \(F\).
Thus \(\partial R\) cannot be contained in one of the three colored once-punctured disks.
If \(\partial R\) lies in the twice-punctured disk, then \(\partial R\) must be isotopic to \(\tau_4(\partial B')\).
But then after the isotopy, \(R\cup B'\) would be a splitting sphere for \(L\), a contradiction since \(L\) is non-split.
We see in Figures \ref{fig:boundary_of_Bprime_two_twists_revised_3} and \ref{fig:4_circles_with_crossing_info} that \(\gamma^2\) and \(\gamma^3\) are disjoint from the annulus component of \(\Lab\), so if \(\partial R\) is contained in the annulus, \(\partial R\) would fail to intersect \(\gamma^2\) or \(\gamma^3\), contradicting Lemma \ref{lem:Cintersectsgamma}.
We conclude \(\partial R\) cannot lie in a single component of \(\Lab\backslash\left(\tau_4(\partial B')\cup G^\text{b}\cup G^\text{o}\cup G^\text{p}\right)\);
therefore it must intersect \(\tau_4(\partial B')\cup G^\text{b}\cup G^\text{o}\cup G^\text{p}\).
\end{proof}

%
%

\section{All Red Disks Above \(F\) Intersect All Blue Disks Below}\label{sec:proof}
%
%


Recall the disks and arcs defined in Section \ref{sec:setting} and pictured in Figures \ref{fig:bridges} and \ref{fig:alpha_4_prime}.
Let \(\mathfrak{D}=D^2\cup D^3\cup D^4\) as in Section \ref{sec:Cintersectsgamma}.
Let \(\mathcal{R}\) be the set of red compressing disks above \(F\) which are disjoint from \(\beta'\).
Throughout Section \ref{sec:proof}, whenever \(\mathcal{R}\) is nonempty, assume that \(R\) is a disk in \(\mathcal{R}\) such that \(|R\cap \mathfrak{D}|\leq |R'\cap \mathfrak{D}|\) for all \(R'\in\mathcal{R}\), and assume \(R\) is in minimal position with respect to \(\mathfrak{D}\) and with respect to the gates.
\(\beta'\) cuts \(\beta^i\) into multiple subarcs, which we call \textbf{lanes}.

\begin{lem}\label{lem:samelane}
Assume \(\mathcal{R}\) is nonempty.
If two points of \(R\cap\beta^i\) lie in the same lane, then they cannot be the endpoints of a common arc of \(R\cap D^i\).
\end{lem}

\begin{proof} Suppose there exists an arc of \(R\cap D^i\) whose endpoints lie in the same lane.
We may assume the arc is outermost on \(D^i\).
It cuts off a small subdisk \(\tilde{D^i}\subset D^i\) which doesn't intersect \(\beta'\).
Boundary-compress \(R\) along \(\tilde{D^i}\).
The result is two disks, \(R^1,R^2\) whose boundaries lie in \(F\), and which are disjoint from \(\beta'\).
Suppose \(R^1\) is trivial.
Then it represents an isotopy through which we can move \(R\) to decrease its intersection with the bridge disks.
But that implies \(R\) was not in minimal position with respect to the bridge disks, a contradiction.
Thus \(R^1\) is not trivial, and similarly, neither is \(R^2\).

This means both \(R^1\) and \(R^2\) are compressing disks.
By construction, neither of them intersects \(\beta'\), and for both \(i\), \(|R^i\cap \mathfrak{D}|<|R\cap \mathfrak{D}|\).
Thus by our choice of \(R\) as a minimal representative from \(\mathcal{R}\), \(R^i\not\in\mathcal{R}\).
Then by the definition of \(\mathcal{R}\), both \(R^i\) must be blue.
There is only one blue disk above \(F\), which is \(B\), so \(R^1\) and \(R^2\) must be parallel copies of \(B\).
This implies \(R\) is a band sum of parallel disks, but any band sum of parallel disks is trivial, so \(R\) is a trivial disk, which is a contradiction.
\end{proof}

By definition, \(R\in\mathcal{R}\) implies  \(R\cap\beta'=\emptyset\), so \(R\) can also be made disjoint from \(\partial B'\) since \(\partial B'\) is the boundary of a regular neighborhood of \(\beta'\).
It follows from Corollary \ref{cor:onegate}, that \(R\) must intersect at least one gate.
Figure \ref{fig:lab_structure} makes clear something not at all obvious in Figures \ref{fig:boundary_of_Bprime_two_twists_revised_3} and \ref{fig:4_circles_with_crossing_info}, which is that if an arc of \(R\) ``enters" the labyrinth through a particular gate, it must subsequently ``exit" the labyrinth from the same gate.
We will call the components of the intersection of \(\partial R\) and the brown disk \textbf{brown tracks}, and we will define \textbf{orange tracks} and \textbf{purple tracks} similarly.
Observe that these tracks are pairwise disjoint, and each track has endpoints on the gate of the corresponding color.

\begin{figure}[h!]
\labellist \small\hair 2pt
\pinlabel {\(\partial R\)} at -11 103
\pinlabel {brown track} at 58 89
\pinlabel {\(3\)} at 38 116
\pinlabel {\(2\)} at 54 116
\pinlabel {\(3\)} at 70 116
\pinlabel {\(\partial R\)} at -11 56
\pinlabel {orange track} at 58 40
\pinlabel {\(3\)} at 38 69
\pinlabel {\(4\)} at 54 69
\pinlabel {\(3\)} at 70 69
\pinlabel {\(4\)} at 86 69
\pinlabel {\(3\)} at 102 69
\pinlabel {\(2\)} at 118 69
\pinlabel {\(3\)} at 134 69
\pinlabel {\(2\)} at 150 69
\pinlabel {\(3\)} at 166 69
\pinlabel {\(4\)} at 182 69
\pinlabel {\(3\)} at 198 69
\pinlabel {\(4\)} at 214 69
\pinlabel {\(3\)} at 230 69
\pinlabel {\(\partial R\)} at -11 8
\pinlabel {purple track} at 58 -10
\pinlabel {\(4\)} at 38 21
\pinlabel {\(3\)} at 54 21
\pinlabel {\(4\)} at 70 21
\pinlabel {\(3\)} at 86 21
\pinlabel {\(4\)} at 102 21
\pinlabel {\(3\)} at 118 21
\pinlabel {\(4\)} at 134 21
\pinlabel {\(3\)} at 150 21
\pinlabel {\(4\)} at 166 21
\pinlabel {\(3\)} at 182 21
\pinlabel {\(4\)} at 198 21
\pinlabel {\(3\)} at 214 21
\pinlabel {\(4\)} at 230 21
\pinlabel {\(3\)} at 246 21
\pinlabel {\(2\)} at 262 21
\pinlabel {\(3\)} at 278 21
\pinlabel {\(2\)} at 294 21
\pinlabel {\(3\)} at 310 21
\pinlabel {\(4\)} at 326 21
\pinlabel {\(3\)} at 343 21
\pinlabel {\(4\)} at 359 21
\pinlabel {\(3\)} at 375 21
\pinlabel {\(2\)} at 391 21
\pinlabel {\(3\)} at 407 21
\pinlabel {\(2\)} at 423 21
\pinlabel {\(3\)} at 439 21
\pinlabel {\(4\)} at 455 21
\pinlabel {\(3\)} at 471 21
\pinlabel {\(4\)} at 487 21
\pinlabel {\(3\)} at 503 21
\pinlabel {\(4\)} at 519 21
\pinlabel {\(3\)} at 535 21
\pinlabel {\(4\)} at 551 21
\pinlabel {\(3\)} at 567 21
\pinlabel {\(4\)} at 583 21
\pinlabel {\(3\)} at 599 21
\pinlabel {\(4\)} at 615 21
\pinlabel {\(3\)} at 631 21
\pinlabel {\(4\)} at 647 21
\endlabellist
\hfill\includegraphics[width=.96\textwidth]{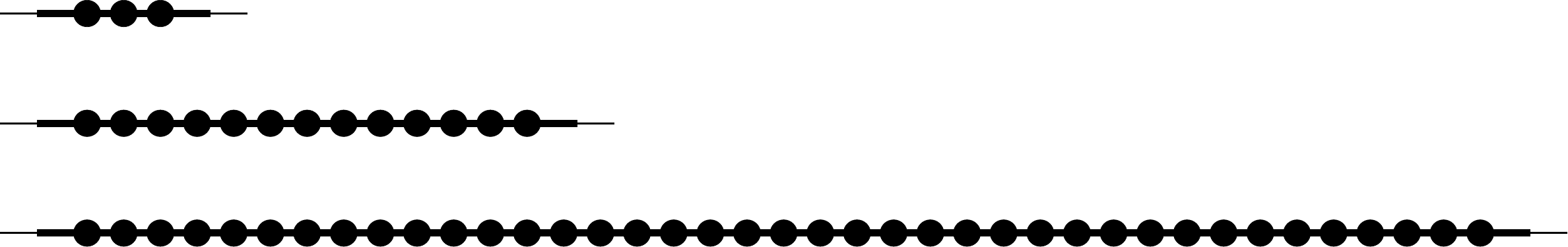}
\vspace{10pt}
\caption{Here we depict the brown, orange, and purple tracks in \(\partial R\) with their respective sequences of numbered points for the example depicted in Figure \protect\ref{fig:boundary_of_Bprime_two_twists_revised_3}, (in which \(t_1^j=t_3^j=2\) and \(t_2^j=-2\) for all \(j\)).
We will not prove this figure, nor does the paper depend on it.
The reader can check it by taking a pencil (and perhaps a magnifying glass) to Figure \protect\ref{fig:boundary_of_Bprime_two_twists_revised_3} and tracing out a path that winds in a gate, around a marked point inside the labyrinth, and then back out.}
\label{fig:tracks_with_colors_two_twist_case}
\end{figure}
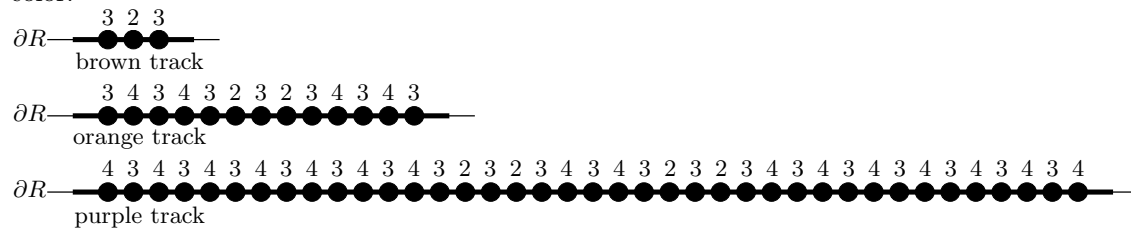

We label the points of \(\partial R\cap D^2\), \(\partial R\cap D^3\), and \(\partial R\cap D^4\) \textbf{2-points}, \textbf{3-points}, and \textbf{4-points}, respectively, and we will collectively refer to these as \textbf{numbered points}.
Likewise, we will label the arcs of \(R\cap D^2\), \(R\cap D^3\), and \(R\cap D^4\) \textbf{2-arcs, 3-arcs}, and \textbf{4-arcs}, respectively, and we will collectively call them \textbf{numbered arcs}.
Observe that the endpoints of a \(j\)-arc are \(j\)-points.
The next result follows from these definitions.

\begin{lem}\label{lem:no_connection_off_of_track}
If \(\mathcal{R}\) is nonempty, at least one endpoint of each 2-arc must lie in a track, and all 3- and 4-points lie in tracks.
\end{lem}

\begin{proof}
The leftmost lane of \(\beta^2\) is not completely contained in \(\Lab\), but all the other lanes of \(\beta^2\), as well as all the lanes of \(\beta^3\) and \(\beta^4\), are contained in the union of the three colored disks and \(U_2^3\).
\(\partial R\) is disjoint from \(U_2^3\), so it can only intersect these lanes inside the colored disks.
Since every arc component of \(\partial R\) inside a colored disk is, by definition, a track, we conclude that aside from the leftmost lane of \(\beta^2\), \(\partial R\) only intersects the lanes of \(\beta^2\), \(\beta^3\) and \(\beta^4\) in tracks.
Thus all 3- and 4-points lie in tracks.
If both endpoints of a 2-arc lie outside of the tracks, then those endpoints must lie in the leftmost lane of \(\beta^2\), contradicting Lemma \ref{lem:samelane}.
Therefore at least one endpoint of each 2-arc must lie in a track.
\end{proof}

Corresponding to each track color is a particular sequence of numbered points.
Each of these three sequences is symmetric in the sense that the outermost numbered points in the track have the same number, the second-outermost numbered points match each other, the third-outermost numbered points match, etc.
Figure \ref{fig:tracks_with_colors_two_twist_case} depicts the sequence of numbered points along each colored track in the case depicted in Figure \ref{fig:boundary_of_Bprime_two_twists_revised_3}.

\begin{figure}[h!]
\centering
\labellist \small\hair 2pt
\pinlabel {\(\partial R\)} at -11 103
\pinlabel {brown track} at 90 89
\pinlabel {\(3\)} at 41 116
\pinlabel {\(2\)} at 60 116
\pinlabel {\(2\)} at 120.5 116
\pinlabel {\(3\)} at 140 116
\pinlabel {\(\partial R\)} at -11 56
\pinlabel {orange track} at 90 40
\pinlabel {\(3\)} at 41 69
\pinlabel {\(4\)} at 60 69
\pinlabel {\(4\)} at 120.5 69
\pinlabel {\(3\)} at 140 69
\pinlabel {\(\partial R\)} at -11 8
\pinlabel {purple track} at 90 -10
\pinlabel {\(4\)} at 41 21
\pinlabel {\(3\)} at 60 21
\pinlabel {\(3\)} at 120.5 21
\pinlabel {\(4\)} at 140 21
\endlabellist
\includegraphics[scale=.7]{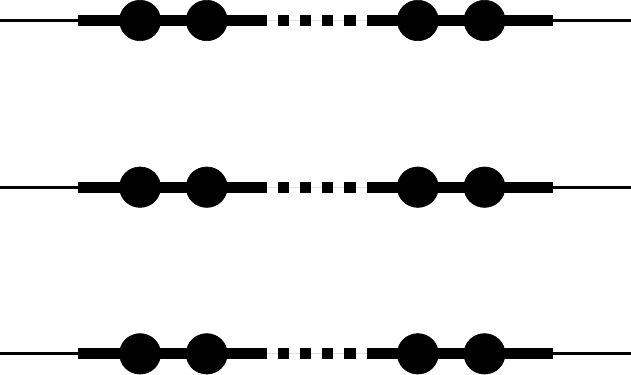}
\vspace{10pt}
\caption{Each colored track contains a sequence of numbered points.
No matter the twist numbers \(\left\{t_i^j\right\}\), we can say what the outermost and second-outermost points will be for each type of track.}
\label{fig:tracks_with_colors}
\end{figure}
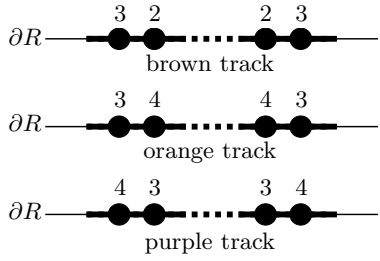

Observe Figure \ref{fig:alternating_green_points}, which depicts the possibilities for which directions tracks can go after they intersect the \(\beta\)-arcs.
We see that the outermost numbered points of a brown track are 3-points because the first and last \(\beta\)-arc that the track will intersect is \(\beta^3\).
After intersecting \(\beta^3\), a brown track will loop around and intersect \(\beta^2\), so the second-outermost numbered points of a brown track are 2-points.
Similarly, the outermost numbered points of an orange track are 3-points, and the second-outermost numbered points of an orange track are 4-points.
Finally, the outermost numbered points of a purple track are 4-points, and the second-outermost numbered points of a purple track are 3-points.
This information is summarized in Figure \ref{fig:tracks_with_colors}.
Note that this information about the outermost and second-outermost numbered points does not depend on the set \(\left\{t_i^j\right\}\) of twist numbers.

\begin{figure}[h!]
\centering
\labellist \small\hair 2pt
\pinlabel {\(\beta^2\)} at 80 134
\pinlabel {\(\beta^3\)} at 430 134
\pinlabel {\(\beta^4\)} at 705 134
\pinlabel {\(\beta^2\)} at 80 405
\pinlabel {\(\beta^3\)} at 430 405
\pinlabel {\(\beta^4\)} at 705 405
\pinlabel {\(G^\text{o}\)} at 480 244
\pinlabel {\(G^\text{o}\)} at 455 517
\pinlabel {\(G^\text{b}\)} at 325 6
\pinlabel {\(G^\text{b}\)} at 285 279
\pinlabel {\(G^\text{p}\)} at 675 6
\pinlabel {\(G^\text{p}\)} at 655 279
\endlabellist
\includegraphics[width=\textwidth]{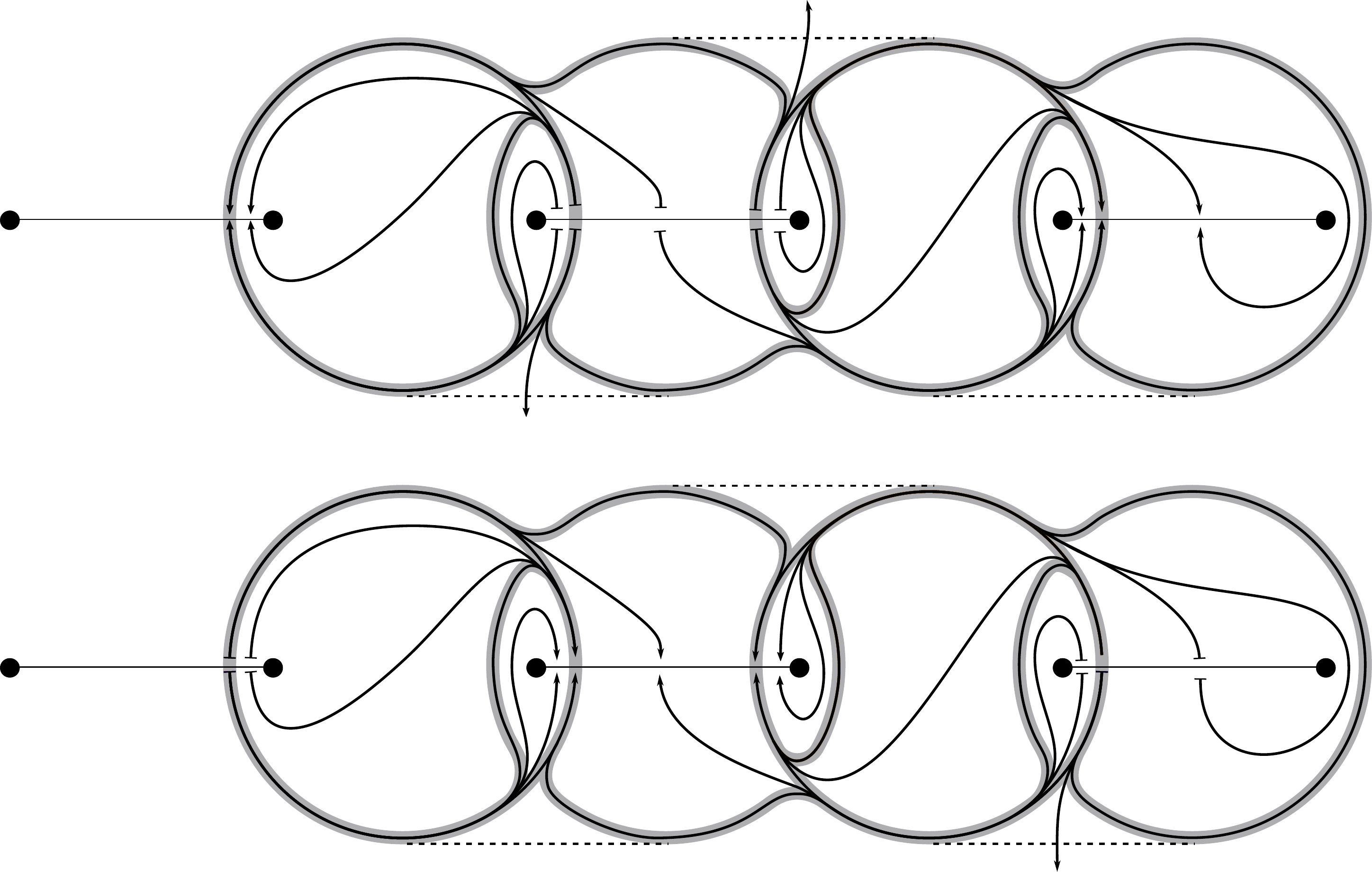}
\caption{If \(\lambda\) is a track, then \(\lambda\) contains a sequence numbered points.
The top picture shows that each 3-point on \(\lambda\) either has a 2-point or a 4-point to both sides, or it has a 2-point or a 4-point to one side and exits through a gate to the other side.
The bottom picture shows that each 2-point on \(\lambda\) has a 3-point to either side, and that each 4-point either has a 3-point to either side, or it has a 3-point to one side and exits through a gate to the other side.}
\label{fig:alternating_green_points}
\end{figure}
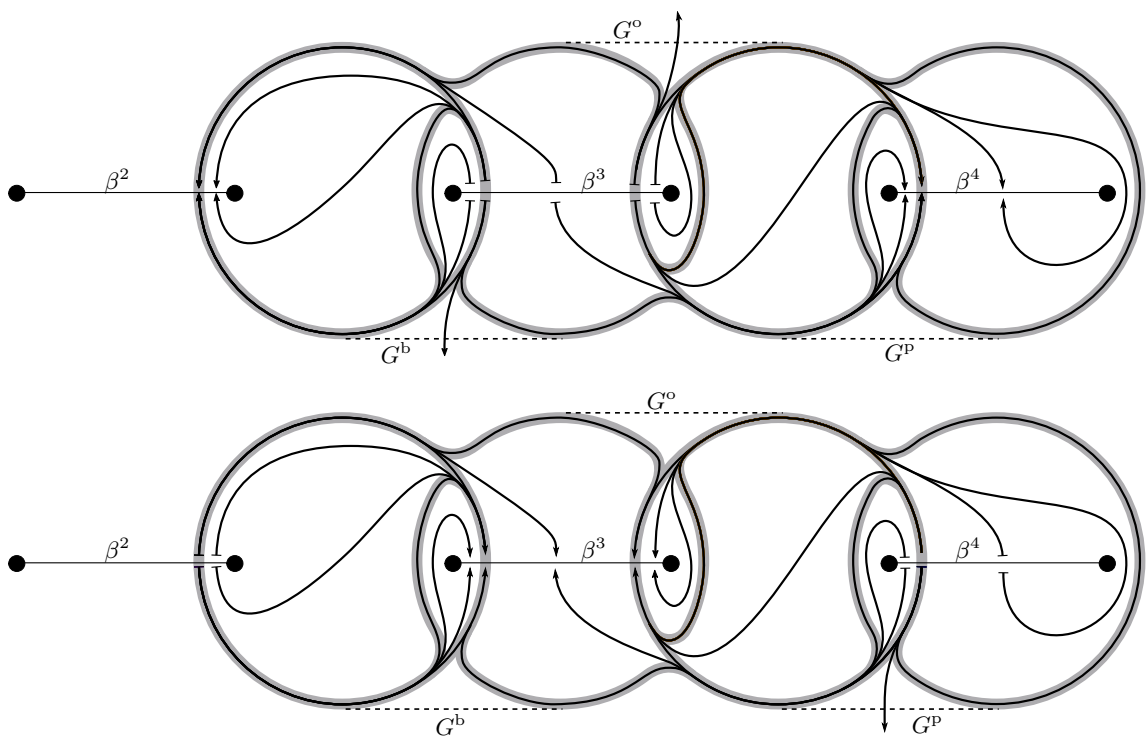

\begin{lem}\label{lem:no_connection_on_track}
Assume \(\mathcal{R}\) is nonempty.
No numbered arc can have endpoints which lie on the same track.
\end{lem}

\begin{proof}
Since \(\partial R\) must intersect at least one of the gates, it follows that there exists at least one track in \(\partial R\).
Refer again to Figure \ref{fig:alternating_green_points};
In the top picture, we see that after intersecting \(\beta^3\), a track will always either intersect \(\beta^2\) or \(\beta^4\) or leave the labyrinth.
In the bottom picture, we see that after intersecting \(\beta^2\) or \(\beta^4\), a track will always either intersect \(\beta^3\) or leave the labyrinth.
This means that in each sequence of numbered points in a track, every other numbered point is a 3-point.

Suppose two numbered points from the same track are connected in \(R\) by a numbered arc \(\lambda\).
Since the two endpoints of \(\lambda\) must have the same number, and since every other numbered point in the track is a 3-point, there must be an odd number of numbered points between the endpoints of \(\lambda\).
But this leads to a contradiction: numbered arcs never intersect each other, so it is impossible to pair up the numbered points between \(\lambda\)'s endpoints with disjoint numbered arcs.
\end{proof}

\begin{lem}\label{lem:threecolors}
If \(\mathcal{R}\) is nonempty, \(\partial R\) contains tracks of all three colors.
\end{lem}

\begin{proof}
Consider a numbered arc \(\lambda\) which is outermost in \(R\).
Being outermost guarantees that \(\lambda\)'s endpoints are adjacent same-numbered points on \(\partial R\).
Since numbered points in each track alternate between 3-points and 2- or 4-points, there are only three ways to have adjacent same-numbered points on \(\partial R\):
1) The endpoints of \(\lambda\) are consecutive 2-points which do not lie in tracks, which would contradict Lemma \ref{lem:no_connection_off_of_track}.
2) The endpoints of \(\lambda\) are a 2-point off of a track, and an outermost numbered point on a track, which is impossible since no outermost point on any track is a 2-point.
3) The endpoints of \(\lambda\) are outermost same-numbered points of adjacent tracks.

Suppose \(\partial R\) contains no brown tracks.
Since the outermost numbered points of orange tracks are 3-points, and the outermost numbered points of purple tracks are 4-points, \(\lambda\) must have endpoints which are adjacent outermost numbered points of two adjacent same-colored tracks, and we reach a contradiction by Lemma \ref{lem:samelane} since these two endpoints lie in the same lane.
Therefore \(\partial R\) must contain at least one brown track.
By a similar argument, \(\partial R\) contains at least one orange track.

Suppose \(\partial R\) contains no purple tracks.
Consider a numbered 4-arc \(\zeta\) of \(\partial R\cap D^4\) which is outermost in \(R\) (i.e., outermost among the set of 4-arcs).
\(\zeta\) cuts \(\partial R\) into two components.
Since \(\zeta\) is outermost among the 4-arcs, one of these components contains no 4-points.
Call this piece the \textbf{primary} piece. 
There are three possibilities;
we will derive a contradiction from each one.

\begin{description}
\item[Case 1] \(\zeta\) connects two 4-points in two brown tracks.
\item[Case 2] \(\zeta\) connects two 4-points in two orange tracks.
\item[Case 3] \(\zeta\) connects a 4-point of a brown track and a 4-point of an adjacent orange track.
\end{description}

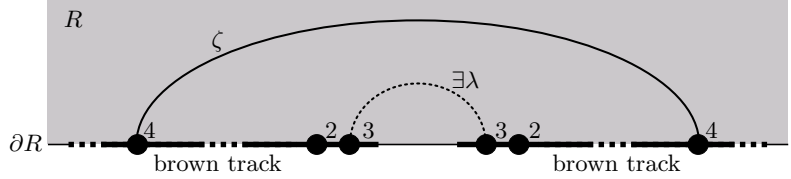
\begin{figure}[h!]
\centering
\labellist \small\hair 2pt
\pinlabel {\(R\)} at 15 80
\pinlabel {\(\zeta\)} at 100 67
\pinlabel {brown track} [B] at 100 -10
\pinlabel {brown track} [B] at 335 -10
\pinlabel {\(4\)} at 60 14
\pinlabel {\(2\)} at 167 14
\pinlabel {\(3\)} at 189 14
\pinlabel {\(3\)} at 267 14
\pinlabel {\(2\)} at 287 14
\pinlabel {\(4\)} at 391 14
\pinlabel {\(\partial R\)} at -13 7
\pinlabel {\(\exists\lambda\)} at 246 44
\endlabellist
\includegraphics[scale=.65]{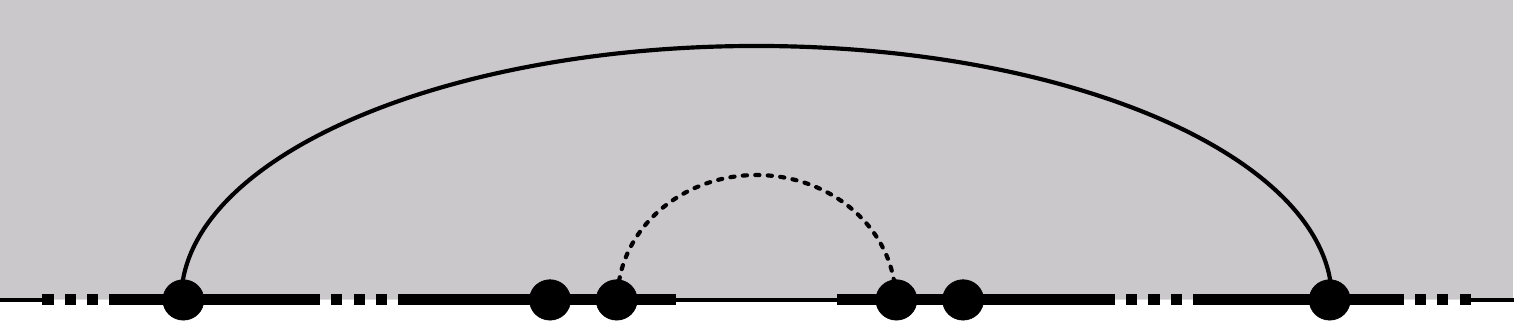}
\caption{Case 1: The outermost 4-arc \(\zeta\) connects 4-points in adjacent brown tracks.}
\label{fig:adjacent_blue_pts_in_brown_tracks}
\end{figure}

We start with Case 1.
Note that for this to be possible, the set \(\left\{t_i^j\right\}\) of twist numbers must be such that brown tracks contain 4-points, which is only true for some sets of twist numbers.
In contrast, orange and purple tracks necessarily contain 4-points, no matter the twist numbers.
This means that since the primary piece of \(\partial R\) cannot contain any 4-points, it cannot contain any complete track of any color.
Then the 4-points which are the endpoints of \(\zeta\) must lie in adjacent brown tracks.
See Figure \ref{fig:adjacent_blue_pts_in_brown_tracks}.
Consider all of the numbered points in the primary piece of \(\partial R\).
They must be paired with numbered arcs, and there must be an outermost such arc \(\lambda\).
\(\lambda\) connects same-numbered points adjacent on \(\partial R\).
If there are no 2-points between the brown tracks in the primary piece of \(\partial R\), then the only pair of adjacent same-numbered points in the primary part of \(\partial R\) are an outermost 3-point of one brown track, and an outermost 3-point of the other brown track.
But these points are the same lane of \(\beta^3\), which contradicts Lemma \ref{lem:samelane}.
If there is exactly one 2-point between the brown tracks in the primary piece of \(\partial R\), then there does not exist a pair of adjacent same-numbered points in the primary part of \(\partial R\) to be the endpoints of \(\lambda\).
If there is more than one 2-point between the brown tracks in the primary piece of \(\partial R\), then the only pair(s) of adjacent same-numbered points in the primary part of \(\partial R\) are pairs of these 2-points, but they cannot be the endpoints of \(\lambda\) by Lemma \ref{lem:no_connection_off_of_track}.
We conclude Case 1 is not possible.

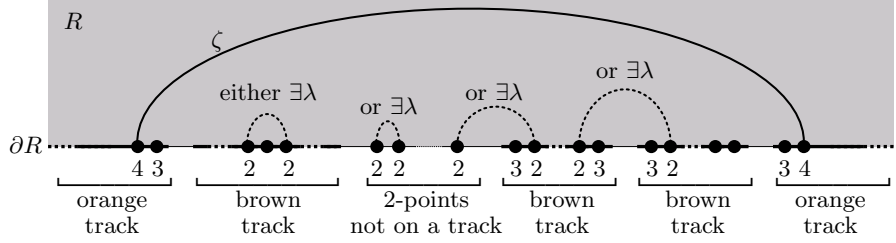
\begin{figure}[h!]
\centering
\labellist \small\hair 2pt
\pinlabel {\(R\)} at 15 80
\pinlabel {\(\zeta\)} at 100 68
\pinlabel {orange} [B] at 38 -30
\pinlabel {track} [B] at 38 -45
\pinlabel {\begin{rotate}{90}\(\left[\begin{tabular}{c}• \\ • \\ • \\ • \\ \end{tabular} \right.\)\end{rotate}} at 43 -22
\pinlabel {\(4\)} at 53 -7
\pinlabel {\(3\)} at 65 -7
\pinlabel {\(2\)} at 119 -7
\pinlabel {either \(\exists\lambda\)} at 130 38
\pinlabel {brown} [B] at 130 -30
\pinlabel {track} [B] at 130 -45
\pinlabel {\begin{rotate}{90}\(\left[\begin{tabular}{c}• \\ • \\ • \\ • \\ • \\ \end{tabular} \right.\)\end{rotate}} at 133 -22
\pinlabel {or \(\exists\lambda\)} at 201 30
\pinlabel {2-points} [B] at 223 -30
\pinlabel {not on a track} [B] at 223 -45
\pinlabel {\begin{rotate}{90}\(\left[\begin{tabular}{c}• \\ • \\ • \\ • \\ \end{tabular} \right.\)\end{rotate}} at 225 -22
\pinlabel {\(2\)} at 194 -7
\pinlabel {\(2\)} at 207 -7
\pinlabel {\(2\)} at 242 -7
\pinlabel {\(3\)} at 276 -7
\pinlabel {\(2\)} at 288 -7
\pinlabel {or \(\exists\lambda\)} at 263 38
\pinlabel {\(2\)} at 142 -7
\pinlabel {\(2\)} at 313 -7
\pinlabel {\(3\)} at 325 -7
\pinlabel {brown} [B] at 303 -30
\pinlabel {track} [B] at 303 -45
\pinlabel {\begin{rotate}{90}\(\left[\begin{tabular}{c}• \\ • \\ • \\ • \\ \end{tabular} \right.\)\end{rotate}} at 305 -22
\pinlabel {or \(\exists\lambda\)} at 340 51
\pinlabel {\(3\)} at 356 -7
\pinlabel {\(2\)} at 368 -7
\pinlabel {brown} [B] at 383 -30
\pinlabel {track} [B] at 383 -45
\pinlabel {\begin{rotate}{90}\(\left[\begin{tabular}{c}• \\ • \\ • \\ • \\ \end{tabular} \right.\)\end{rotate}} at 385 -22
\pinlabel {\(3\)} at 434 -7
\pinlabel {\(4\)} at 446 -7
\pinlabel {orange} [B] at 461 -30
\pinlabel {track} [B] at 461 -45
\pinlabel {\begin{rotate}{90}\(\left[\begin{tabular}{c}• \\ • \\ • \\ • \\ \end{tabular} \right.\)\end{rotate}} at 466 -22
\pinlabel {\(\partial R\)} at -13 7
\endlabellist
\includegraphics[scale=.65]{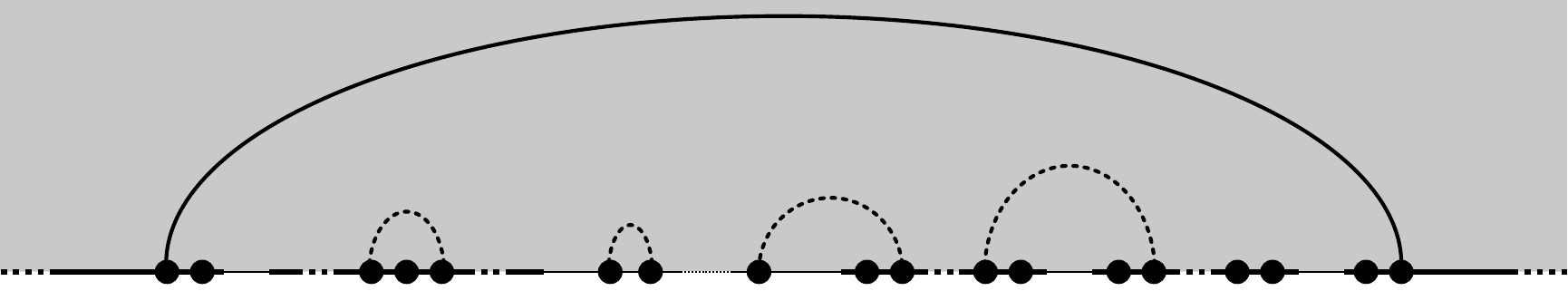}
\vspace{25pt}
\caption{Case 2: The outermost 4-arc \(\zeta\) connects 4-points in adjacent orange tracks. If brown tracks do not contain 4-points, then there may be any number of brown tracks between these two orange tracks (and there may be any number of 2-points between each pair of tracks).}
\label{fig:adjacent_blue_pts_in_orange_tracks}
\end{figure}

Next we consider Case 2, in which \(\zeta\) connects two 4-points in two orange tracks.
Brown tracks may or may not contain 4-points, depending on the twist numbers; if they do not, then the primary piece of \(\partial R\) may contain any number of brown tracks.
If there are zero brown tracks in the primary piece of \(\partial R\), then the proof is similar to Case 1.
Suppose the primary piece of \(\partial R\) contains at least one brown track.
Then since the second-outermost numbered points of brown tracks are 2-points, there are at least two 2-points in the primary piece of \(\partial R\).
Since numbered arcs cannot intersect each other, they cannot intersect \(\zeta\) in particular, so all of the 2-points in the primary piece must be paired with each other by a set \(S\) of 2-arcs.
Consider an outermost 2-arc \(\lambda\) in \(S\) (i.e., outermost with respect to the other 2-arcs of \(S\)).
\(\lambda\)'s endpoints are a pair of 2-points with no other 2-points between them.
There are four ways this can happen (depicted in Figure \ref{fig:adjacent_blue_pts_in_orange_tracks}), though each one leads to a contradiction:
1) The pair of 2-points lies in a single brown track, contradicting Lemma \ref{lem:no_connection_on_track}.
2) Both 2-points are off of tracks, contradicting Lemma \ref{lem:samelane}.
3) The pair consists of a second-outermost numbered point in a brown track and a 2-point not contained in a track.
But then \(\lambda\) would separate a single 3-point from all the other numbered points of \(\partial R\), so the numbered arc corresponding to that 3-point would have to intersect \(\lambda\).
4) The pair consists of second-outermost numbered points in adjacent brown tracks, points which lie in the same lane, again contradicting Lemma \ref{lem:samelane}.
Therefore Case 2 is impossible.


\begin{figure}[h!]
\centering
\labellist \small\hair 2pt
\pinlabel {\(R\)} at 15 80
\pinlabel {\(\zeta\)} at 100 75
\pinlabel {brown track} [B] at 100 -25
\pinlabel {orange track} [B] at 290 -25
\pinlabel {\(4\)} at 53 -7
\pinlabel {\(3\)} at 140 -7
\pinlabel {\(2\)} at 160 -7
\pinlabel {\(3\)} at 179 -7
\pinlabel {\(3\)} at 259 -7
\pinlabel {\(4\)} at 277 -7
\pinlabel {\(3\)} at 298 -7
\pinlabel {\(\partial R\)} at -13 7
\endlabellist
\includegraphics[scale=.65]{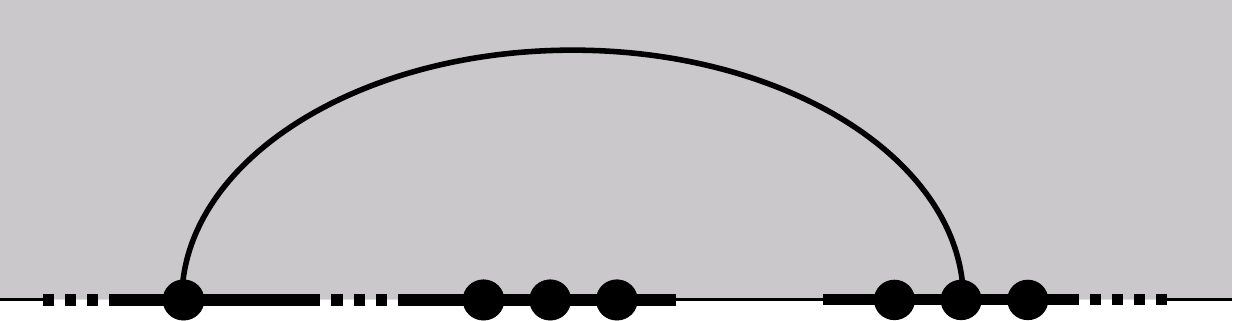}
\vspace{15pt}\caption{Adjacent brown and orange tracks with 4-points connected by the outermost 4-arc \(\zeta\)}
\label{fig:adjacent_blue_pts_in_orange_and_brown}
\end{figure}
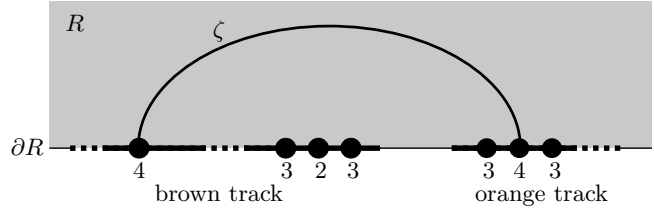

Case 3:  \(\zeta\) connects a 4-point of a brown track and a 4-point of an adjacent orange track, as in Figure \ref{fig:adjacent_blue_pts_in_orange_and_brown}.
Again, the numbered points in the primary piece of \(\partial R\) must all be connected in pairs.
Recall that in every track the numbered points alternate between 3-points and 2- or 4-points, so in the intersection of the primary piece and the brown track, there are at least two 3-points, one on either side of the second-outermost 2-point.
In the orange track, the second-outermost numbered point is a 4-point, so it must be an endpoint of \(\zeta\), so there is exactly one other numbered point in the intersection of the primary piece and the orange track: an outermost 3-point.
Since the primary piece of the brown track contains strictly more than one 3-point, and the primary piece of the orange track contains exactly one 3-point, it is impossible to pair up the 3-points points in such a way that no pair lies same track, so Case 3 contradicts Lemma \ref{lem:no_connection_on_track} and is therefore impossible.
We conclude \(\partial R\) must contain purple tracks, finishing the proof of Lemma \ref{lem:threecolors}.
\end{proof}

\begin{figure}[h!]
\centering
\labellist \small\hair 2pt
\pinlabel {\(\partial B'\)} at 255 32
\pinlabel {\(\beta'\)} at 220 10
\pinlabel {\(\partial B'\)} at  55 60
\pinlabel {\(\beta'\)} at  60 23
\pinlabel {\(\partial B'\)} at  145 12
\pinlabel {
\begin{rotate}{30}
colored point
\end{rotate}
} at 158 33
\pinlabel {route} at  265 212
\pinlabel {escape} at  265 227
\pinlabel {\(\tilde{b}\)} at 208 107
\pinlabel {\(\beta^i\)} at 55 107
\pinlabel {\(D^i\)} at 135 280
\pinlabel {\(\Lab\)} at 535 280
\pinlabel {\(\beta'\)} at 648 179
\pinlabel {\(\partial B'\)} at 606 141
\pinlabel {\(\tilde{b}\)} at 688 214
\endlabellist
\includegraphics[scale=.5]{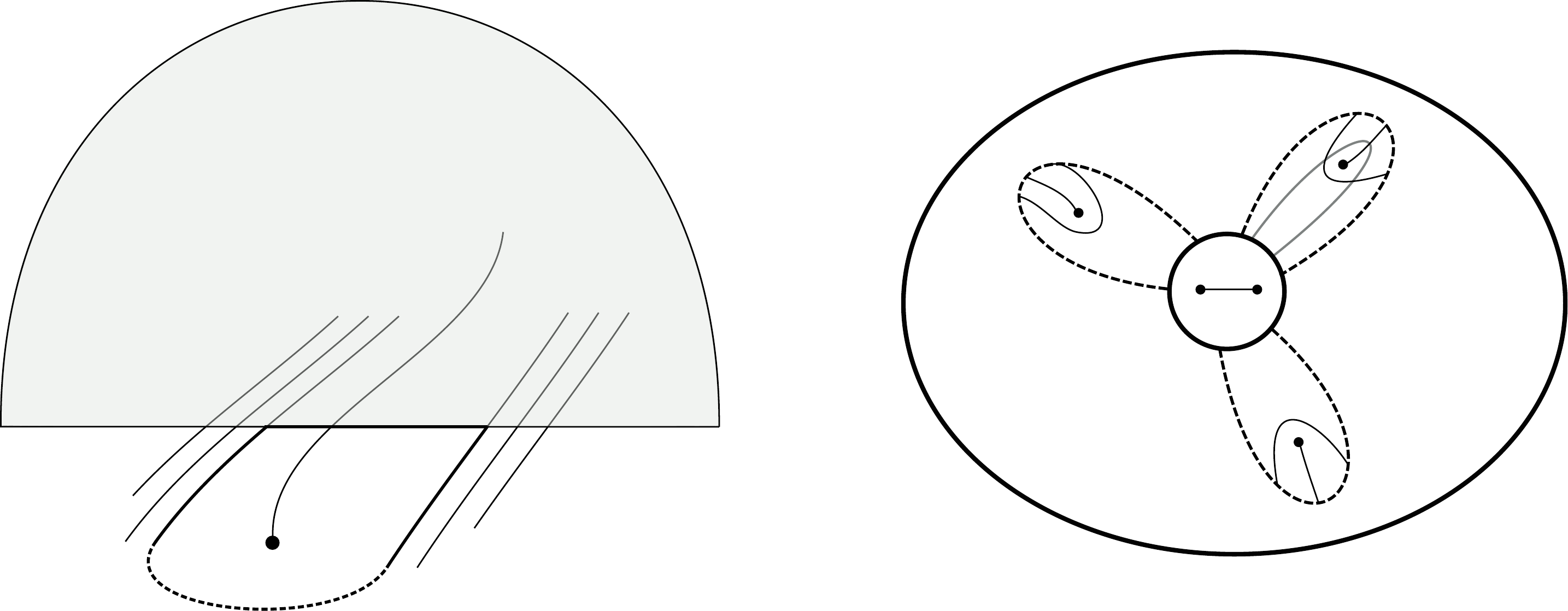}
\caption{On the left, \(\tilde{b}\) is the bold horizontal segment of \(\beta^i\).
On the right, \(\tilde{b}\) is the gray arc.
Also depicted on the right are the three escape routes, and the three corresponding tracks.}
\label{fig:escape_route_intersects_b}
\end{figure}
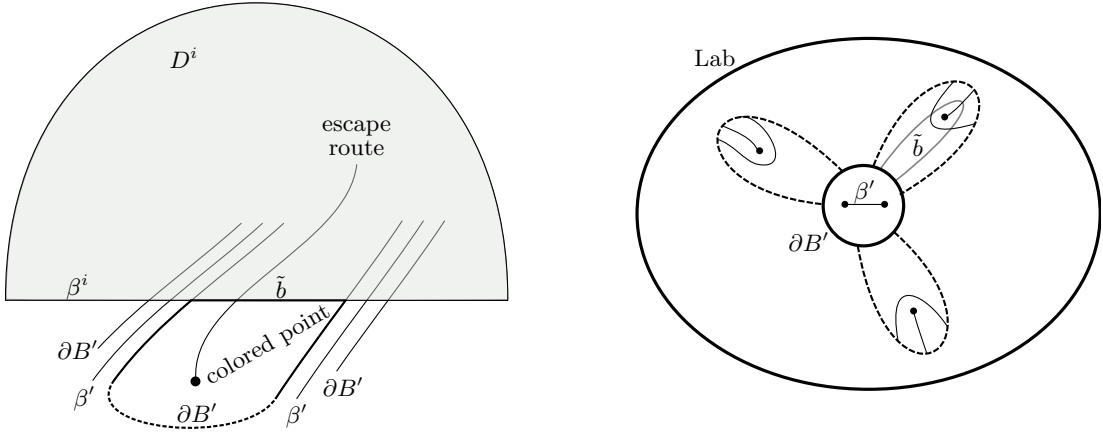

\begin{lem}\label{lanes_intersect_tracks}
Assume \(\mathcal{R}\) is nonempty.
Let \(b\subset\beta^i\) be a lane whose endpoints are interior points of \(\beta'\).
Then a track must intersect \(b\).
\end{lem}

\begin{proof}
Recall that \(\partial B'\) and the three gates cut \(\Lab\) into five components, one of which is the annulus \(A\), and observe that no lane of \(\beta^2,\beta^3\), or \(\beta^4\) intersects \(A\) except for the leftmost lane of \(\beta^2\).
By definition, \(b\) cannot be that leftmost lane, so \(b\) is properly contained in the union of \(U_i^j\) and the three colored disks.
Let \(\tilde{b}=b\backslash ({U_2^3}^\circ)\) as in the first picture of Figure \ref{fig:escape_route_intersects_b}.
\(\tilde{b}\) is an arc properly contained in a colored disk, with endpoints on \(\partial U_2^3\).
\(\tilde{b}\) cannot cobound a bigon with \(\partial B'\) because that would contradict Proposition \ref{prop:BprimeandDminpos}, so \(\tilde{b}\) must be an arc cutting out a punctured disk from a colored disk, as in the second picture of Figure \ref{fig:escape_route_intersects_b}.
Suppose the colored disk in question is brown.
By definition, the brown escape route does not intersect \(\partial B'\), so the only way for it to connect the brown point with the brown gate is to intersect \(\tilde{b}\).
Thus if \(\tilde{b}\) is contained in the brown disk, it must intersect the brown escape route.
Every brown track is a frontier in the brown disk of the brown escape route, so every brown track must also intersect \(\tilde{b}\). 
Since the color was arbitrary, the lemma is proved. 
\end{proof}

\begin{prop}\label{prop:R_empty_or_L_perturbed}
If \(L\) is not perturbed at \({\alpha^4}'\), then \(\mathcal{R}\) is empty.
\end{prop}

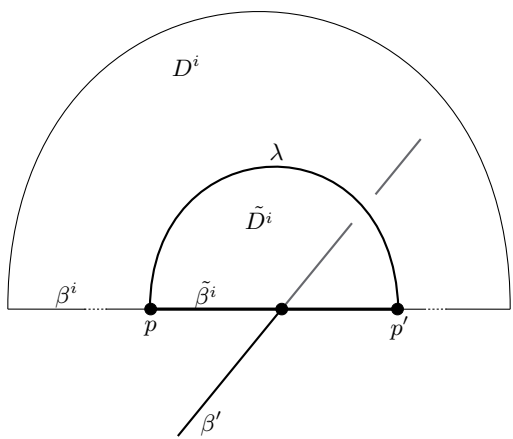
\begin{figure}[h!]
\centering
\labellist \small\hair 2pt
\pinlabel {\(\beta'\)} at 154 10
\pinlabel {\(\lambda\)} at 203 215
\pinlabel {\(\tilde{D^i}\)} at 190 165
\pinlabel {\(\tilde{\beta^i}\)} at 150 107
\pinlabel {\(p\)} at 108 82 
\pinlabel {\(p'\)} at 296 82
\pinlabel {\(D^i\)} at 135 280
\pinlabel {\(\beta^i\)} at 45 107
\endlabellist
\includegraphics[scale=.5]{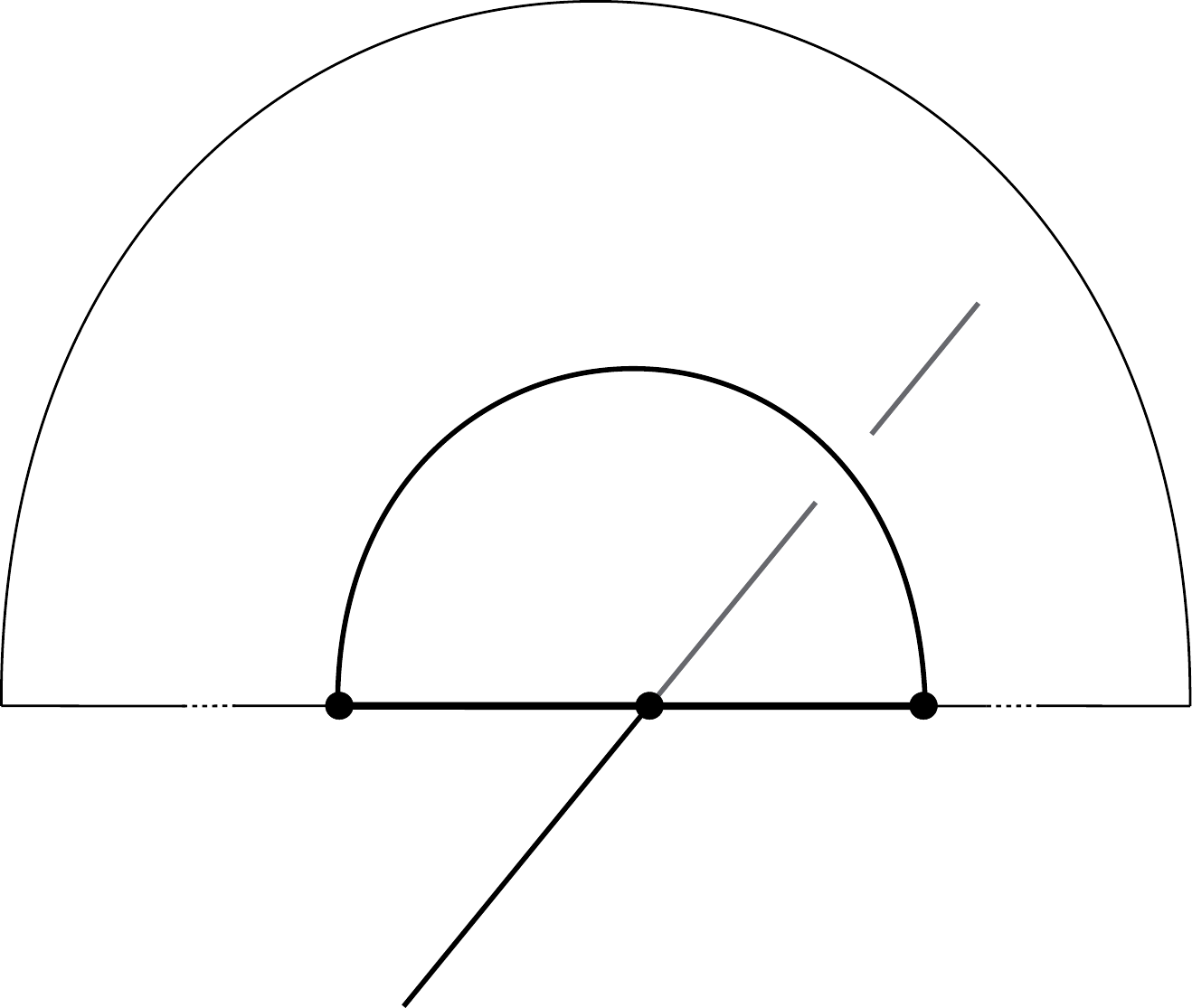}
\caption{\(\lambda\) straddles exactly one strand of \(\beta'\).}
\label{fig:one_beta_prime_bt_lambda}
\end{figure}

\begin{proof}
Assume \(\mathcal{R}\) is nonempty.
Consider a numbered arc \(\lambda\) which for some \(i\) is outermost in \(D^i\).
Let \(p\) and \(p'\) be the endpoints of \(\lambda\), and let \(\tilde{\beta^i}\) be the segment of \(\beta^i\) between \(p\) and \(p'\), as in Figure \ref{fig:one_beta_prime_bt_lambda}.
\(p\) and \(p'\) cannot lie in a single lane by Lemma \ref{lem:samelane}.
If the two lanes containing \(p\) and \(p'\) have at least one lane between them (i.e., if the lanes are not adjacent), then by Lemma \ref{lanes_intersect_tracks}, there must be at least one numbered point between \(p\) and \(p'\).
However, \(\lambda\) is outermost in \(D^i\), so that cannot happen.
Therefore the lanes containing \(p\) and \(p'\) must be adjacent, so \(\tilde{\beta^i}\) contains exactly one point of \(\beta'\). 

\(\lambda\) cuts off a small disk \(\tilde{D^i}\) from \(D^i\).
We boundary-compress \(R\) along \(\tilde{D^i}\), which results in two new disks, \(R^1\) and \(R^2\), properly embedded in \(M_+\).
Since each intersects \(\beta'\) exactly once, \(\partial R^1\) and \(\partial R^2\) cut \(F\) into an annulus and two disks, and each of the two disks contains an endpoint of \(\beta'\).
Label these three regions \(F^{R^1}\), \(F^{R^2}\), and \(F^A\).
\(R^i\) cannot be trivial since \(F\) is punctured to either side of \(\partial R^i\).
Therefore both \(R^i\) are compressing disks.
Notice that \(R^1\) and \(R^2\) cannot be parallel because the band sum dual to the boundary compression just performed would recover \(R\), but any band sum of parallel disks is a trivial disk.

Define a pair of marked points on \(F\) to be \textbf{partners} if they are endpoints of a common bridge arc above \(F\).
Let the endpoints of \(\beta'\) be \(q^1\) and \(q^2\).

\begin{figure}[h!]
\centering
\labellist \small\hair 2pt
\pinlabel {\(\beta'\)} at 296 117
\pinlabel {\(\hat{D}\)} at 415 161
\pinlabel {\(\alpha^j\)} at 444 174
\pinlabel {\(\partial R^2\)} at 525 188
\pinlabel {\(\partial R^1\)} at 240 128
\pinlabel {\(F^{R^1}\)} at 105 30
\pinlabel {\(F^{R^2}\)} at 372 100
\pinlabel {\(\hat{d}\)} at 440 125
\pinlabel {\(F^A\)} at 240 178
\endlabellist
\includegraphics[scale=.5]{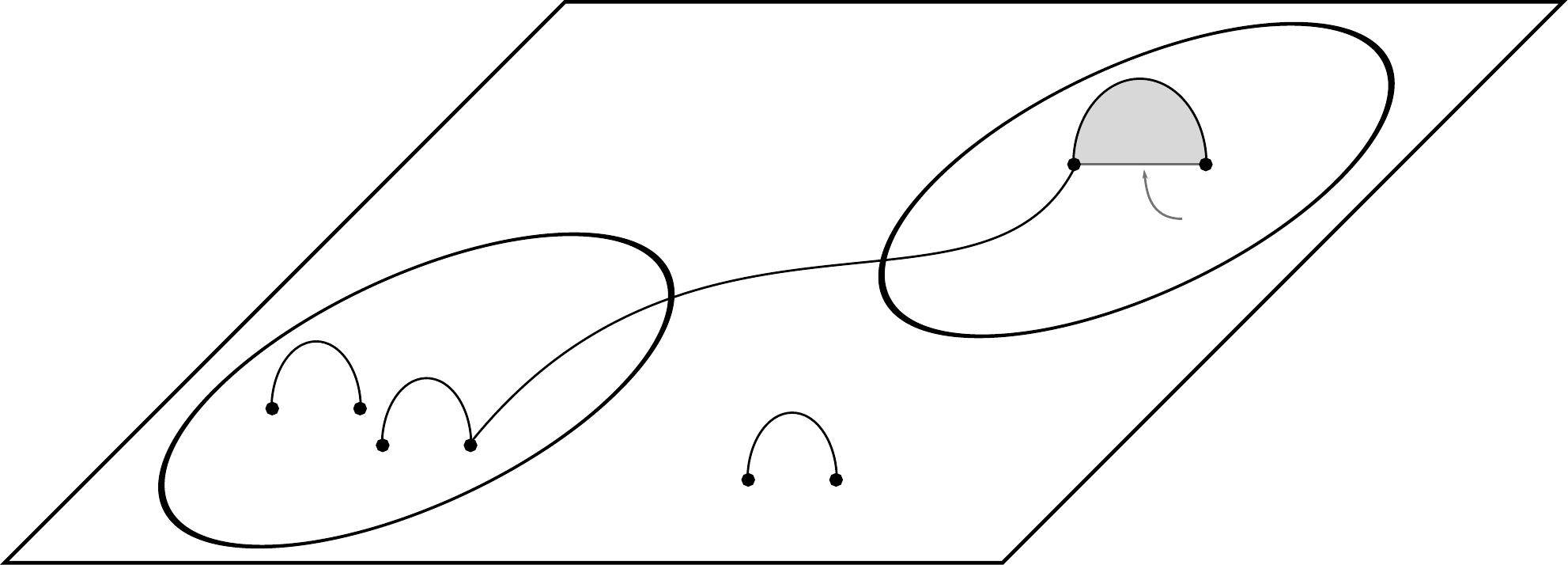}
\caption{In Case 1, the bridge arcs are all disjoint from \(R^1\) and \(R^2\).
One of the disks \(F^{R^1}\) or \(F^{R^2}\) must contain exactly one pair of partners, the endpoints of a bridge arc \(\alpha^i\).
Then \(L\) is perturbed at \({\alpha^4}'\).}
\label{fig:critical_or_perturbed}
\end{figure}
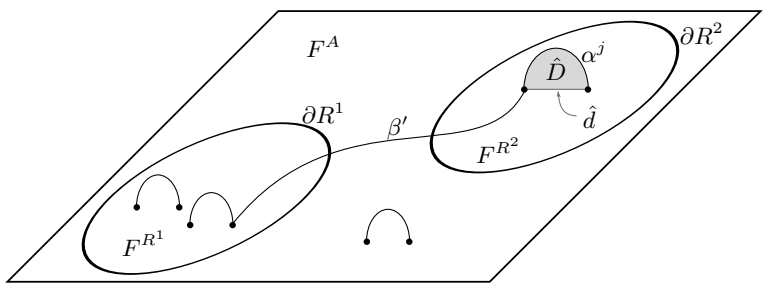

\textbf{Case 1:} For all four pairs of partners, it is the case that both points are contained in the same region: \(F^{R^1}\), \(F^{R^2}\), or \(F^A\).
Since an endpoint of \(\beta'\) lies in each of the disks, a pair of partners lies in each of the disks.
One pair of partners must be in \(F^A\);
otherwise, \(R^1\) and \(R^2\) would be parallel compressing disks.
This accounts for three of the four pairs of marked points.
The fourth pair may be in any of the three regions \(F^{R^1}\), \(F^{R^2}\), or \(F^A\).
This means that \(F^{R^1}\) or \(F^{R^2}\) must contain exactly one pair of partners.
See Figure \ref{fig:critical_or_perturbed}.
Suppose these partners are the endpoints of \(\alpha^j\) for some \(j\).
Then there is some bridge disk \(\hat{D}\) for \(\alpha^j\) (which is not necessarily isotopic to \(D^j\)) which is disjoint from \(R^i\).
Let \(\hat{d}=F\cap\hat{D}\).
\(\hat{d}\) and \(\beta'\) can be made disjoint except for their shared endpoint.
Then \(\hat{d}\cup\beta'\) is an embedded arc.
This means that \(\hat{D}\) and \({D^4}'\) are bridge disks above and below \(F\) which intersect in a single point of the link, so \(L\) is perturbed at \({\alpha^4}'\).
This contradicts the hypothesis of the proposition, so Case 1 is not possible.

\textbf{Case 2:} There is a pair of partners  which lie in different regions \(F^{R^1}\), \(F^{R^2}\), and \(F^A\).
But then the bridge arc connecting them must intersect \(R^1\) or \(R^2\), a contradiction since compressing disks are disjoint from the link.
Having arrived at a contradiction, we conclude \(R\) cannot exist, so \(\mathcal{R}\) is empty, and the proposition is proved.

\end{proof}

Define \(\mathcal{R'}\) to be the set of red compressing disks below \(F\) disjoint from \(B\).
The next result follows directly from Proposition \ref{prop:R_empty_or_L_perturbed} and rotational symmetry of \(L\).

\begin{cor}\label{cor:180symmetry}
If \(L\) is not perturbed at \(\alpha^1\), then \(\mathcal{R'}\) is empty.
\end{cor}


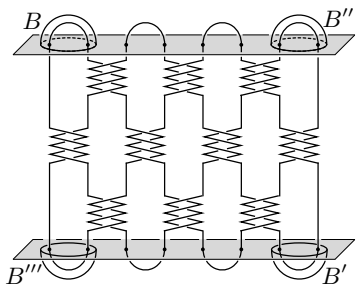
\begin{figure}[h!]
\centering
\labellist \small\hair 2pt
\pinlabel {\(B'''\)} at 15 7
\pinlabel {\(B''\)} at 410 330
\pinlabel {\(B'\)} at 405 7
\pinlabel {\(B\)} at 25 325
\endlabellist
\includegraphics[scale=.3]{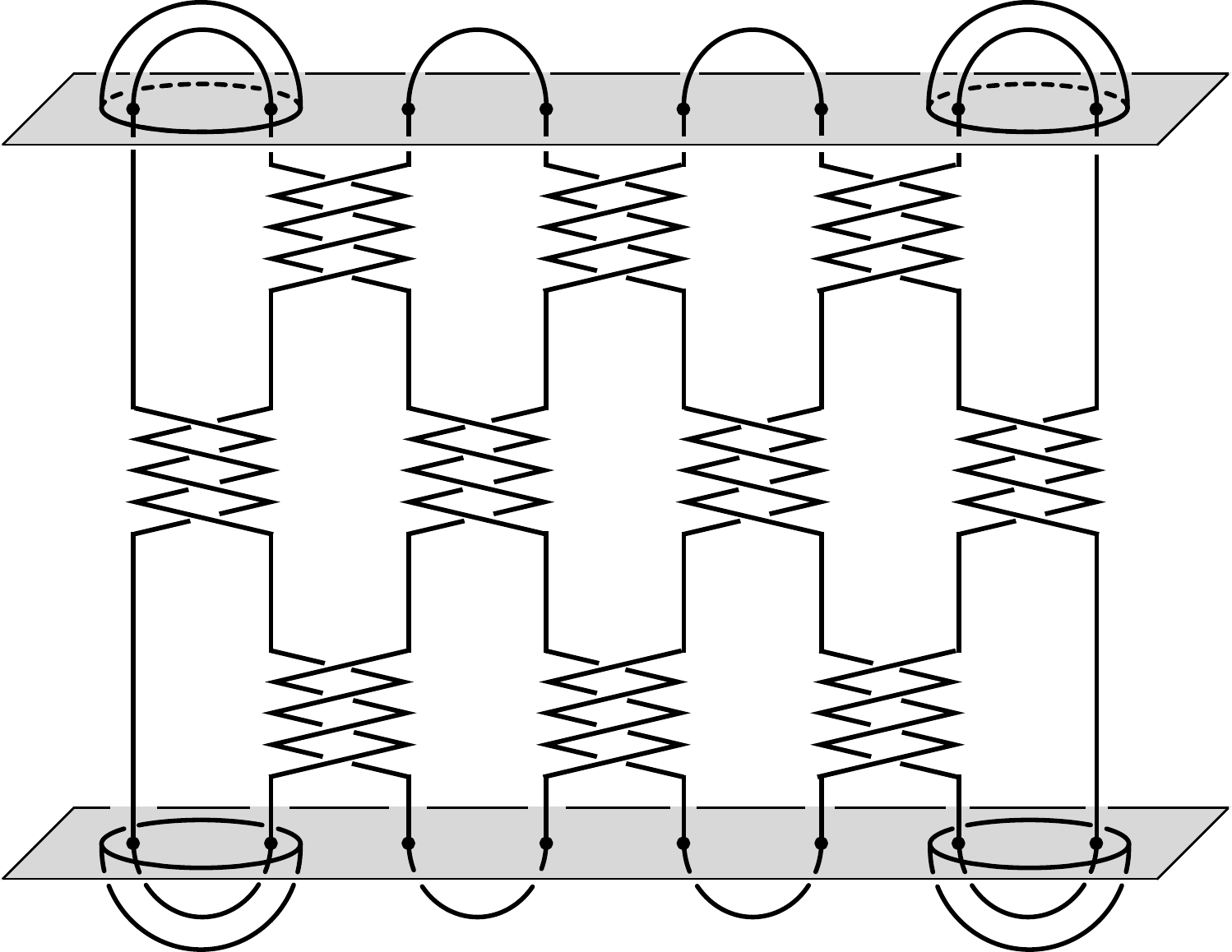}
\caption{A disjoint pair of blue disks (\(B\) and \(B'\)) and a disjoint pair of red disks (\(B''\) and \(B'''\)) which fulfill Condition 2 of the definition of a critical surface}
\label{fig:corner_disks}
\end{figure}

\begin{thm}\label{thm:main_thm_specific}
Let \(L\) be a 2-twisted link in \((4,4)\)-plat position with bridge sphere \(F\simeq F_4=f^{-1}(4)\subset S^3\).
Suppose \(L\) has twist numbers \(\left\{t_i^j\right\}\) such that for all \(j\), the twist numbers \(t_2^j,t_4^j\) are positive, and \(t_3^j\) are negative.
If \(L\) is perturbed at neither \(\alpha^1\) nor \({\alpha^4}'\), then \(F\) is a critical bridge sphere.
\end{thm}

\begin{proof}
By Proposition \ref{prop:R_empty_or_L_perturbed} and Corollary \ref{cor:180symmetry} we conclude that \(\mathcal{R}\) and \(\mathcal{R'}\) are empty.
In other words, all red disks above \(F\) intersect all blue disks below \(F\), and vice versa, which fulfills Condition 1 of the definition of a critical surface.
Observe that the link diagram representing \(\partial B'\) (in Section \ref{sec:labyrinth}) is disjoint from \(\partial B\), so \(B\) and \(B'\) are disjoint.
This gives us a pair of disjoint blue disks, one above \(F\) and one below.
Consider \(B''\) and \(B'''\) depicted in Figure \ref{fig:corner_disks}. 
(Explicitly, these are the frontiers of regular neighborhoods of \(D^4\) and \({D^1}'\) above and below \(F\), respectively.)
By a symmetric argument, these two compressing disks are a pair of disjoint red disks, one above \(F\) and one below.
These two pairs of disjoint disks fulfill the Condition 2 of the definition of a critical surface. 
\end{proof}

\begin{figure}[h!]
\centering
\labellist \small\hair 2pt
%
\pinlabel {even} at 183.5 197
\pinlabel {all \(t_i^j\geq 2\)} [l] at 500 197
\pinlabel {all \(t_i^j\leq -2\)} [l] at 500 133.5
\pinlabel {even} at 183.5 66
\pinlabel {all \(t_i^j\geq 2\)} [l] at 500 66
%
\endlabellist
\includegraphics[width=.4\textwidth]{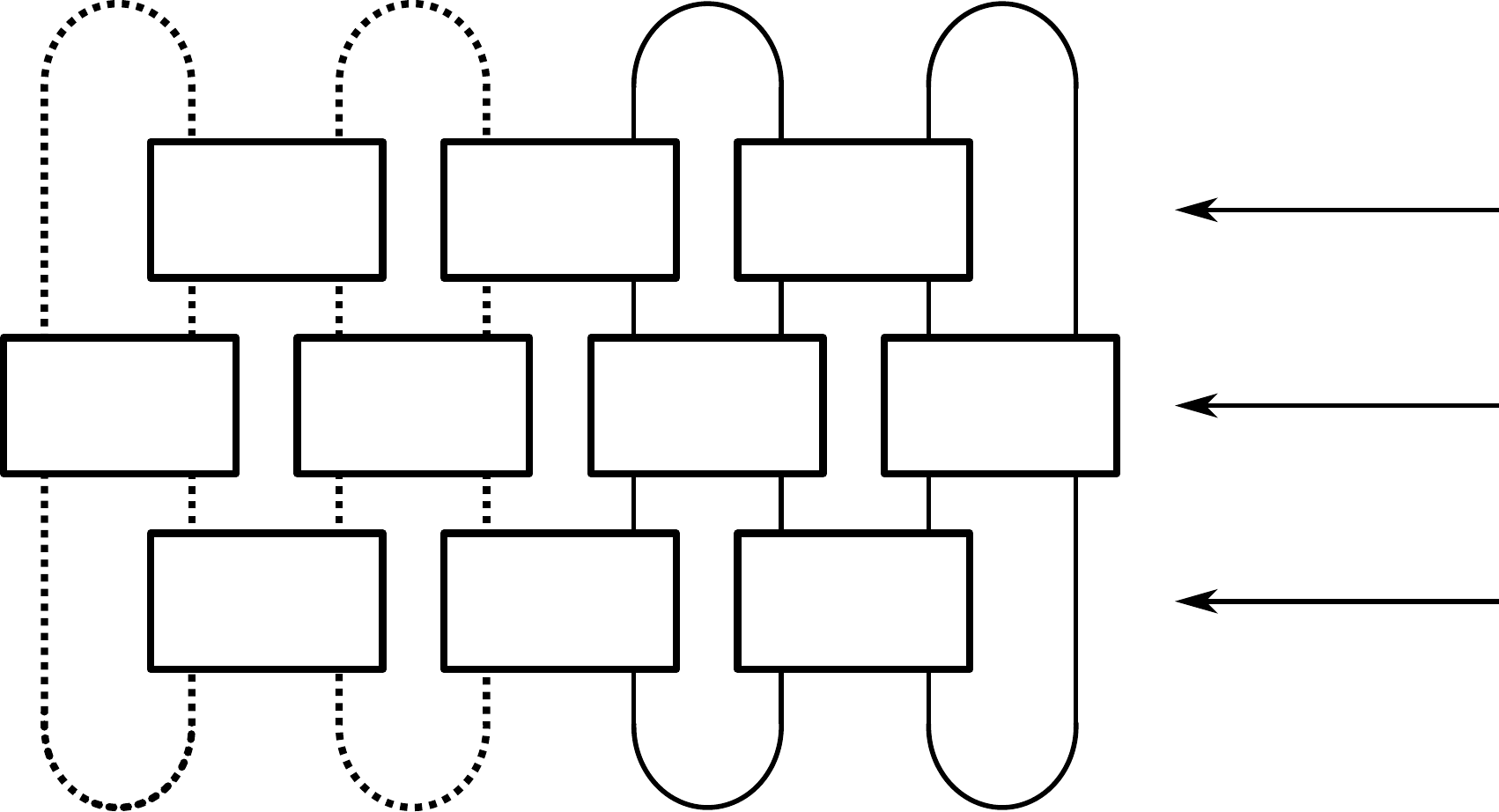}
\caption{This figure represents links \(L\) in the family \(\mathcal{L}\). Since both \(t_2^2\) and \(t_4^2\) are even, \(L\) will have at least two components, \(L^1\) (drawn with dotted arcs) and \(L^2\) (drawn with solid arcs).}
\label{fig:final_family}
\end{figure}
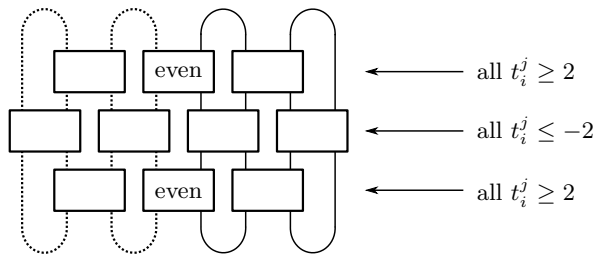

We finally come to the proof of Theorem \ref{thm:mainthm}.

\begin{thm}\label{thm:mainthm}
There is an infinite family of nontrivial links with critical bridge spheres.
\end{thm}

\begin{proof}

All we must do to prove Theorem \ref{thm:mainthm} is to demonstrate that there exists an infinite family of links which satisfy the hypothesis of Theorem \ref{thm:main_thm_specific}.
Let \(\mathcal{L}\) be the set of 2-twisted links in \((4,4)\)-plat position with twist numbers \(\left\{t_i^j\right\}\), such that all twist numbers in the top and bottom rows are positive, and all twist numbers in the middle row are negative, and such that \(t_2^2\) and \(t_4^2\) are both even.
The parity requirement guarantees that if \(L\in\mathcal{L}\), then \(L\) is a union of two different links: \(L^1\) and \(L^2\).
See Figure \ref{fig:final_family}.
Consider \(L^1\);
it is a link in a 2-bridge position (with respect to \(F\)), so it can have at most two components.

\textbf{Case 1:} \(L^1\) has one component, i.e., it is a knot.
Let \(D(L^1)\) be the diagram for \(L^1\) obtained by projection to the \((x,z)\)-plane.
\(D(L^1)\) is alternating (for the same reason \(D(L)\) was shown to be alternating in Section \ref{sec:setting}).
Observe that \(D(L^1)\) is a reduced diagram.
(In other words, as a graph, it has no cut vertex.)
One of the famous Tait Conjectures states that a reduced alternating diagram for a knot realizes the knot's minimal crossing number \cite{Lickorish}.
(This was proved in 1987 by Kauffman \cite{Kauffman} and Murasugi \cite{Murasugi}.)
Therefore as \(D(L^1)\) has at least 8 crossings, \(L^1\) is not the unknot, so it must be a 2-bridge knot.
Since \(L^1\) is in bridge position, it is by definition not perturbed at any of its bridge arcs, one of which is \(\alpha^1\).
We conclude that \(L^1\) is not perturbed at \(\alpha^1\). 
Since \(L^1\) and \(L^2\) are not connected to each other, \(\alpha^1\) cannot share an endpoint with either of the two bridge arcs below \(F\) contained in \(L^2\).
Therefore we conclude that \(L\) is not perturbed at \(\alpha^1\).

\textbf{Case 2:} \(L^1\) has two components.
Then each component is individually an unknot in bridge position.
Unknots are not perturbed, so in particular, \(L^1\) is not perturbed at \(\alpha^1\).
From there, it follows as in Case 1 that \(L\) is not perturbed at \(\alpha^1\) either.

By a similar argument, \(L\) is not perturbed at \({\alpha^4}'\) either.
Therefore \(L\) satisfies Theorem \ref{thm:main_thm_specific} and its bridge sphere \(F\) is critical.
This proves \(\mathcal{L}\) is an infinite family of links satisfying Theorem \ref{thm:mainthm}.
\end{proof}

\end{document}